\documentclass[reqno]{amsart}
\usepackage{graphicx} % Required for inserting images
\usepackage{subfiles}
\usepackage{enumitem}
\usepackage[T1]{fontenc}
\usepackage{euscript}
\usepackage{amssymb}
\usepackage{dsfont}
\usepackage{cleveref}
\newtheorem{theorem}{Theorem}[section]
\newcommand{\ds}[1]{\displaystyle #1}
\newtheorem{lemma}{Lemma}[section]
\newtheorem{proposition}{Proposition}[section]
\newtheorem{notation}{Notation}[section]
\theoremstyle{definition}
\newtheorem{definition}{Definition}[section]
\theoremstyle{remark}
\newtheorem{remark}{Remark}[section]

\usepackage{physics}
\newcommand{\R}{\mathbb{R}}
\newcommand{\eps}{\epsilon}
\newcommand{\N}{\mathbb{N}}

\newcommand{\T}{\mathbb{T}}
\newcommand{\di}{\displaystyle}

\usepackage{xcolor}
\usepackage{bbm}

\makeatletter
\def\Let@{\def\\{\notag\math@cr}}
\makeatother
\usepackage[margin=1in]{geometry}
\allowdisplaybreaks
\title[Inviscid Limit for 2D Navier-Stokes]{Uniqueness and Stability of planar shock wave to 2D compressible Euler equations in the vanishing viscosity limit of compressible Navier-Stokes equations}

\author{Geng Chen}
\address[Geng Chen]{\newline Department of Mathematics, \newline The University of Kansas, Lawrence, KS 66049, USA.}
\email{gengchen@ku.edu}

\author{Jeffrey Cheng}
\address[Jeffrey Cheng]{\newline Department of Mathematics, \newline The University of Texas at Austin, Austin, TX 78712, USA.}
\email{jeffrey.cheng@utexas.edu}

\author{Alexis F. Vasseur}
\address[Alexis F. Vasseur]{\newline Department of Mathematics, \newline The University of Texas at Austin, Austin, TX 78712, USA.}
\email{vasseur@math.utexas.edu}

\author{Yi Wang}
\address[Yi Wang]{\newline Institute of Applied Mathematics, AMSS, CAS, Beijing 100190, P. R. China \newline and School of Mathematical Sciences, University of Chinese Academy of Sciences,
Beijing 100049, P. R. China}
\email{wangyi@amss.ac.cn}

\date{\today}
\thanks{\textbf{Acknowledgement}: G. Chen was partially supported by NSF grants DMS 2306258 and 2625028. A. Vasseur was partially supported by NSF grants DMS 2306852 and 2219434. Y. Wang was partially supported by NSFC Grants No. 12421001 and 12288201 and CAS Project for Young Scientists in Basic Research, Grant No. YSBR-031.}

\begin{document}
\begin{abstract}    
%Forthe compressible Euler equationsis in one space dimension  including shock waves, one can select the unique solutionto by adding the entropy condition. However,
It is well known that the compressible Euler equations with  planar shock initial data may admit infinitely many entropy weak solutions, as shown by Chiodaroli, De Lellis and Kreml [Comm. Pure Appl. Math. (2015)]. In this paper, we introduce a suitable notion of global-in-time weak solutions to the two-dimensional (2D) compressible Navier-Stokes equations with density-dependent viscosities in $\R \times \T$. Within this solution framework, we establish the first vanishing viscosity limit selection criterion for the physically relevant planar shock solution to the 2D compressible Euler equations. More precisely, we prove the uniqueness and stability of  the 2D Euler planar shock solution within a class of inviscid limits of Navier-Stokes solutions emanating from near-planar initial data while the shock direction perturbation can be arbitrarily large. Starting from the relative $\kappa$-entropy for 2D compressible Navier-Stokes equations in $\T^2$ by Bresch, Noble and Vila [ESAIM Proc. Surveys (2017)], we first establish the relative $\frac12$-entropy between 2D weak solution of compressible Navier-Stokes equations in $\R \times \T$ and one-dimensional (1D)  strong solution for averaged initial data in $\R$. Then we achieve the desired vanishing viscosity limit selection criterion by successfully bridging the fundamental structural gap between 2D Eulerian formulation of the original Navier-Stokes system in the relative $\frac12$-entropy and 1D vanishing viscosity limit to a shock solution of Euler equations in Lagrangian coordinates by Kang and Vasseur [Invent. Math.  (2021)]. 
The framework developed herein can be applied to the three-dimensional limit selection criterion in $\R\times \T^2$, as well as to more complex planar Riemann solutions and interacting wave patterns.

\end{abstract}
\maketitle

\section{Introduction}
\renewcommand{\theequation}{\arabic{section}.\arabic{equation}}
\setcounter{equation}{0}
Consider the following 2D  barotropic compressible Navier-Stokes equations with density-dependent viscosities:
\begin{equation}\label{eq:2DNSintro}
\left\{
\begin{array}{ll}
    \rho_t +{\rm div}(\rho u)=0, \qquad\qquad\  \indent (t,x)=(t,x_1,x_2) \in \R^+ \times \R \times \T, \\[3mm]
     (\rho u)_t+{\rm div}(\rho u\otimes u)+\nabla p(\rho)={\rm div} (2 \nu\mu(\rho) \mathbb{D}u)+\nabla (\nu\lambda(\rho){\rm div} u),
\end{array}
\right.
\end{equation}
where $\rho=\rho(t,x)$ and $u=u(t,x)=(u_1, u_2)(t,x)$ denote the fluid density and velocity field, respectively. The deformation rate tensor is defined as $\mathbb{D}u:=\frac{\nabla u+\nabla u^T}{2}$ where $\nabla u^T$ denotes the transpose of the velocity gradient matrix $\nabla u$.  
The pressure is given by the barotropic equation of state $p(\rho)$, and the density-dependent shear and bulk viscosities take the form $2\nu \mu(\rho)$ and $\nu \lambda(\rho)$, with $\nu>0$ being a constant kinetic viscosity parameter. These viscosity functions satisfy the standard physical constraints
$$
 \mu(\rho)\geq 0,\quad  \mu(\rho)+\lambda(\rho)\geq 0,
$$
along with the Bresch–Desjardins relation
 \begin{equation}\label{BD}
 \lambda(\rho)=2(\mu^\prime(\rho)\rho-\mu(\rho)).
 \end{equation}
We adopt the polytropic pressure law $p(\rho)=b\rho^\gamma$, for $b>0$ and $\gamma>1$ being fluid constants. Without loss of generality, we normalize the coefficient $b=1$ in the sequel. 
The spatial domain under consideration is $\R\times\T$ with $\T:=\R/\mathbb{Z}$ denoting the one-dimensional unit flat torus. 
The initial density and momentum are prescribed as
$$
 (\rho(0,x), \rho u(0,x))=(\rho^0(x),m^0(x)),
$$
where $\rho^0(x)\geq 0$ on $\R\times \T$, and the initial momentum $m^0(x)$ satisfies the physical compatibility condition $m^0(x)=0$ almost everywhere on the vacuum set $\{x|\rho^0(x)=0\}.$

In the present work, we restrict our analysis to the specific regime defined by
\begin{equation}\label{sw}
\mu(\rho)=\rho,\quad \lambda(\rho)\equiv0, \quad \gamma\in (1,2],
\end{equation}
which gives the system
\begin{equation}\label{eq:2DNS}
\left\{
\begin{array}{ll}
    \rho_t +{\rm div}(\rho u)=0, \qquad  \indent (t,x)=(t,x_1,x_2) \in \R^+ \times \R \times \T, \\[3mm]
     (\rho u)_t+{\rm div}(\rho u\otimes u)+\nabla p(\rho)={\rm div} (2 \nu\rho \mathbb{D}u).
\end{array}
\right.
\end{equation}
However, the framework developed here naturally extends to the broader class of viscosities satisfying the Bresch–Desjardins relation \eqref{BD}. Notably, the classical shallow water system, corresponding to the case $\gamma=2$ in \eqref{sw}, is fully covered within the scope of our results.

\par For the compressible Navier-Stokes models with constant viscosities, that is, when both viscosities $\mu(\rho)$ and $\lambda(\rho)$ are constants, the global existence of multi-dimensional weak solutions can be traced back to the classical works by Lions \cite{MR1637634}, Jiang and Zhang \cite{MR1810944}, and Feireisl, Novotný and Petzeltová \cite{MR1867887}, while the well-posedness of small smooth solutions near a constant equilibrium state was established by Matsumura and Nishida \cite{MR0584357‌}.
However, from the Chapman–Enskog expansion of the Boltzmann equation, viscosity is intrinsically a temperature-dependent transport coefficient, which in turn reduces to a density-dependent function in the barotropic regime \cite{MR1148892}.  
It turns out that the viscosity is expected to degenerate near vacuum ($\rho=0$) as a power of the density, which corresponds to the model we have taken above where the viscosity degenerates linearly in the density near vacuum. 
\par We are interested in the limit $\nu \to 0+$ for the selection criteria of 2D physical planar shock. Formally, we recover the 2D compressible Euler equations:
\begin{equation}\label{2DE}
\left\{
\begin{array}{ll}
    \rho_t +{\rm div}(\rho u)=0,  %\hspace{3.5cm} \indent (t,x_1,x_2) \in \R^+ \times \R \times \T,
    \\[2mm]
     (\rho u)_t+{\rm div}(\rho u\otimes u)+\nabla p(\rho)=0.
\end{array}
\right.
\end{equation}
\par There is a large body of work on the idea of using the vanishing viscosity limit to establish existence and stability results for inviscid gases. In the 1D setting, the vanishing viscosity limit using the so-called ``artificial viscosity'' (when the viscosity does not depend on the density, and there is also an artificial diffusion on the continuity equation) has been used to construct entropy solutions to the isentropic Euler system (see \cite{MR719807}, \cite{MR1188982}). The same limit was shown to be valid and stable in the case when the limit solution is small-BV \cite{MR2150387}. The vanishing physical viscosity limit to the small-BV solution, i.e. the limit from Navier-Stokes to Euler equations, is a much harder problem, which was solved recently for the isentropic case in one space dimension, see \cite{KV-Inven}, \cite{NSVS}. Also see the recent work \cite{ACG} for a detailed description of shock formation in the vanishing viscosity limit, and the earlier result for the limit in the relative energy space \cite{MR2683391} using the method of compensated compactness.
\par In the multi-dimensional setting, the vanishing viscosity limit has been established for short times \cite{MR4412066}, \cite{MR3995050}, \cite{MR4704523}. The short-time limitation corresponds to the fact that the limiting Euler solution must be a strong solution, which exists only locally in time. There are very few results that hold after the time of shock formation; let us mention the works of Zumbrun et. al (\cite{MR2224659}, \cite{MR2012648}, \cite{MR2118476}) which construct smooth viscous shocks converging to curved inviscid shocks as the viscosity tends to zero. 
\par 

However, to the best of our knowledge, no existing results have established the vanishing physical viscosity as a uniqueness criterion for multi-dimensional planar shock waves, while simultaneously verifying both the validity and stability of the multi-dimensional vanishing viscosity limit in the regime where the solutions to system (1.1) are themselves weak and potentially discontinuous. The aim of the present paper is to provide some results in this direction, in the case when the limiting inviscid profile is a single planar shock.
It is important to study the inviscid limit to the compressible Euler equations because it is a natural option to select physical solution for Euler equations in multiple space dimensions. It is well known that the entropy BV solution of Euler equations in one space dimension is unique \cite{Bressan-book, Bressan-DeLelis, Chen-Faile-Krupa,Chen-Faile-Krupa2, Chen-Krupa-Vasseur, Chen-Vasseur}. However, the entropy condition is not strong enough to select a unique solution for Euler equations in multiple spatial dimensions, when there is an initial shock wave, even just a planar shock. In a seminal paper  \cite{CDK}, Chiodaroli, De Lellis and Kreml constructed infinitely many 2D solutions starting from the same planar initial shock, using the method of convex integration (see also \cite{MR2564474}). We refer readers to \cite{MR3744380}, \cite{MR4385531}, \cite{MR4164684}, for subsequent non-uniqueness results associated to more general Riemann data. See also results showing density of the set of wild initial data in various topologies \cite{MR4137806}, \cite{MR4772576}, \cite{MR4333736}, and \cite{MR4451484} for non-uniqueness in the class of continuous solutions. A fundamental open problem is how to define a selection criterion for the physical unique solution of Euler equation. In this paper, we provide such a selection criterion using the method of inviscid physical limit, to show the uniqueness of 2D solution for compressible Euler equations, including an initial planar shock. 
%This results shows that among the infinitely many solutions of \cite{CDK}, only one, the planar solution, comes from the zero viscosity limit of the Navier-Stokes equations.
See also the recent interesting works \cite{MR5082200, Feireisl}, in which refined selection criteria based on maximal energy dissipation--purely at the level of compressible Euler equations \eqref{2DE}--are given for dissipative solutions. 
\par Given two states $(\rho_-, u_{1,-},0)$ and $(\rho_+, u_{1,+}, 0)$ verifying the Rankine-Hugoniot conditions for the planar shock, i.e. the 1D jump condition, 
\begin{align}\label{eq:rhcond}
&\text{there exists $\sigma \in \R$ such that } \begin{cases}
(\rho_- u_{1,-}-\rho_+ u_{1,+})=\sigma(\rho_--\rho_+) \\
(\rho_-u_{1,-}^2+p(\rho_-)-\rho_+u_{1,+}^2-p(\rho_+))=\sigma(\rho_-u_{1,-}-\rho_+u_{1,+}) \\
u_{2,-}=u_{2,+}=0,
\end{cases} \\
&\text{and either $\rho_- < \rho_+$ and $u_{1,-}>u_{1,+}$ or $\rho_- > \rho_+$ and $u_{1,-} > u_{1,+}$ holds,} 
\end{align}
there exists an inviscid entropy shock $(\hat{\rho}, \hat{u})$ solving \eqref{2DE} given by the following equation:
\begin{equation}\label{shock}
(\hat{\rho}, \hat{u})(x_1)=\begin{cases}
    (\rho_-, u_{1,-}, 0) & x_1 < 0, \\
    (\rho_+, u_{1,+}, 0) & x_1 > 0.
\end{cases}
\end{equation}
For $\sigma$ verifying \eqref{eq:rhcond}, $(\hat{\rho}, \hat{u})(x_1-\sigma t)$ is an entropy solution to the compressible Euler equations \eqref{2DE}. In this paper, we are interested in the stability of $(\hat{\rho}, \hat{u})(x-\sigma t)$ in the vanishing viscosity limit from 2D Navier-Stokes equations \eqref{eq:2DNS} to 2D Euler equations \eqref{2DE}. 

\subsection{Main results}\label{sec:mainresults}
Our first main result establishes global-in-time existence of weak solutions to \eqref{eq:2DNS} which take different limits when $x_1=\pm \infty$. To handle this, let the functions $\tilde{\rho}(x_1)$ and $\tilde{u}(x_1)$ be two smooth monotone functions satisfying
\begin{align*}
&\tilde{\rho}(x_1)=\rho_\pm \text{ for } \pm x_1 \geq 1, \tilde{\rho}(x_1) > 0, \ \forall \ x_1 \in \R, \\
&\tilde{u}(x_1)=\begin{bmatrix}
\tilde{u_1}(x_1) \\
\tilde{u_2}(x_1) 
\end{bmatrix}=\begin{bmatrix}
u_{1,\pm} \\
u_{2,\pm}
\end{bmatrix} \text{ for } \pm x_1 \geq 1.
\end{align*}
Also define
\begin{align*}
\tilde{v}&=\tilde{u}+\nu \nabla \ln(\tilde{\rho}), \\
\tilde{w}&=\nu \nabla \ln(\tilde{\rho}).
\end{align*}
Further, for functions $(\rho, v, w), (r,V,W):\R^+\times\R \times \T \to (\R_{\geq 0} \times \R^2 \times \R^2)$, we may define the relative $\frac{1}{2}$-entropy on the channel $\R \times \T$:
\begin{align}\label{eq:halfrelentonchannel}
\mathcal{E}_{\frac{1}{2}}(\rho,v,w|r,V,W)=\int_{\R \times \T}\big(\rho\frac{|v-V|^2+|w-W|^2}{2}+F(\rho|r)\big)dx,
\end{align}
where
\begin{align*}
F(\rho)=\frac{\rho^\gamma}{\gamma-1},
\end{align*}
and for any function $f$, the relative quantity is
\begin{align}\label{eq:relativequantity}
f(a|b):=f(a)-f(b)-f'(b)(a-b).
\end{align}
Then, we have the following definition.
\begin{definition}\label{def:suitableinitial}
The initial data $(\rho_0, v_0, w_0):\R \times \T \to \R_{\geq 0} \times \R^2 \times \R^2$ is ``suitable'' if it verifies 
\begin{align*}
\sqrt{\rho_0}u_0&=\sqrt{\rho_0}v_0-\sqrt{\rho_0}w_0 \text{ in $L^2_{\text{loc}}(\R \times \T)$}, \\
\sqrt{\rho_0} w_0&=2\nu \nabla \sqrt{\rho_0}, \indent \rho_0 w_0=\nu \nabla \rho_0 \text{ in }\mathcal{D}'(\R \times \T),
\end{align*}
and there exists a constant $c_0 >0$ such that the following conditions hold. Let 
\begin{align}\label{eq:mean}
(\bar\rho_0, \bar u_0)(x_1)&:=\int_{\T}(\rho_0,u_0)(x_1,x_2)dx_2, \\
\bar w_0(x_1)&:=\nu \nabla \ln(\bar \rho_0(x_1)), \\
\bar v_0(x_1)&:=\bar u_0(x_1)+\bar w_0(x_1).
\end{align}
Then
\begin{align}
&\mathcal{E}_{\frac{1}{2}}(\rho_0,v_0, w_0|\tilde{\rho}, \tilde{v}, \tilde{w}) < \infty, \\
&\mathcal{E}_{\frac{1}{2}}(\rho_0,v_0, w_0|\bar\rho_0,\bar v_0, \bar{w}_0) <\infty, \\
&||(\bar\rho_0,\bar u_0)-(\tilde{\rho}, \tilde{u})||_{H^6(\R)} <\infty, \label{item:3} \\
&|\partial_{x_1}(\bar{u}_{0})_1| \leq \frac{1}{2\nu}(\bar{\rho}_0)^{\gamma-1}, \\
& \bar{\rho}_0(x_1) \geq c_0 > 0. 
\end{align}
\end{definition}
The first main theorem is the following.
\begin{theorem}\label{thm:intro1}
Fix any far field states $\rho_\pm, \begin{bmatrix}
u_{1,\pm} \\
u_{2,\pm}
\end{bmatrix}$. Let $(\rho_0, v_0, w_0)$  be suitable initial data in the sense of \Cref{def:suitableinitial}. Then, there exists a global-in-time weak solution $(\rho, v, w)$, in the sense of \Cref{def:defofweak}, to $\eqref{eq:2DNS}$ verifying, for any $T >0$,
\begin{align*}
\mathcal{E}_{\frac{1}{2}}(\rho,v,w|\tilde{\rho},\tilde{v},\tilde{w})(t) \leq C(T) \text{ for any } t \in [0,T].
\end{align*}
 Further, there exists a unique  planar strong solution $(\bar{\rho}, \bar{v}, \bar{w})(x_1)$ to $\eqref{eq:2DNS}$ with averaged initial data $(\overline{\rho}_0, \overline{v}_0, \overline{w}_0)$ defined in \eqref{eq:mean}.
Finally, these verify the estimate
\begin{align}\label{eq:weakstrongestimatechannel}
&\mathcal{E}_{\frac{1}{2}}(\rho, v, w|\bar{\rho}, \bar{v}, \bar{w})(t) \leq \\
&Ce^{\big(||\partial_{x_1} \bar{v}||_{L^\infty([0,t] \times \R)}+||\partial_{x_1}\bar{w}||_{L^\infty([0,t] \times \R )}+||\partial_{x_1x_1}\ln(\bar{\rho})||_{L^\infty([0,t] \times \R) }+||\partial_{x_1} \ln(\bar{\rho})||_{L^\infty([0,t] \times \R )}^2\big)t} \\
&\times \mathcal{E}_{\frac{1}{2}}(\rho_0, v_0, w_0|\bar{\rho}_0, \bar{v}_0, \bar{w}_0), \qquad \forall t\in [0,T].
\end{align}
\end{theorem}
\begin{remark}
We use here the notation $(\rho,v,w)$ corresponding to the augmented system introduced by Bresch, Desjardins and Zatorska \cite{MR3394617}. Formally, $v=u+\nu\nabla \ln(\rho)$ and $w=\nu \nabla \ln(\rho)$. However, the a priori estimates do not allow defining $\ln(\rho)$. Note that, considering the pair $(\rho,u)=(\rho,v-w)$, it is enforced in \Cref{def:defofweak} that $(\rho,u)$ verifies \eqref{eq:2DNS} in the sense of distributions. So our notion of weak solution in \Cref{def:defofweak} recovers the classical notion of weak solution.
\end{remark}
    \begin{remark}
   There are no smallness constraints on the initial data in \Cref{def:suitableinitial}.
\end{remark}
\begin{remark}
The condition $||(\bar\rho_0,\bar u_0)-(\tilde{\rho}, \tilde{u})||_{H^6(\R)} <\infty$ may be relaxed to $H^4(\R)$; we enforce $H^6(\R)$ to additionally obtain that $(\rho, v, w)$ verifies the relative $\frac{1}{2}$-entropy inequality \eqref{eq:weakstrongestimatechannel}.
\end{remark}
Global existence to the compressible Navier-Stokes equations with density-dependent viscosity, as in \eqref{eq:2DNS}, remained an open problem for many years after the corresponding problem with constant viscosity was solved. Global existence of weak solutions on the domain $\T^d$ was first shown by Vasseur-Yu \cite{MR3573976} and refined in \cite{MR3801754}, \cite{MR4404790}, based on the a priori estimates given by the Bresch-Desjardins (BD) entropy \cite{MR1989675}, \cite{MR1978317}). Global existence of weak solutions has also been shown on $\R^d$ (beginning with the %breakthrough 
work of Li-Xin \cite{lixin}), with constant fixed density and zero velocity at $|x|=\infty$ \cite{lixin}, \cite{MR4254549}. \Cref{thm:intro1} extends the constructions to the unbounded domain $\R \times \T$ with possibly two different limits for the density and velocity at $x_1=\pm \infty$, which is a suitable framework to consider shocks.
\par Our second main theorem is on stability and uniqueness of the planar shock \eqref{shock} in the vanishing viscosity limit of the solutions constructed in \Cref{thm:intro1}. To handle the stability and uniqueness of the entropy shock, we use the relative entropy associated to the entropy of the Euler system \eqref{2DE}. For any functions $\rho_1, u_1, \rho_2,u_2$, define
\begin{align*}
\eta(\rho_1, u_1|\rho_2, u_2):=\rho_1\frac{|u_1-u_2|^2}{2}+F(\rho_1|\rho_2).
\end{align*}
The second main theorem is the following.
\begin{theorem}\label{thm:intro2}
Let $\gamma \in (1,2]$. Then, there exists an integer $j > 0$, a constant $C>0$, and a function $C^*:\R^+ \to \R^+$, with $C^*(\nu) \to 0$ as $\nu \to 0+$, so that the following holds. For each $\nu > 0$, consider \eqref{eq:2DNS}. For a given constant state $(\rho_-, u_{1,-},0)$, there exists $\epsilon_0 > 0$ so that for any $\eps < \epsilon_0$ and any $(\rho_+, u_{1,+},0)$ satisfying \eqref{eq:rhcond} with $|\rho_+-\rho_-|=:\eps$, the following holds.
\par Let $E_0 > 0$, and let $(\rho^\nu_0, v^\nu_0, w^\nu_0)$ be a sequence of suitable initial data in the sense of \Cref{def:suitableinitial}, additionally verifying
\begin{align}\label{eq:nearplanar}
    &\limsup_{\nu \to 0} \int_{\R} \eta(\bar{\rho}^\nu_0,\bar{u}^\nu_0|\hat{\rho}, \hat{u})dx_1=E_0, \\
    &\lim_{\nu \to 0}\nu^2\int_{\R}\bar{\rho}^\nu_0|\ln(\bar{\rho}^\nu_0)_{x_1}|^2dx_1=0, \\
    & ||(\bar{\rho}^\nu_0,\bar{u}^\nu_0)-(\tilde{\rho}, \tilde{u})||_{H^6(\R)} \leq \frac{1}{\nu}, \\
    & \bar{\rho}_0^\nu \geq c_0^\nu \geq \frac{1}{\nu}, \\
    &\mathcal{E}_{\frac{1}{2}}(\rho^\nu_0, v^\nu_0, 
    w^\nu_0|\bar{\rho}^\nu_0, \bar{v}^\nu_0, \bar{w}^\nu_0) \leq C^*(\nu). \\
\end{align}
Let $(\rho^\nu, v^\nu, w^\nu)$ be a sequence of solutions to \eqref{eq:2DNS} constructed in \Cref{thm:intro1} with initial data $(\rho_0^\nu, v_0^\nu, w_0^\nu)$. Then, for a given $T>0$, there exists a constant $C>0$ and limits $\rho_\infty, u_\infty$ such that as $\nu \to 0+$ (up to a subsequence)
\begin{align}\label{eq:finalinviscidlimits}
    \sqrt{\rho^\nu}(v^\nu+w^\nu) \rightharpoonup \sqrt{\rho_\infty}u_\infty \text{ weakly in $L^2_{\text{loc}}([0,T] \times \R \times \T)$}, \\
    \sqrt{\rho^\nu}u^\nu \rightharpoonup \sqrt{\rho_\infty}u_\infty \text{ weakly in $H^{-1}_{\text{loc}}([0,T] \times \R \times \T)$}, \\
    \sqrt{\rho^\nu} \rightharpoonup \sqrt{\rho_\infty} \text{ weakly in $L^{2\gamma}_{\text{loc}}([0,T] \times \R \times \T)$}.
\end{align}
In addition, there exists a shift $X_\infty(t) \in L^\infty([0,T])$ so that for almost every $t \in [0,T]$
\begin{align}\label{eq:finalstabilityest}
    \int_{\R \times \T}\Big(\rho_\infty(t,\cdot) \frac{|u_\infty(t,\cdot)-\hat{u}(x-X_\infty(t))|^2}{2}+F(\rho_\infty(t,\cdot)|\hat{\rho}(x-X_\infty(t)))\Big)dx \leq C(E_0+\sqrt{E_0}).
\end{align}
Moreover, the shift $X_\infty(t)$ satisfies 
\begin{align}\label{eq:finalshiftest}
||X_\infty(\cdot)-\sigma \cdot||_{L^\infty([0,T])} \leq \frac{C}{|\rho_--\rho_+|^2}\Big(E_0+\sqrt{E_0}\Big).
\end{align}
Therefore, planar entropy shocks (with small amplitude) of the $2D$ isentropic Euler system \eqref{2DE} are unique and stable in the class of %near-planar 
weak inviscid limits of solutions to the $2D$ compressible Navier-Stokes system \eqref{eq:2DNS}.
\end{theorem}
Some comments are in order:
\begin{remark} More growth in $\nu$ of the quantity $||(\bar{\rho}_0^\nu,\bar{u}_0^\nu)-(\tilde{\rho}, \tilde{u})||_{H^6(\R)}$ may be allowed, at the price of faster decay of the function $C^*(\nu)$ to zero as $\nu \to 0+$. For simplicity, we do not address this here.
\end{remark}
    
   % \begin{remark}The smallness condition $|\rho_--\rho_+| < \eps_0$ is not needed in the proof of \Cref{thm:intro2}. It is used only in the proof of \cite[Theorem 1.2]{KV-Inven}, which we use to estimate the difference between a planar strong solution to \eqref{eq:2DNS} and a planar viscous shock.\end{remark}
  \begin{remark}We present the results for $d=2$ as our main example is the shallow water equations for $\gamma=2$. However, the same analysis gives similar results for $d \geq 3$ on $\R \times \T^{d-1}$ and the more general viscosity functions satisfying Bresch-Desjardins relation \eqref{BD}.
   % \item
\end{remark}
The rest of the paper is as follows. We explain the main ideas of the proofs of \Cref{thm:intro1} and \Cref{thm:intro2} in \Cref{sec:ideas}. In \Cref{sec:multid_existence_channel}, we prove \Cref{thm:intro1}. Finally, in \Cref{sec:vvlimit}, we prove \Cref{thm:intro2}.

\section{Ideas of the proof}\label{sec:ideas}
We describe in this section the methodology of our results. \newline

\underline{\textbf{Weak solutions to $2D$ Navier-Stokes in $\R \times \T$}.} The first objective is to construct weak solutions to \eqref{eq:2DNS} on the unbounded domain $\R \times \T$ for fixed $\nu >0$. In this step, $\nu > 0$ is fixed, so we will omit indexing the solution by $\nu$. This already presents significant challenges due to the density dependent viscosity. The natural conserved quantities are the energy and the BD-entropy. However, these are not strong enough to control many desired quantities; for example, $\ln(\rho)$ and $\nabla u$ are generally not defineable, and there is no way to rule out vacuum formation, even if it is not initially present. 
\par This issue was solved (on bounded domains) in \cite{MR3573976}, \cite{MR3801754}, where the authors pass into the limit $\eps_0, \eps_1, \eps_2 \to 0+$ for solutions of the following ``quantum Navier-Stokes equations with drag terms'' (see also \cite{lixin} for global existence of weak solutions in $\mathbb{R}^2$ via a different strategy):
\begin{equation}\label{eq:idea1}
\left\{
\begin{array}{ll}
    \di (\rho_R)_t +{\rm div}(\rho_R u_R)=0,\qquad    (t,x) \in \R^+ \times (\R/3R\mathbb{Z}) \times \T,\  R>0, \\[3mm]
    \di  (\rho_R u_R)_t+{\rm div}(\rho_R u_R\otimes u_R)+\nabla p(\rho_R)={\rm div} (2 \nu \rho_R\mathbb{D}u_R)\\[3mm]
   \di \hspace{5.5cm}  +2\eps_2 \rho_R \nabla\big(\frac{\Delta \sqrt{\rho_R}}{\sqrt{\rho_R}})-\eps_0u_R-\eps_1\rho_R|u_R|^2u_R, \\[3mm]
     (\rho_R(0,x), (\rho_R u_R)(0,x))=(\rho_{R,0}(x),m_{R,0}(x)).
\end{array}
\right.
\end{equation}
\par Our strategy is to pass into the limits $\eps_2,\eps_0, \eps_1 \to 0+$ and $R \to +\infty$ for the solutions to the above system constructed in \cite{MR3801754}. Note that on unbounded domains it is not natural to expect for the initial data to have finite energy and entropy, especially if we want to enforce that the solution has the same far field states as the planar shock $(\hat{\rho}, \hat{u})$. Instead, we impose that the initial data has finite \textit{relative} energy and entropy to its mean (in $x_2$). The first difficulty is to prove that the relative energy and entropy between the weak solution and its mean remains bounded \textit{uniformly} in the width $R$ of the spatial domain. To solve this, we use the ``relative $\kappa$-entropy'' structure first introduced in \cite{MR3734208} at the level of the Navier-Stokes equations ($\eps_2=\eps_0=\eps_1=0$). However, we rigorously justify the relative entropy structure by establishing it at the level $\eps_2, \eps_0, \eps_1 >0$, where all the computations may be justified, and then passing into the limit $\eps_2, \eps_0, \eps_1 \to 0+$. The estimate we obtain is roughly the following
\begin{align}\label{eq:idea2}
&\sup_{t \in [0,T]}\Big(\int_{(\R / 3R\mathbb{Z}) \times \T}\eta(\rho_R, u_R|\bar{\rho}_R, \bar{u}_R)dx_1dx_2\Big)(t) \\
&\leq Ce^{\big(\|\partial_{x_1} \bar{\rho}_R\|_{L^\infty([0,T];H^3)}+\|\partial_{x_1} \bar{u}_R\|_{L^\infty([0,T];H^2)}\big)T}  \int_{(\R / 3R\mathbb{Z}) \times \T}\eta(\rho_{R,0}, u_{R,0}|\bar{\rho}_{R,0}, \bar{u}_{R,0})dx_1dx_2,
\end{align}
where $(\rho_R, u_R)$ is the weak solution to \eqref{eq:idea1} with $\eps_2, \eps_0, \eps_1=0$ and
\begin{align*}
(\bar{\rho}_{R,0}, \bar{u}_{R,0})(x_1):=\int_{\T}(\rho_{R,0}(x_1, x_2), u_{R,0}(x_1,x_2))dx_2,
\end{align*}
and $(\bar{\rho}_R, \bar{u}_R)$ solves the following reduction of \eqref{eq:idea1} with $\eps_2,\eps_0,\eps_1=0$ for solutions depending only on $x_1$
\begin{equation}\label{eq:idea3}
\left\{
\begin{array}{ll}
    (\bar \rho_R)_t +(\bar \rho_R \bar u_{R,1})_{x_1}=0,  \indent (t,x_1) \in \R^+ \times (\R/3R\mathbb{Z}), \\[2mm]
     (\bar\rho_R \bar u_{R,1})_t+(\bar\rho_R \bar u_{R,1}^2)_{x_1}+p(\bar\rho_R)_{x_1}=(2 \nu \bar\rho_R (\bar u_{R,1})_{x_1})_{x_1}, \\[2mm]
     (\bar\rho_R \bar u_{R,2})_t+(\bar\rho_R \bar u_{R,1}\bar u_{R,2})_{x_1}=(\nu \bar\rho_R (\bar u_{R,2})_{x_1})_{x_1},\\[2mm]
     (\bar\rho_R(0,x_1), (\bar\rho_R \bar u_R)(0,x_1))=(\bar{\rho}_{R,0}(x_1), \bar{m}_{R,0}(x_1)).
\end{array}
\right.
\end{equation}
The next step is to establish uniform-in-$R$ bounds on the derivatives of the planar strong solution $(\bar{\rho}_R, \bar{u}_R)$ to \eqref{eq:idea3}. To do so we revisit the regularity theory for the $1$-d Navier-Stokes equations (see e.g. \cite{MR2368905}, \cite{MR4049919}) and prove that one may have uniform derivative bounds invariant in the size of the spatial domain. Note that \eqref{eq:idea3} is not exactly the $1$-d Navier-Stokes equations due to the presence of the equation for $\bar{u}_{R,2}$, so additional considerations must be made. Combined with \eqref{eq:idea2}, this provides uniform-in-$R$ bounds on some norms of the weak solutions.
\par With the bounds in hand, we are ready to pass into the limit $R \to +\infty$. In order to pass into the limit in the momentum equation (more specifically the term ${\rm div}(\rho_R u_R\otimes u_R)$), it is introduced in \cite{MR3801754} the concept of solutions which are renormalized in the velocity variable. This was used already to pass into the limit $\eps_2, \eps_0, \eps_1 \to 0+$ in \eqref{eq:idea1}. Here, we use the renormalization of the velocity variable again to pass into the limit in $R \to +\infty$ and obtain a weak solution $(\rho^\nu, u^\nu)$ to \eqref{eq:2DNS}, \eqref{sw}. \newline 

\underline{\textbf{The vanishing viscosity limit}.} Next, we want to take the limit $\nu \to 0+$ in the solutions $(\rho^\nu, u^\nu)$ constructed in the previous step. We can retain $\eqref{eq:idea2}$ in the limit $R \to +\infty$ and already have the estimate
\begin{align}\label{eq:idea4}
&\sup_{t \in [0,T]}\Big(\int_{\R \times \T}\eta(\rho^\nu, u^\nu|\bar{\rho}^\nu, \bar{u}^\nu)dx_1dx_2\Big)(t) \\
&\leq Ce^{\big(||\partial_{x_1} \bar{\rho}^\nu||_{L^\infty([0,T];H^3(\R))}+||\partial_{x_1} \bar{u}^\nu||_{L^\infty([0,T];H^2(\R))}\big)T} \int_{\R \times \T}\eta(\rho^\nu_0, u^\nu_0|\bar{\rho}^\nu_0, \bar{u}^\nu_0)dx_1dx_2,
\end{align}
where again $(\bar{\rho}^\nu, \bar{u}^\nu)$ is the strong solution to \eqref{eq:2DNS} depending only on $x_1$ generated by the mean of the initial data in $x_2$. The next goal is to obtain the estimate on the relative entropy between the planar strong solution and the inviscid shock
\begin{align}\label{eq:idea5}
\sup_{t \in [0,T]}\Big(\int_{\R \times \T}\eta\big(\bar{\rho}^\nu, \bar{u}^\nu|\hat{\rho}(x_1-X^\nu_E(\cdot)), \hat{u}(x_1-X^\nu_E(\cdot))\big)dx_1dx_2\Big)(t) \leq C(E_0+\nu),
\end{align}
for some appropriate shift function $X^\nu_E(t)$ (recall $E_0$ from \Cref{thm:intro2}). This can be obtained by swapping to the Lagrangian coordinates under the Euler-Lagrange change of variable
\begin{align*}
y_1(t,x_1)= \int_{(0,0)}^{(t,x_1)}\overline{\rho}^\nu dx_1- \overline{\rho}^\nu \overline{u}^\nu dt,
\end{align*}
using the $a$-contraction estimate \cite[Theorem 1.2]{KV-Inven}, which shows a relative entropy estimate (uniformly-in-$\nu$) between the Lagrangian solution and the Lagrangian shock, and then swapping back to the Eulerian coordinates. This gives \eqref{eq:idea5}. Note, however, that the Euler-Lagrange change of variables itself depends on $\nu$. Since all the $L^\infty$ bounds on $\bar{\rho}^\nu$ and $\bar{u}^\nu$ grow in $\nu$, the change of variables itself becomes degenerated as $\nu \to 0+$. This causes loss of compactness on the shift functions $X^\nu_E(t)$, which has to be very delicately recovered. In the end, however, we can only obtain uniform bounds on the shifts in $L^\infty$ compared to the uniform bounds in $W^{1,1}$ obtained in \cite{KV-Inven}. Actually, we expect this loss due to the degeneration of the change of variables. 
\par Then, we expect that \eqref{eq:idea4} and \eqref{eq:idea5} may be combined to show
\begin{align}\label{eq:idea6}
&\sup_{t \in [0,T]}\Big(\int_{\R \times \T}\eta\big(\rho^\nu, u^\nu|\hat{\rho}(x_1-X^\nu_E(\cdot)), \hat{u}(x_1-X^\nu_E(\cdot))\big)dx_1dx_2\Big)(t) \\
&\leq C(E_0+\nu)+Ce^{\big(\|\partial_{x_1} \bar{\rho}^\nu\|_{L^\infty([0,T];H^3(\R))}+\|\partial_{x_1} \bar{u}^\nu\|_{L^\infty([0,T];H^2(\R))}\big)T} \int_{\R \times \T}\eta(\rho^\nu_0, u^\nu_0|\bar{\rho}^\nu_0, \bar{u}^\nu_0)dx_1dx_2.
\end{align}
However, the relative entropy is not exactly a norm, so some error terms arise when estimating the difference. Since we are working with genuinely weak solutions to the 2D Navier-Stokes equation \eqref{eq:2DNS}, these error terms present major difficulties. In particular, we have to deal with the fact that $\rho^\nu$ may attain vacuum and there are no uniform $L^\infty$ bounds on $(\rho^\nu, u^\nu)$. Nevertheless, detailed analysis near the vacuum regions allows for \eqref{eq:idea6} to be obtained.
\par Finally, we are ready to take the vanishing viscosity limit. The idea is to enforce that $\di\int_{\R \times \T}\eta(\rho^\nu_0, u^\nu_0|\bar{\rho}^\nu_0, \bar{u}^\nu_0)dx_1dx_2$ is sufficiently small so that the term
\begin{align*}
Ce^{\big(\|\partial_{x_1}\bar{\rho}^\nu\|_{L^\infty([0,T];H^3(\R))}+\|\partial_{x_1}\bar{u}^\nu\|_{L^\infty([0,T];H^2(\R))}\big)t} \int_{\R \times \T}\eta(\rho^\nu_0, u^\nu_0|\bar{\rho}^\nu_0, \bar{u}^\nu_0)dx_1dx_2,
\end{align*}
converges to zero as $\nu \to 0+$. Clearly, this will depend on how fast the Sobolev norms of the planar strong solution $(\bar{\rho}^\nu, \bar{u}^\nu)$ grow asymptotically as $\nu \to 0+$. %\red{Via previous analyses, one would expect these high order norms to grow at a rate in $\nu$ like $\exp(\exp(\exp(\frac{1}{\nu^j})))$ for some integer $j > 0$. This would force the constant $C^*(\nu)$ in \Cref{thm:intro2} to converge to zero as $\nu \to 0$ at a rate $\exp(\exp(\exp(\exp(-\nu^{-j}))))$ to beat the growth term.} 
In this last step, we revisit the analysis of the $1$-d Navier-Stokes equations and prove the optimal bound
\begin{align*}
||\partial_{x_1} \bar{\rho}^\nu||_{L^\infty([0,T];H^3(\R))}+||\partial_{x_1}\bar{u}^\nu||_{L^\infty([0,T];H^2(\R))} \leq \frac{C}{\nu^j},
\end{align*}
for some integer $j >0$. Thus, enforcing 
\begin{align*}
C^*(\nu) \leq \exp(-C\nu^{-j}),
\end{align*}
gives 
\begin{align*}
\sup_{t \in [0,T]}\Big(\int_{\R \times \T}\eta\big(\rho^\nu, u^\nu|\hat{\rho}(x_1-X^\nu_E(\cdot)), \hat{u}(x_1-X^\nu_E(\cdot))\big)dx_1dx_2\Big)(t) \text{ is uniformly bounded in $\nu$},
\end{align*}
which allows to take the vanishing viscosity limit $\nu \to 0+$ via standard techniques. 

\section{Weak solutions to 2D Navier-Stokes equations in $\R \times \T$}\label{sec:multid_existence_channel}
\setcounter{equation}{0}
Fix $\nu > 0$. In this section, under suitable conditions on the initial data (\Cref{def:suitableinitial}), we will define a notion of weak solution to \eqref{eq:2DNS} on the channel $\R \times \T$ that takes possibly two different limits at $x_1=\pm \infty$. Let $\T_R:=[-R,2R]_{\text{periodic}} \times \T \sim \R/3R\mathbb{Z} \times \T$, where we have identified $R/3R\mathbb{Z}$ with $[-R,2R]_{\text{periodic}}$. The strategy is to consider the renormalized weak solutions on $\T_R$ constructed in \cite{MR3801754} and pass into the limit $R \to \infty$ on the partial domain $[-M, M] \times \T \subset [-R,R/2]\times \T$. First, we revisit the construction of renormalized weak solutions to the Navier-Stokes equations on $\T_R$. 
\begin{notation}
In this section, we will work with a fixed constant $\nu > 0$, so we will omit the dependence of any functions on $\nu$ for brevity. However, we will carefully track the dependence of any constants on $\nu$, as this will be important when we take the limit $\nu \to 0+$ in forthcoming sections.
\end{notation}
\subsection{Renormalized weak solutions to 2D Navier-Stokes equations on $\T_R$} 
In this section, we review and refine the construction of renormalized weak solutions for the following system:
\begin{equation}\label{eq:2DNS_T_R}
\left\{
\begin{array}{ll}
    \rho_t +{\rm div}(\rho u)=0, \qquad\qquad\  \indent (t,x) \in \R^+ \times \T_R,\\[3mm]
     (\rho u)_t+{\rm div}(\rho u\otimes u)+\nabla p(\rho)={\rm div} (2 \nu \sqrt{\rho} \mathcal{S}(u)), \\[3mm]
     (\rho(0,x), (\rho u)(0,x))=(\rho_0(x),m_0(x)),
\end{array}
\right.
\end{equation}
where $\sqrt{\rho}\mathcal{S}(u)$ is formally equal to $\rho \mathbb{D}u$. More specifically, it will be shown that $\mathcal{S}(u)$ is the symmetric part of the matrix $\mathcal{T}(u)$ verifying
\begin{align}\label{eq:T(u)def}
\sqrt{\rho}\mathcal{T}(u)=\nabla(\rho u)-2\sqrt{\rho}u \otimes \nabla\sqrt{\rho}, \text{ in $\mathcal{D}'$}.  
\end{align}
The energy and BD entropy inequalities provide control on the following quantity
\begin{align*}
\mathcal{E}_0(\rho, u)=\int_{\T_R}\big(\rho \frac{|u|^2}{2}+4\nu^2|\nabla(\sqrt{\rho})|^2+F(\rho)\big)dx,
\end{align*}
where 
\begin{align}\label{eq:enthalpy}
F(\rho)=\frac{\rho^\gamma}{\gamma-1},  
\end{align}
with associated dissipation
\begin{align*}
\mathcal{D}_0(\rho, u)=\int_{\T_R}\big(\nu|\nabla \rho^{\frac{\gamma}{2}}|^2+\nu|\mathcal{T}(u)|^2\big)dx.
\end{align*}
For technical reasons, we will also work with the so-called quantum Navier-Stokes equations with drag forces, which constitute the following system
\begin{equation}\label{eq:2DQNS_T_R}
\left\{
\begin{array}{ll}
    \rho_t +{\rm div}(\rho u)=0, \qquad \qquad\quad  \indent (t,x) \in \R^+ \times \T_R,\\[3mm]
     (\rho u)_t+{\rm div}(\rho u\otimes u)+\nabla p(\rho)={\rm div} (2 \nu \sqrt{\rho}\mathcal{S}(u)+\frac{\eps_2}{\nu}\sqrt{\rho}\mathcal{T}(w))-\eps_0u-\eps_1\rho|u|^2u, \\[3mm]
     (\rho(0,x), (\rho u)(0,x))=(\rho_0(x),m_0(x)),
\end{array}
\right.
\end{equation}
for $\eps_0, \eps_1, \eps_2 \geq 0$, where formally 
\begin{align*}
{\rm div}(\frac{\eps_2}{\nu}\sqrt{\rho}\mathcal{T}(w))=2\eps_2 \rho \nabla \big(\frac{\Delta \sqrt{\rho}}{\sqrt{\rho}}\big).
\end{align*}
More specifically, it is enforced that 
\begin{align}\label{eq:S_epsdef}
\frac{1}{\nu}\sqrt{\rho}\mathcal{T}(w)=2\big(\sqrt{\rho}\big(\nabla^2\sqrt{\rho}-4\big(\nabla \rho^{\frac{1}{4}} \otimes \nabla \rho^{\frac{1}{4}})\big)\big)=2\nabla(\sqrt{\rho}\nabla \sqrt{\rho})-4\nabla \sqrt{\rho} \otimes \nabla \sqrt{\rho}.
\end{align}
At the point, the notation $\mathcal{T}(w)$ for the quantum term may appear unmotivated, as no quantity $w$ has been defined. We anticipate formally defining $w=\nu \nabla \ln(\rho)$, and $\mathcal{T}(w)$ is formally $\sqrt{\rho}\nabla w$. In the end, however, there is not enough a priori estimates to define $\ln(\rho)$, and $w$ must be implicitly defined through \eqref{eq:vandwrelations}, and $\mathcal{T}(w)$ implicitly defined through \eqref{eq:S_epsdef}.
\par The energy and BD-entropy for the system \eqref{eq:2DQNS_T_R} provide control on
\begin{align*}
\mathcal{E}_{\eps}(\rho, u)=\int_{\T_R}\big(\rho \frac{|u|^2}{2}+(2\eps_2+4\nu^2)|\nabla(\sqrt{\rho})|^2+F(\rho)+\eps_0\nu(\rho-\ln(\rho))\big)dx,
\end{align*}
with associated dissipation
\begin{align*}
\mathcal{D}_{\eps}(\rho, u)=\int_{\T_R}\Big(\nu|\nabla \rho^{\frac{\gamma}{2}}|^2+\nu \eps_2(|\nabla \rho^{\frac{1}{4}}|^4+|\nabla^2\sqrt{\rho}|^2)+\nu |\mathcal{T}(u)|^2+\eps_0|u|^2+\eps_1\rho |u|^4\Big)dx.
\end{align*}

These provide the following a priori estimates.
\begin{align}\label{eq:aprioriestimatesforquantum}
&\sqrt{\rho} \in L^\infty(\R^+;L^{2\gamma}(\T_R)), \indent \nabla \sqrt{\rho} \in L^\infty(\R^+;L^{2}(\T_R)), \indent \nabla \rho^{\frac{\gamma}{2}} \in L^2(\R^+;L^{2}(\T_R)), \notag \\
&\sqrt{\rho}u \in L^\infty(\R^+;L^{2}(\T_R)), \indent \mathcal{T}(u) \in L^2(\R^+;L^{2}(\T_R)), \indent \sqrt{\kappa}\nabla^2\sqrt{\rho} \in L^2(\R^+;L^{2}(\T_R)), \notag \\
&\kappa^{\frac{1}{4}}\nabla \rho^{\frac{1}{4}}\in L^4(\R^+;L^4(\T_R)), \indent \eps_1^{\frac{1}{4}}\rho^{\frac{1}{4}}u \in L^4(\R^+;L^4(\T_R)), \notag \\
& \eps_0^{\frac{1}{2}}u \in L^2(\R^+;L^{2}(\T_R)), \indent \eps_0\ln(\rho) \in L^\infty(\R^+;L^{1}(\T_R)).
\end{align}

Now, we give the various definitions of weak solutions that we will consider. The following definitions will be valid for all ranges of parameters $\eps_0,\eps_2 \geq 0$, $0 < \nu < \frac{4}{3}$, $0 \leq \eps_1 < \frac{2}{3}$.
\begin{definition}
We say that $(\rho, u)$ is a weak solution to $\eqref{eq:2DQNS_T_R}$ if it verifies the a priori estimates \eqref{eq:aprioriestimatesforquantum}, and for any function $\psi \in C^\infty_c(\R^+ \times \T_R)$:
\begin{align}\label{eq:uequation}
&\int_0^\infty \int_{\T_R}(\rho \psi_t+\rho u \cdot \nabla \psi) dx dt =0, \\
&\int_0^\infty \int_{\T_R}\Big[\rho u \psi_t+\big(\rho u \otimes u-(2\nu \sqrt{\rho}\mathcal{S}(u)+\frac{\eps_2}{\nu}\sqrt{\rho}\mathcal{T}(w))\big)\cdot \nabla \psi\\
&\qquad\qquad\quad  +(-2\rho^{\frac{\gamma}{2}}\nabla (\rho)^{\frac{\gamma}{2}}-\eps_0u-\eps_1\rho|u|^2u)\psi\Big]dx dt=0,
\end{align}
where $\mathcal{T}(w)$ verifies \eqref{eq:S_epsdef} and $\mathcal{S}(u)$ is the symmetric part of $\mathcal{T}(u)$ which verifies $\eqref{eq:T(u)def}$, and for any $\overline{\psi} \in C^\infty_c(\T_R)$
\begin{align*}
&\lim_{t \to 0} \int_{\T_R}\rho(t,x)\overline{\psi}(x)dx=\int_{\T_R}\rho_0(x)\overline{\psi}(x)dx, \\
&\lim_{t \to 0} \int_{\T_R}\rho(t,x)u(t,x)\overline{\psi}(x)dx=\int_{\T_R}m_0(x)\overline{\psi}(x)dx.
\end{align*}
\end{definition}

\begin{definition}
We say that $(\rho, u)$ is a renormalized weak solution to \eqref{eq:2DQNS_T_R} if it verifies the a priori estimates \eqref{eq:aprioriestimatesforquantum}, and for any function $\phi \in W^{2,\infty}(\R^2)$ compactly supported, there exist two Radon measures $\Xi_{\phi}, \widetilde\Xi_{\phi} \in \mathcal{M}(\R^+ \times \T_R)$, with
\begin{align*}
||(\Xi_{\phi}, \widetilde\Xi_{\phi})||_{\mathcal{M}([0,T] \times \T_R)} \leq C||\phi''||_{L^\infty(\R^2)},
\end{align*}
for any $T>0$, where the constant $C$ depends only on the solution $(\rho, u)$ and $T$, and for any function $\psi \in C^\infty_c(\R^+ \times \T_R)$:
\begin{align*}
&\int_0^\infty \int_{\T_R}(\rho \psi_t+\rho u \cdot \nabla \psi) dx dt =0, \\
&\int_0^\infty \int_{\T_R}\Big[\rho \phi(u) \psi_t+\big(\rho \phi(u) u-(2\nu \sqrt{\rho}\mathcal{S}(u)+\frac{\eps_2}{\nu}\sqrt{\rho}\mathcal{T}(w))\phi'(u)\big)\cdot \nabla \psi\\
&\qquad\qquad\quad  (-2\rho^{\frac{\gamma}{2}}\nabla (\rho)^{\frac{\gamma}{2}}-\eps_0u-\eps_1\rho|u|^2u)\phi'(u)\psi\Big]dx dt=\left\langle \Xi_{\phi}, \psi\right\rangle,
\end{align*}
where $\sqrt{\rho}\mathcal{T}(w)$ verifies \eqref{eq:S_epsdef} and $\mathcal{S}(u)$ is the symmetric part of $\mathcal{T}(u)$ which verifies
$$
\nu\sqrt{\rho}\phi^\prime(u)[\mathcal{T}(u)]_{jk}:=\nu\partial_j\big(\rho \phi^\prime(u) u_{k}\big)-2\nu\sqrt{\rho}u_{k}\phi^\prime(u)\partial_j\sqrt{\rho}+\widetilde\Xi _{\phi},
$$
and for any $\overline{\psi} \in C^\infty_c(\T_R)$
\begin{align}\label{eq:uinitdata}
&\lim_{t \to 0} \int_{\T_R}\rho(t,x)\overline{\psi}(x)dx=\int_{\T_R}\rho_0(x)\overline{\psi}(x)dx, \\
&\lim_{t \to 0} \int_{\T_R}\rho(t,x)u(t,x)\overline{\psi}(x)dx=\int_{\T_R}m_0(x)\overline{\psi}(x)dx.
\end{align}
\end{definition}

We note that a renormalized weak solution is actually a weak solution in the usual sense \cite[Section 4]{MR3801754}. Finally, we can put even more structure on a (renormalized) weak solution to obtain a ``$\frac{1}{2}$-entropy'' (renormalized) weak solution. For this, we need to define the auxiliary velocities $v$ and $w$, which are, formally, $v=u+\nu \nabla \ln(\rho)$ and $w=v-u$. However, we really need to define $w$ through \Cref{eq:vandwrelations} for $v$ and $w$ to be defined at the limit $\eps_0=\eps_1=\eps_2=0$.
\begin{definition}
A (renormalized) weak solution to \eqref{eq:2DNS_T_R} is called {\it $\frac{1}{2}$-entropic} if there additionally exist two functions $v,w$ verifying
\begin{align}\label{eq:vandwrelations}
\sqrt{\rho}v, \sqrt{\rho}w \in L^\infty(\R^+;L^2(\T_R)), &\indent \sqrt{\rho}u=\sqrt{\rho}v-\sqrt{\rho}w \text{ in $L^\infty([\R^+;L^2(\T_R))$}, \\
\sqrt{\rho_0}u_0=\sqrt{\rho_0}v_0-\sqrt{\rho_0}w_0 \text{ in $L^2(\T_R)$}, &\indent \sqrt{\rho} w=2\nu \nabla \sqrt{\rho}, \indent \rho w=\nu \nabla \rho \text{ in }\mathcal{D}'
\end{align}
such that $(\rho, v, w)$ verify, in the sense of distributions
\begin{equation}\label{eq:vandweqn}
\left\{
\begin{array}{ll}
    \rho_t +{\rm div}(\rho u)=0, \qquad \qquad \indent (t,x) \in \R^+ \times \T_R,\\[3mm]
     (\rho v)_t+{\rm div}(\rho v\otimes u)+\nabla p(\rho)-\nu{\rm div}(\sqrt{\rho}\mathcal{T}(v))+(\nu-\frac{\eps_2}{\nu}){\rm div}(\sqrt{\rho}\mathcal{T}(w))\\[3mm]
\hspace{8cm}+\eps_0u+\eps_1\rho|u|^2u=0, \\[3mm]
     (\rho w)_t+{\rm div}(\rho w\otimes u)+\nu{\rm div}(\sqrt{\rho}\mathcal{T}(u)^t)=0, \\[3mm]
     (\rho(0,x), (\rho v)(0,x), (\rho w)(0,x))=(\rho_0(x),\rho_0v_0(x), \rho_0w_0(x)).
\end{array}
\right.
\end{equation}
More specifically, for any function $\psi \in C^\infty_c(\R^+ \times \T_R)$:
\begin{align}\label{eq:vandweqnintegrated}
&\int_0^\infty \int_{\T_R}(\rho \psi_t+\rho u \cdot \nabla \psi) dx dt =0, \\
&\int_0^\infty \int_{\T_R}\Big[\rho v \psi_t+\big(\rho v \otimes u-(\nu \sqrt{\rho}\mathcal{T}(v)-(\nu-\frac{\eps_2}{\nu}) \sqrt{\rho}\mathcal{T}(w))\big)\cdot \nabla \psi\\
&\qquad\qquad\quad  (-2\rho^{\frac{\gamma}{2}}\nabla (\rho)^{\frac{\gamma}{2}}-\eps_0u-\eps_1\rho|u|^2u)\psi\Big]dx dt=0, \\
&\int_0^\infty \int_{\T_R}\Big[\rho w \psi_t+\big(\rho w \otimes u+\nu \sqrt{\rho}\mathcal{T}(u)^t\cdot \nabla \psi\Big]dxdt=0,
\end{align}
where $\mathcal{T}(v), \mathcal{T}(w)$ are defined by 
\begin{align}\label{eq:vandwdiffusions}
\sqrt{\rho}\mathcal{T}(\theta)=\nabla(\rho \theta)-\rho \theta \otimes \frac{w}{\nu}, \text{ in $\mathcal{D'}$ for $\theta=u,v,w$},
\end{align}
and for any $\overline{\psi} \in C^\infty_c(\T_R)$
\begin{align}\label{eq:vandwinitdata}
&\lim_{t \to 0} \int_{\T_R}\rho(t,x)\overline{\psi}(x)dx=\int_{\T_R}\rho_0(x)\overline{\psi}(x)dx, \\
&\lim_{t \to 0} \int_{\T_R}\rho(t,x)v(t,x)\overline{\psi}(x)dx=\int_{\T_R}\rho_0(x)v_0\overline{\psi}(x)dx, \\
&\lim_{t \to 0} \int_{\T_R}\rho(t,x)w(t,x)\overline{\psi}(x)dx=\int_{\T_R}\rho_0(x)w_0\overline{\psi}(x)dx,
\end{align} 
and also, it obeys the energy estimate for any $t>0$
\begin{align}\label{eq:1/2entropyestimate}
\mathcal{E}_{\frac{1}{2},\eps}(\rho,v,w)(t)+\int_0^t\mathcal{D}_{\frac{1}{2},\eps}(\rho,v,w)(\tau)d \tau \leq \mathcal{E}_{\frac{1}{2},\eps}(\rho,v,w)(0),
\end{align}
where 
\begin{align}\label{eq:halfentropy}
\mathcal{E}_{\frac{1}{2},\eps}(\rho,v,w)&=\int_{\T_R}\big(\rho(\frac{|v|^2+|w|^2}{2})+F(\rho)+\eps_0\nu(\rho-\ln(\rho))+2\eps_2|\nabla \sqrt{\rho}|^2\big)dx, \\
\mathcal{D}_{\frac{1}{2},\eps}(\rho,v,w)&=(\nu-\zeta)\int_{\T_R}|\mathcal{T}(u)|^2dx+\eps_0\int_{\T_R}|u|^2dx+(\eps_1-\zeta)\int_{\T_R}\rho|u|^4dx \\
&+\frac{\nu}{2}\int_{\T_R}\frac{8}{\gamma}|\nabla \rho^{\frac{\gamma}{2}}|^2dx+\frac{\eps_2 \nu}{2}\int_{\T_R}\rho|\nabla^2\ln(\rho)|^2dx,
\end{align}
where 
\begin{align}\label{eq:zetadef}
\zeta=\frac{3\nu \eps_1}{4}.   
\end{align}
\end{definition}
A couple of remarks are in order:
\begin{remark}\label{rem:rem1}
In the case $\eps_2=0$, there is not enough a priori estimates to define $\mathcal{T}(v), \mathcal{T}(w)$ sensibly. In this case, the term $-\nu{\rm div}(\sqrt{\rho}\mathcal{T}(v))+(\nu-\frac{\eps_2}{\nu}){\rm div}(\sqrt{\rho}\mathcal{T}(w))$ in \eqref{eq:vandweqn} should be understood as $-\nu{\rm div}(\sqrt{\rho}\mathcal{T}(u))$.
\end{remark}
\begin{remark}The existence of $\frac{1}{2}$-entropic renormalized weak solutions for finite energy and entropy initial data is shown via a small modification of \cite[Theorem 2.1]{MR3801754}. The inequality \eqref{eq:1/2entropyestimate} may be proven at the level of the Faedo-Galerkin approximation for positive $\eps_0, \eps_1, \eps_2$ and passed into the limit to obtain it for $\eps_0, \eps_1, \eps_2=0$. For convenience, we will show this in \Cref{sec:halfentropyestimatedetails}. 
\end{remark}
\begin{remark}If $(\rho, u)$ verifies \eqref{eq:uequation} and $w$ verifies $\eqref{eq:vandwrelations}$, it is straightforward to see that $(\rho, v,w)$ verify \eqref{eq:vandweqnintegrated}. Indeed, differentiating the continuity equation gives the equation on $\rho w$, and the $\rho v$ equation is obtained by linearity. 
\end{remark}
We also define the tensors $\mathcal{S}(\theta), \mathcal{A}(\theta)$ for $\theta=u,v,w$ as the symmetric and anti-symmetric parts respectively of $\mathcal{T}(\theta)$. 
\par Given two $\frac{1}{2}$-entropy weak solutions $(\rho, v, w)$ and $(r,V,W)$, we may define the relative $\frac{1}{2}$-entropy:
\begin{align}\label{eq:relativehalfentropy}
\mathcal{E}_{\frac{1}{2},\eps}(\rho,v,w|r,V,W)&=\int_{\T_R}\big(\rho\frac{|v-V|^2+|w-W|^2}{2}+F(\rho|r)+\eps_0\nu G(\rho|r) \\
&+2\eps_2|\nabla\sqrt{\rho}-\nabla \sqrt{r}|^2\big)dx,
\end{align}
where 
\begin{align*}
G(x)=x-\ln(x).
\end{align*} 
The quantity $\mathcal{E}_{\frac{1}{2},0}(\rho,v,w)$ and relative quantity $\mathcal{E}_{\frac{1}{2},0}(\rho,v,w|r,V,W)$ are defined in the obvious way (omit the terms with $\eps_0, \eps_2$ in \eqref{eq:halfentropy}, \eqref{eq:relativehalfentropy}):
\begin{align}\label{eq:halfentropy0}
\mathcal{E}_{\frac{1}{2},0}(\rho,v,w)&=\int_{\T_R}\big(\rho\frac{|v|^2+|w|^2}{2}+F(\rho)\big)dx, \\
\mathcal{E}_{\frac{1}{2},0}(\rho,v,w|r,V,W)&=\int_{\T_R}\big(\rho\frac{|v-V|^2+|w-W|^2}{2}+F(\rho|r)\big)dx.
\end{align}
\par The following proposition shows existence of a $\frac{1}{2}$-entropy renormalized weak solution to \eqref{eq:2DQNS_T_R} for initial data with finite energy and entropy which additionally obeys a \textit{relative entropy inequality}. In what follows, by a strong solution to $\eqref{eq:2DNS_T_R}$, we mean a pair of functions $(r,U)$ verifying $r \in C^2(\R^+ \times \T_R), r > 0$,  $U \in C^2(\R^+ \times \T_R)$ which verify \eqref{eq:2DNS_T_R} pointwise, and we may define the quantities $V=U+\nu\nabla \ln(r)$, $W=\nu \nabla \ln(r)$. 
\begin{proposition}\label{prop:1/2entropyrenormalizedexistenceT_Rw/q+drag}
Fix $\eps_0, \eps_1, \eps_2 > 0$. Let $\mathcal{E}_{\eps}(\rho_0, u_0) < \infty$. Then, there exists a $\frac{1}{2}$-entropic renormalized weak solution to \eqref{eq:2DQNS_T_R}. Further, if $(r,V,W)$ is a strong solution to the Navier-Stokes equations \eqref{eq:2DNS_T_R} (without quantum and drag terms), then for any $t>0$, the following relative entropy inequality is verified:
\begin{align}\label{eq:weakstronginequality}
&\mathcal{E}_{\frac{1}{2},0}(\rho,v,w|r,V,W)(t)+\int_0^t\int_{\T_R}\Big[(\nu-\zeta)|\mathcal{T}(u)-\sqrt{\rho}\nabla U|^2\\[3mm]
&\hspace{5cm}\quad+\nu|\frac{2}{\sqrt{\gamma}} \nabla \rho^{\frac{\gamma}{2}}-\sqrt{\rho p'(\rho)}\nabla \ln(r)|^2\Big]dxd\tau \\
&\leq Ce^{\big(||\nabla V||_{L^\infty([0,t] \times \T_R)}+||\nabla W||_{L^\infty([0,t] \times \T_R)}+||\Delta \ln(r)||_{L^\infty([0,t] \times \T_R)}+||\nabla \ln(r)||_{L^\infty([0,t] \times \T_R)}^2\big)t} \\
& \times \Big(\mathcal{E}_{\frac{1}{2},\eps}(\rho,v,w|r,V,W)(0)+E(t)\Big),
\end{align}
where $E(t)$ is given by
\begin{align*}
E(t)=&\int_0^t\int_{\T_R}\big(|\frac{\eps_2}{\nu}\sqrt{\rho}\mathcal{S}(w):\nabla V|+|\eps_0u\cdot V|+|\eps_1\rho|u|^2u\cdot V| \\
&+|2\zeta \sqrt{\rho}\mathcal{A}(v):\mathbb{A}(V)|+|2\zeta \mathcal{S}(u):\mathbb{D}(U)|\big)dxd \tau \\
&+\eps_0\nu\int_{\T_R}|G(r_0)-G'(r_0)(\rho_0-r_0)|dx+2\eps_2 \int_{\T_R}|\nabla \sqrt{r_0}|^2dx \\
&+4\eps_2 \int_{\T_R}|\nabla\sqrt{\rho_0}\cdot \nabla \sqrt{r_0}|dx.
\end{align*}
Further, for any compact subset $K \subset \T_R$, we have an estimate
\begin{align}\label{eq:renormalizationmeasureestimate}
&(\nu-\zeta)||\Xi_\phi||_{\mathcal{M}([0,t] \times K)} \leq C||\phi''||_{L^\infty}\Bigg( (\nu-\zeta)\int_0^t\int_K \rho|\nabla (U)|^2dxd\tau \\
&+ Ce^{\big(||\nabla V||_{L^\infty([0,t] \times \T_R)}+||\nabla W||_{L^\infty([0,t] \times \T_R)}+||\Delta \ln(r)||_{L^\infty([0,t] \times \T_R)}+||\nabla \ln(r)||_{L^\infty([0,t] \times \T_R)}^2\big)t} \\
& \times \Big(\mathcal{E}_{\frac{1}{2},\eps}(\rho,v,w|r,V,W)(0)+E(t)\Big) \\
&+||\mathcal{T}(u)||_{L^2([0,t] \times \T_R)}||\frac{\eps_2}{\nu}\mathcal{T}(w)||_{L^2([0,t] \times \T_R)}\Bigg).
\end{align}
The constant $C$ does not depend on $R$.
\end{proposition}
By passing into the limit $\eps_0, \eps_1, \eps_2 \to 0+$, we obtain the following corollary:
\begin{proposition}\label{prop:1/2entropyrenormalizedexistenceT_R}
Let $\eps_0, \eps_1, \eps_2=0$. Let $\mathcal{E}_0(\rho_0, u_0) < \infty$. Then, there exists a $\frac{1}{2}$-entropic renormalized weak solution to \eqref{eq:2DNS_T_R}. Further, if $(r,V,W)$ is a strong solution to the Navier-Stokes equations \eqref{eq:2DNS_T_R}, then the following relative entropy inequality is verified:
\begin{align}\label{eq:weakstrongestimateonTR}
&\mathcal{E}_{\frac{1}{2},0}(\rho,v,w|r,V,W)(t)+\nu \int_0^t\int_{\T_R}\big(|\mathcal{T}(u)-\sqrt{\rho}\nabla U|^2+|\frac{2}{\sqrt{\gamma}}\nabla \rho^{\frac{\gamma}{2}}-\sqrt{\rho p'(\rho)}\nabla \ln(r)|^2\big)dxd\tau \\
&\leq Ce^{\big(||\nabla V||_{L^\infty([0,t] \times \T_R)}+||\nabla W||_{L^\infty([0,t] \times \T_R)}+||\Delta \ln(r))||_{L^\infty([0,t] \times \T_R)}+||\nabla \ln(r)||_{L^\infty([0,t] \times \T_R)}^2\big)t} \\
&\times \Big(\mathcal{E}_{\frac{1}{2},0}(\rho,v,w|r,V,W)(0)\Big).
\end{align}
Further, for any compact subset $K \subset \T_R$, we have an estimate
\begin{align*}
&\nu||\Xi_\phi||_{\mathcal{M}([0,t] \times K)} \leq C||\phi''||_{L^\infty}\Bigg( \nu \int_0^t\int_K \rho|\nabla (U)|^2dxd\tau \\
&+ Ce^{\big(||\nabla V||_{L^\infty([0,t] \times \T_R)}+||\nabla W||_{L^\infty([0,t] \times \T_R)}+||\Delta \ln(r)||_{L^\infty([0,t] \times \T_R)}+||\nabla \ln(r)||_{L^\infty([0,t] \times \T_R)}^2\big)t} \\
&\times \mathcal{E}_{\frac{1}{2},0}(\rho,v,w|r,V,W)(0)\Bigg).
\end{align*}
The constant $C$ does not depend on $R$.
\end{proposition}
\Cref{prop:1/2entropyrenormalizedexistenceT_R} is a simple corollary of \Cref{prop:1/2entropyrenormalizedexistenceT_Rw/q+drag} and \cite[Theorem 2.1]{MR3801754}. Indeed, for any initial data $(\rho_0, u_0)$ with $\mathcal{E}_0(\rho_0, u_0) < \infty$, take positive sequences $\eps_{0,n}, \eps_{1,n}, \eps_{2,n}$ all converging to $0$. A slight regularization of the initial data gives a sequence of initial data $(\rho_0^{\eps_n}, u_0^{\eps_n})$ such that $\mathcal{E}_{\eps_n}(\rho_0, u_0)$ is uniformly bounded and converges to $\mathcal{E}_0$. \Cref{prop:1/2entropyrenormalizedexistenceT_Rw/q+drag} gives existence of $\frac{1}{2}$-entropic renormalized weak solutions verifying the relative entropy inequalities for $\eps_{0,n}, \eps_{1,n}, \eps_{2,n} >0$ with initial data $(\rho_0^{\eps_n}, u_0^{\eps_n})$. Finally, \cite[Theorem 2.1, (4)]{MR3801754} gives at the limit a $\frac{1}{2}$-entropic renormalized weak solution to \eqref{eq:2DNS_T_R} verifying the relative entropy inequalities. 

\subsubsection{Proof of \Cref{prop:1/2entropyrenormalizedexistenceT_Rw/q+drag}}
Let $\eps_0, \eps_1, \eps_2 >0$. Let $(\rho_0, u_0)$ be such that $\mathcal{E}_{\eps}(\rho_0, u_0) < \infty$. Let $\sqrt{\rho_0}w_0=2\nu \nabla \sqrt{\rho_0}$. \cite[Theorem 1]{MR3490496} shows existence of a weak solution to $\eqref{eq:2DQNS_T_R}$. A small modification of their arguments can be done to show that the inequality \eqref{eq:1/2entropyestimate} holds at the level of the Faedo-Galerkin approximation and the a priori estimates \eqref{eq:aprioriestimatesforquantum} allow to pass it into the limit (see \Cref{sec:halfentropyestimatedetails}). The a priori estimates also allow to pass into the limit of the Faedo-Galerkin in \eqref{eq:vandwdiffusions} and \eqref{eq:vandwrelations}. It is straightforward to obtain the equation \eqref{eq:vandweqnintegrated} from \eqref{eq:uequation} and \eqref{eq:vandwrelations} (see \Cref{rem:rem1}). The attainment of the initial data \eqref{eq:vandwinitdata} follows from \eqref{eq:vandwrelations} and \eqref{eq:uinitdata}. Finally, \cite[Theorem 2.1, (3)]{MR3801754} shows that this is actually a $\frac{1}{2}$-entropy renormalized weak solution.

\par It remains to prove the relative entropy inequality \eqref{eq:weakstronginequality}. The argument mostly follows the basic strategy in \cite{MR3734208}. However, we choose to do the computation at the level of the system \eqref{eq:2DQNS_T_R} to ensure that the weak solution has enough regularity to justify all the computations. This gives rise to the error terms $E(t)$. However, these will all vanish at the limit $\eps_0, \eps_1, \eps_2 \to 0$. See the papers \cite{MR4198204}, \cite{MR4801843}, \cite{MR3961293} for similar analyses in the situation when $\eps_2 > \nu$. However, this is not the regime we are interested in here.  
\newline \textbf{The relative entropy estimate.} Let us first estimate
\begin{align}\label{eq:relent1}
&\mathcal{E}_{\frac{1}{2},0}(\rho,v,w|r,V,W)(t)-\mathcal{E}_{\frac{1}{2},\eps}(\rho,v,w|r,V,W)(0) \\
&=\Big[\frac{1}{2}\int_{\T_R}\rho(|v|^2+|w|^2)+F(\rho)\Big](t)\\
&\quad -\Big[\frac{1}{2}\int_{\T_R}\rho(|v|^2+|w|^2)+F(\rho)+\frac{\eps_0 \nu}{2}G(\rho)+2\eps_2|\nabla \sqrt{\rho}|^2\Big](0) \\
&\quad +\Big[\frac{1}{2}\int_{\T_R}\rho(|V|^2+|W|^2)-\int_{\T_R}(F(r)+F'(r)(\rho-r))-\int_{\T_R}\rho(w\cdot W+v\cdot V)\Big](t) \\
&\quad -\Big[\frac{1}{2}\int_{\T_R}\rho(|V|^2+|W|^2)-\int_{\T_R}(F(r)+F'(r)(\rho-r))-\int_{\T_R}\rho(w\cdot W+v\cdot V)\Big](0) \\
&\quad +\Big[\frac{\eps_0 \nu}{2}\int_{\T_R}(G(r)-G'(r)(\rho-r))+2\eps_2 \int_{\T_R}(|\nabla \sqrt{r}|^2+2 \nabla \sqrt{\rho}\cdot \sqrt{r})\big](0).
\end{align}
Next, we compute
\begin{align}\label{eq:relent2}
&(\nu-\zeta)\int_0^t\int_{\T_R}|\mathcal{A}(v)-\sqrt{\rho}\mathbb{A}(V)|^2+(\nu-\zeta)\int_0^t\int_{\T_R}|\mathcal{S}(U)-\sqrt{\rho}\mathbb{D}(U)|^2 \\
&+\nu \int_0^t\int_{\T_R}(p'(\rho)2\nabla \sqrt{\rho}-p'(r)\nabla \ln(r)\sqrt{\rho})(2 \nabla \sqrt{\rho}-\sqrt{\rho}\nabla \ln(r)) \\
&=(\nu-\zeta)\int_0^t\int_{\T_R}|\mathcal{T}(u)|^2+\frac{\nu}{2}\int_{\T_R}\frac{8}{\gamma}|\nabla \rho^{\frac{\gamma}{2}}|^2 \\
&+(\zeta-\nu)\int_0^t\int_{\T_R}\sqrt{\rho}\mathbb{A}(V):(\sqrt{\rho}\mathbb{A}(V)-\mathcal{A}(v))+(\zeta-\nu)\int_0^t\int_{\T_R}\sqrt{\rho}\mathbb{D}(U):(\sqrt{\rho}\mathbb{D}(V)-\mathcal{S}(u)) \\
&+\nu\int_0^t\int_{\T_R}p'(r)\nabla \ln(r)\sqrt{\rho}(\sqrt{\rho}\nabla \ln(r)-2\nabla \sqrt{\rho}) \\
&-\Big[(\nu-\zeta)\int_0^t\int_{\T_R}\sqrt{\rho}\mathcal{A}(v):\mathbb{A}(V)+\int_0^t\int_{\T_R}\sqrt{\rho}\mathcal{S}(U):\mathbb{D}(U)\\
&\quad +\nu\int_0^t\int_{\T_R}2\nabla \sqrt{\rho}\sqrt{\rho}p'(\rho)\nabla \ln(r)\Big].
\end{align}
Now, testing the equations for $\rho v$ and $\rho w$ (\eqref{eq:vandweqn}) by $V$ and $W$ respectively, we obtain
\begin{align}\label{eq:relent3}
&\int_{\T_R}\rho(v \cdot V+w\cdot W)(t)-\int_{\T_R}\rho(v \cdot V+w\cdot W)(t)(0) \\
&-\int_0^t\int_{\T_R}\partial_tV\cdot \rho v+\partial_tW\cdot \rho w-\int_0^t\int_{\T_R}\rho u\cdot (\nabla V \cdot v+\nabla W \cdot w) \\
&-\int_0^t\int_{\T_R}p(\rho){\rm div}(U)+\int_0^t\int_{\T_R}p'(\rho)\nabla \rho \cdot W \\
&+\nu \int_0^t\int_{\T_R}\sqrt{\rho}\mathcal{S}(u):\mathbb{D}(U)+\nu \int_0^t\int_{\T_R}\sqrt{\rho}\mathcal{A}(v):\mathbb{A}(U)\\
&+\nu \int_0^t\int_{\T_R}\sqrt{\rho}(\mathcal{A}(v):\mathbb{A}(W)-\mathcal{A}(w):\mathbb{A}(V)) \\
&+\frac{\eps_2}{\nu}\int_0^t\int_{\T_R}\sqrt{\rho}\mathcal{S}(w):\nabla V+\eps_0\int_0^t\int_{\T_R}u\cdot V+\eps_1\int_0^t\int_{\T_R}\rho |u|^2u \cdot V=0.
\end{align}
Then, testing the mass equation by $\frac{|V|^2+|W|^2}{2}$ gives
\begin{align}\label{eq:relent4}
&\frac{1}{2}\int_{\T_R}\rho(|V|^2+|W|^2)(t)-\frac{1}{2}\int_{\T_R}\rho(|V|^2+|W|^2)(0) \\
&=\int_0^t\int_{\T_R}\rho(V\cdot \partial_tV+W\cdot \partial_tW)-\int_0^t\int_{\T_R}\rho u\cdot(\nabla V \cdot v+\nabla W\cdot w).
\end{align}
Using $\partial_t(F(r)-F'(r)r)=-\partial_tF'(r)r$:
\begin{align}\label{eq:relent5}
-\int_{\T_R}(F(r)-F'(r)r)(t)+\int_{\T_R}(F(r)-F'(r)r)(0)=\int_0^t\int_{\T_R}r\partial_tF'(r).
\end{align}
Next, multiplying the mass equation by $F'(r)$ and integrating in space and time
\begin{align}\label{eq:relent6}
-\int_{\T_R}F'(r)\rho(t)+\int_{\T_R}F'(r)\rho(0)=-\int_0^t\int_{\T_R}\rho \partial_tF'(r)-\int_0^t\int_{\T_R}\rho u \cdot \nabla F'(r). 
\end{align}
Finally, we have the identity
\begin{align}\label{eq:relent7}
0=\int_0^t\int_{\T_R}p(r){\rm div}(U)+rU\cdot\nabla F'(r).
\end{align}
Using \eqref{eq:relent1}-\eqref{eq:relent7} gives the relative entropy inequality:
\begin{align}\label{eq:baserelentinequality}
&\mathcal{E}_{\frac{1}{2},0}(\rho,v,w|r,V,W)(t)-\mathcal{E}_{\frac{1}{2},\eps}(\rho,v,w|r,V,W)(0) \\
&+(\nu-\zeta)\int_0^t\int_{\T_R}|\mathcal{A}(v)-\sqrt{\rho}\mathbb{A}(V)|^2+(\nu-\zeta)\int_0^t\int_{\T_R}|\mathcal{S}(U)-\sqrt{\rho}\mathbb{D}(U)|^2 \\
&+\nu \int_0^t\int_{\T_R}(p'(\rho)2\nabla \sqrt{\rho}-p'(r)\nabla \ln(r)\sqrt{\rho})(2 \nabla \sqrt{\rho}-\sqrt{\rho}\nabla \ln(r)) \\
&\leq \int_0^t\int_{\T_R}\rho(\partial_tW\cdot (W-w)+\partial_tV \cdot(V-v))+\int_0^t\int_{\T_R}\rho u \cdot(\nabla W\cdot (W-w)+\nabla V\cdot(V-v)) \\
&+\int_0^t\int_{\T_R}\partial_tF'(r)(r-\rho)-\int_0^t\int_{\T_R}\nabla F'(r)\cdot (\rho u-rU) \\
&+\int_0^t\int_{\T_R}(p(r)-p(\rho)){\rm div}(U)-\int_0^t\int_{\T_R}p'(\rho)\nabla \rho \cdot(\nu \nabla \ln(r)-W) \\
&+(\nu-\zeta)\int_0^t\int_{\T_R}\sqrt{\rho}\mathbb{A}(V):(\sqrt{\rho}\mathbb{A}(V)-\mathcal{A}(v))+(\nu-\zeta)\int_0^t\int_{\T_R}\sqrt{\rho}\mathbb{D}(U):(\sqrt{\rho}\mathbb{D}(U)-\mathcal{S}(u)) \\
&+\nu \int_0^t\int_{\T_R}\sqrt{\rho}(\mathbb{A}(W):(\mathcal{A}(v)-\sqrt{\rho}\mathbb{A}(V))-(\mathcal{A}(w)-\sqrt{\rho}\mathbb{A}(W)):\mathbb{A}(V)) \\
&+\nu\int_0^t\int_{\T_R}p'(r)\nabla \ln(r)\sqrt{\rho}(\sqrt{\rho}\nabla \ln(r)-2\nabla \sqrt{\rho}) \\
&+\zeta \int_0^t\int_{\T_R}\sqrt{\rho}\mathcal{S}(u):\mathbb{D}(U)+\zeta \int_0^t\int_{\T_R}\sqrt{\rho}\mathcal{A}(v):\mathbb{A}(U) \\
&+\frac{\eps_2}{\nu}\int_0^t\int_{\T_R}\sqrt{\rho}\mathcal{S}(w):\nabla V+\eps_0\int_0^t\int_{\T_R}u\cdot V+\eps_1\int_0^t\int_{\T_R}\rho |u|^2u \cdot V \\
&+\big[\frac{\eps_0 \nu}{2}\int_{\T_R}(G(r)-G'(r)(\rho-r))+2\eps_2 \int_{\T_R}(|\nabla \sqrt{r}|^2+2 \nabla \sqrt{\rho}\cdot \sqrt{r})\big](0).
\end{align}
Now, we can use the fact that $(r,V,W)$ is a strong solution to \eqref{eq:vandweqn} (with $\eps_0=\eps_1=\eps_2=0$) to further simplify \eqref{eq:baserelentinequality}. Testing the equations for $r V$ and $rW$ by $\frac{\rho (V-v)}{r}$ and $\frac{\rho (W-w)}{r}$ respectively and summing the result, we obtain 
\begin{align}\label{eq:relent8}
&\int_0^t\int_{\T_R}\rho(\partial_tV \cdot (V-v)+\partial_tW\cdot(W-w)) \\
&+\nu \int_{0}^t\int_{\T_R}\sqrt{\rho}\mathbb{D}(U):(\sqrt{\rho}\mathbb{D}(U)-\mathcal{S}(u))+\nu \int_{0}^t\int_{\T_R}\sqrt{\rho}\mathbb{A}(U):(\sqrt{\rho}\mathbb{A}(U)-\mathcal{A}(u))\\
&+\nu \int_0^t\int_{\T_R}\sqrt{\rho}(\mathbb{A}(W):(\mathcal{A}(v)-\sqrt{\rho}\mathbb{A}(V))-(\mathcal{A}(w)-\sqrt{\rho}\mathbb{A}(W)):\mathbb{A}(V)) \\
&=-\int_0^t\int_{\T_R}\rho U \cdot(\nabla V \cdot(V-v)+\nabla W\cdot(W-w)) \\
&-\int_0^t\int_{\T_R}\frac{\rho}{r}\nabla p(r)\cdot (V-v)-\nu \int_0^t\int_{\T_R}r\mathbb{D}(V):\big(\frac{\nabla \sqrt{\rho}}{r}+\nabla(\frac{\sqrt{\rho}}{r})\big)\sqrt{\rho}(V-v) \\
&-\nu \int_0^t\int_{\T_R}r\mathbb{A}(V):\big(\frac{\nabla \sqrt{\rho}}{r}+\nabla(\frac{\sqrt{\rho}}{r})\big)\sqrt{\rho}(V-v)\\
&+\nu \int_0^t\int_{\T_R}r\nabla W:\big(\frac{\nabla \sqrt{\rho}}{r}+\nabla(\frac{\sqrt{\rho}}{r})\big)\sqrt{\rho}(V-v) \\
&-\nu \int_0^t\int_{\T_R}r\nabla W:\big(\frac{\nabla \sqrt{\rho}}{r}+\nabla(\frac{\sqrt{\rho}}{r})\big)\sqrt{\rho}(W-w)\\
&+\nu \int_0^t\int_{\T_R}r(\nabla V)^t:\big(\frac{\nabla \sqrt{\rho}}{r}+\nabla(\frac{\sqrt{\rho}}{r})\big)\sqrt{\rho}(W-w).
\end{align}
We will use this alongside the other following simplifications to simplify \eqref{eq:baserelentinequality}. We have
\begin{align}\label{eq:relent9}
-\int_0^t\int_{\T_R}\frac{\rho}{r}\nabla p(r)\cdot (V-v)=-\int_0^t\int_{\T_R}\rho \nabla F'(r)\cdot U-\int_0^t\int_{\T_R}\frac{\rho}{r}p'(r)\nabla r \cdot (W-w).
\end{align}
Further, we have the identities
\begin{align}\label{eq:relent10}
p'(\rho)\nabla \rho\cdot(\nu \nabla \ln(r)-W)=p'(r)\nabla \ln(r)\sqrt{\rho}(\sqrt{\rho}\nabla \ln(r)-2\nabla \sqrt{\rho})-\frac{\rho}{r}p'(r)\nabla r \cdot (W-w)=0,
\end{align}
which is true when $W=\nu \nabla \ln(r)$ and $\rho w=\nu \nabla \rho$. Consider also the identity
\begin{align}\label{eq:relent11}
(\partial_tF'(r)+U\cdot \nabla F'(r))(r-\rho)=-{\rm div}(U)(r-\rho)p'(r).
\end{align}
Plugging \eqref{eq:relent8}-\eqref{eq:relent11} into \eqref{eq:baserelentinequality}, we obtain
\begin{align}\label{eq:relent12}
&\mathcal{E}_{\frac{1}{2},0}(\rho,v,w|r,V,W)(t)-\mathcal{E}_{\frac{1}{2},\eps}(\rho,v,w|r,V,W)(0) \\
&+(\nu-\zeta)\int_0^t\int_{\T_R}|\mathcal{A}(v)-\sqrt{\rho}\mathbb{A}(V)|^2+(\nu-\zeta)\int_0^t\int_{\T_R}|\mathcal{S}(U)-\sqrt{\rho}\mathbb{D}(U)|^2 \\
&+\nu \int_0^t\int_{\T_R}(p'(\rho)2\nabla \sqrt{\rho}-p'(r)\nabla \ln(r)\sqrt{\rho})(2 \nabla \sqrt{\rho}-\sqrt{\rho}\nabla \ln(r)) \\
&\leq \int_0^t\int_{\T_R}\rho (u-U) \cdot(\nabla W\cdot (W-w)+\nabla V\cdot(V-v)) \\
&+\int_0^t\int_{\T_R}(p(r)-p(\rho)+p'(r)(\rho-r)){\rm div}(U) \\
&-\nu \int_0^t\int_{\T_R}r\mathbb{D}(V):\big(\frac{\nabla \sqrt{\rho}}{r}+\nabla(\frac{\sqrt{\rho}}{r})\big)\sqrt{\rho}(V-v)\\
&-\nu \int_0^t\int_{\T_R}r\mathbb{A}(V):\big(\frac{\nabla \sqrt{\rho}}{r}+\nabla(\frac{\sqrt{\rho}}{r})\big)\sqrt{\rho}(V-v) \\
&+\nu \int_0^t\int_{\T_R}r\nabla W:\big(\frac{\nabla \sqrt{\rho}}{r}+\nabla(\frac{\sqrt{\rho}}{r})\big)\sqrt{\rho}(V-v)\\
&-\nu \int_0^t\int_{\T_R}r\nabla W:\big(\frac{\nabla \sqrt{\rho}}{r}+\nabla(\frac{\sqrt{\rho}}{r})\big)\sqrt{\rho}(W-w) \\
&+\nu \int_0^t\int_{\T_R}r(\nabla V)^t:\big(\frac{\nabla \sqrt{\rho}}{r}+\nabla(\frac{\sqrt{\rho}}{r})\big)\sqrt{\rho}(W-w) \\
&+2\zeta \int_0^t\int_{\T_R}\sqrt{\rho}\mathcal{S}(u):\mathbb{D}(U)+2\zeta \int_0^t\int_{\T_R}\sqrt{\rho}\mathcal{A}(v):\mathbb{A}(U) \\
&+\frac{\eps_2}{\nu}\int_0^t\int_{\T_R}\sqrt{\rho}\mathcal{S}(w):\nabla V+\eps_0\int_0^t\int_{\T_R}u\cdot V+\eps_1\int_0^t\int_{\T_R}\rho |u|^2u \cdot V \\
&+\big[\frac{\eps_0 \nu}{2}\int_{\T_R}(G(r)-G'(r)(\rho-r))+2\eps_2 \int_{\T_R}(|\nabla \sqrt{r}|^2+2 \nabla \sqrt{\rho}\cdot \sqrt{r})\big](0).
\end{align}
Let us simplify some of the terms on the right-hand side. Firstly, note
\begin{align}\label{eq:relent13}
&|\int_{\T_R}\rho (u-U) \cdot(\nabla W\cdot (W-w)+\nabla V\cdot(V-v)) \\
&+\int_{\T_R}(p(r)-p(\rho)+p'(r)(\rho-r)){\rm div}(U)| \\
&\leq c(||\nabla V||_{L^\infty([0,t] \times \T_R)}+||\nabla W||_{L^\infty([0,t] \times \T_R)})\mathcal{E}_{\frac{1}{2},0}(\rho,v,w|r,V,W)(t).
\end{align}
Further, using $\sqrt{\rho}w=2\nu\nabla \sqrt{\rho}$, we have
\begin{align*}
\big(\frac{\nabla \sqrt{\rho}}{r}+\nabla(\frac{\sqrt{\rho}}{r})\big)\sqrt{\rho}=\frac{1}{\nu}\frac{\rho}{r}(w-W).
\end{align*}
This allows to control
\begin{align}\label{eq:relent14}
&|-\nu \int_{\T_R}r\mathbb{D}(V):\big(\frac{\nabla \sqrt{\rho}}{r}+\nabla(\frac{\sqrt{\rho}}{r})\big)\sqrt{\rho}(V-v)-\nu \int_{\T_R}r\mathbb{A}(V):\big(\frac{\nabla \sqrt{\rho}}{r}+ \\
&\nabla(\frac{\sqrt{\rho}}{r})\big)\sqrt{\rho}(V-v)+\nu \int_{\T_R}r\nabla W:\big(\frac{\nabla \sqrt{\rho}}{r}+\nabla(\frac{\sqrt{\rho}}{r})\big)\sqrt{\rho}(V-v)-\nu \int_{\T_R}r\nabla W:\big(\frac{\nabla \sqrt{\rho}}{r} \\
&+\nabla(\frac{\sqrt{\rho}}{r})\big)\sqrt{\rho}(W-w)+\nu \int_{\T_R}r(\nabla V)^t:\big(\frac{\nabla \sqrt{\rho}}{r}+\nabla(\frac{\sqrt{\rho}}{r})\big)\sqrt{\rho}(W-w)| \\
&\leq C(||\nabla V||_{L^\infty([0,t] \times \T_R)}+||\nabla W||_{L^\infty([0,t] \times \T_R)})\mathcal{E}_{\frac{1}{2},\eps}(\rho,v,w|r,V,W)(t).
\end{align}
It remains to control the term on the left-hand side
\begin{align*}
I=\nu \int_0^t\int_{\T_R}(p'(\rho)2\nabla \sqrt{\rho}-p'(r)\nabla \ln(r)\sqrt{\rho})(2 \nabla \sqrt{\rho}-\sqrt{\rho}\nabla \ln(r)).
\end{align*}
By adding and subtracting $p'(\rho)\sqrt{\rho}\nabla \ln(r)$, note that we can write it as 
\begin{align*}
I&=\int_0^t\int_{\T_R}|\frac{2}{\sqrt{\gamma}} \nabla \rho^{\frac{\gamma}{2}}-\sqrt{\rho}p'(\rho)\nabla \ln(r)|^2dxd\tau \\
&+\int_0^t\int_{\T_R}\sqrt{\rho}\nabla \ln(r)(p'(\rho)-p'(r))(2\nabla \sqrt{\rho}-\sqrt{\rho}\nabla \ln(r))dxd\tau=I_1+I_2.
\end{align*}
To control the term $I_2$, we use the identity
\begin{align*}
\sqrt{\rho}\nabla \ln(r)(p'(\rho)-p'(r))(2\nabla \sqrt{\rho}-\sqrt{\rho}\nabla \ln(r))&=\nabla \big(p(\rho)-p(r)-p'(r)(\rho-r)\big)\cdot \nabla \ln(r) \\
&-\big(\rho(p'(\rho)-p'(r))-p''(r)(\rho-r)r\big)|\nabla \ln(r)|^2.
\end{align*}
So
\begin{align}\label{eq:relent15}
I_2 \leq \int_0^t\int_{\T_R}|p(\rho|r)||\Delta \ln(r)|+|\big(\rho(p'(\rho)-p'(r))-p''(r)(\rho-r)r\big)||\nabla \ln(r)|^2.
\end{align}
Finally, we note 
\begin{align*}
\rho(p'(\rho)-p'(r))-p''(r)(\rho-r)r=\gamma (\gamma-1)F(\rho|r).
\end{align*}
So, plugging \eqref{eq:relent13}-\eqref{eq:relent15} into \eqref{eq:relent12}, we obtain \eqref{eq:weakstronginequality} via Gronwall's inequality. Finally, the estimate \eqref{eq:renormalizationmeasureestimate} follows straightforwardly from \eqref{eq:weakstronginequality} and the representation of the renormalization measure (see \cite[page 15]{MR3801754})
\begin{align*}
\Xi_\phi=\phi''(u)\mathcal{T}(u)(2\nu \mathcal{S}(u)+\frac{\eps_2}{\nu}\mathcal{T}(w)),
\end{align*}
where we note that $\frac{\eps_2}{\nu}\mathcal{T}(w)=2\sqrt{\kappa}\mathbb{S}_\kappa$ in the terminology of \cite{MR3801754}.

\subsection{%Definition and 
Existence of weak solutions to 2D Navier-Stokes equations on $\R \times \T$}
\subsubsection{Periodic initial data and solutions}\label{sec:periodicsols} 
\par Now, we define the suitable initial data on $\R \times \T$. Recall from \Cref{sec:mainresults} that we let $\tilde{\rho}(x_1)$ and $\tilde{u}(x_1)$ be two smooth monotone functions satisfying
\begin{align*}
&\tilde{\rho}(x_1)=\rho_\pm \text{ for } \pm x_1 \geq 1, \tilde{\rho}(x_1) > 0, \ \forall \ x_1 \in \R, \\
&\tilde{u}(x_1)=\begin{bmatrix}
\tilde{u_1}(x_1) \\
\tilde{u_2}(x_1) 
\end{bmatrix}=u_\pm \text{ for } \pm x_1 \geq 1.
\end{align*}
We have also defined
\begin{align*}
\tilde{v}&=\tilde{u}+\nu \nabla \ln(\tilde{\rho}), \\
\tilde{w}&=\nu \nabla \ln(\tilde{\rho}).
\end{align*}
%Define $\tilde{S}(x_1):=(\tilde{\rho}(x_1), \tilde{u}(x_1))$. Without confusions, we may also refer to $\tilde{S}$ as a function from $\R \times \T$ to $\R^2$ that is constant in the $y$-direction.
Further recall that for functions $(\rho, v, w), (r,V,W):\R \times \T \to (\R_{\geq 0} \times \R^2 \times \R^2)$, we have defined (see \eqref{eq:halfrelentonchannel}) the relative $\frac{1}{2}$-entropy on the channel $\R \times \T$:
\begin{align*}
\mathcal{E}_{\frac{1}{2}}(\rho,v,w|r,V,W)=\int_{\R \times \T}\big(\rho(\frac{|v-V|^2+|w-W|^2}{2})+F(\rho|r)\big)dx.
\end{align*}
Note that compared to the definition of $\mathcal{E}_{\frac{1}{2},0}(\rho,v,w|r,V,W)$ (\eqref{eq:halfentropy0}) we have just changed the domain of integration to have \eqref{eq:halfrelentonchannel}. 
Now, let us define the precise notion of a weak solution to \eqref{eq:2DNS} which verifies the correct far field at $\pm \infty$.
\begin{definition}\label{def:defofweak}
We say that $(\rho, v, w)$ is a \textit{$\frac{1}{2}$-entropy} weak solution to $\eqref{eq:2DNS}$ if for any $T > 0$, it verifies the a priori estimate 
\begin{align}\label{eq:correctff}
\mathcal{E}_{\frac{1}{2}}(\rho,v,w|\tilde{\rho},\tilde{v},\tilde{w})(t) \leq C(T), \indent \text{for all } 0 \leq t \leq T,
\end{align}
and for any function $\psi \in C^\infty_c(\R^+ \times \R \times \T)$:
\begin{align}\label{eq:uequationonchannel}
&\int_0^\infty \int_{\R \times \T}(\rho \psi_t+\rho u \cdot \nabla \psi) dx dt =0, \\
&\int_0^\infty \int_{\R \times \T}\Big[\rho u \psi_t+\big(\rho u \otimes u-(2\nu \sqrt{\rho}\mathcal{S}(u)\big)\cdot \nabla \psi\\
&\qquad\qquad\quad  +(-2\rho^{\frac{\gamma}{2}}\nabla (\rho)^{\frac{\gamma}{2}})\psi\Big]dx dt=0,
\end{align}
where $\mathcal{S}(u)$ is the symmetric part of $\mathcal{T}(u)$ which verifies $\eqref{eq:T(u)def}$, and for any $\overline{\psi} \in C^\infty_c(\R \times \T)$
\begin{align}\label{eq:uequationinitdata}
&\lim_{t \to 0} \int_{\R \times \T}\rho(t,x)\overline{\psi}(x)dx=\int_{\R \times \T}\rho_0(x)\overline{\psi}(x)dx, \\
&\lim_{t \to 0} \int_{\R \times \T}\rho(t,x)u(t,x)\overline{\psi}(x)dx=\int_{\R \times \T}m_0(x)\overline{\psi}(x)dx,
\end{align}
and further, 
\begin{align}\label{eq:vandwrelationsonchannel}
\sqrt{\rho}u&=\sqrt{\rho}v-\sqrt{\rho}w \text{ in $L^2_{\text{loc}}(\R^+ \times \R \times \T)$}, \\
\sqrt{\rho_0}u_0&=\sqrt{\rho_0}v_0-\sqrt{\rho_0}w_0 \text{ in $L^2_{\text{loc}}(\R \times \T)$}, \\
\sqrt{\rho} w&=2\nu \nabla \sqrt{\rho}, \indent \rho w=\nu \nabla \rho \text{ in }\mathcal{D}',
\end{align}
and further, for any function $\psi \in C^\infty_c(\R^+ \times \R \times \T)$, $(\rho, v, w)$ verifies the equations for the augmented system in the sense of distributions:
\begin{align}\label{eq:vandweqnintegratedchannel}
&\int_0^\infty \int_{\R \times \T}(\rho \psi_t+\rho u \cdot \nabla \psi) dx dt =0, \\
&\int_0^\infty \int_{\R \times \T}\Big[\rho v \psi_t+\big(\rho v \otimes u-(\nu \sqrt{\rho}\mathcal{T}(u))\big)\cdot \nabla \psi\\
&\qquad\qquad\quad  (-2\rho^{\frac{\gamma}{2}}\nabla (\rho)^{\frac{\gamma}{2}})\psi\Big]dx dt=0, \\
&\int_0^\infty \int_{\R \times \T}\Big[\rho w \psi_t+\big(\rho w \otimes u+\nu \sqrt{\rho}\mathcal{T}(u)^t\cdot \nabla \psi\Big]dxdt=0,
\end{align}
and for any $\overline{\psi} \in C^\infty_c(\R \times \T)$
\begin{align}\label{eq:vandwinitdatachannel}
&\lim_{t \to 0} \int_{\R \times \T}\rho(t,x)\overline{\psi}(x)dx=\int_{\R \times \T}\rho_0(x)\overline{\psi}(x)dx, \\
&\lim_{t \to 0} \int_{\R \times \T}\rho(t,x)v(t,x)\overline{\psi}(x)dx=\int_{\R \times \T}\rho_0(x)v_0(x) \overline{\psi}(x)dx, \\
&\lim_{t \to 0} \int_{\R \times \T}\rho(t,x)w(t,x)\overline{\psi}(x)dx=\int_{\R \times \T}\rho_0(x)w_0(x) \overline{\psi}(x)dx.
\end{align} 
\end{definition}
\subsection{Proof of \Cref{thm:intro1}}
\par Given suitable initial data $(\rho_0, v_0, w_0)$, the strategy to construct the solution to \eqref{eq:2DNS} is as follows. First, we cutoff and periodize the functions $(\tilde{\rho}, \tilde{v}, \tilde{w})$, $(\bar \rho_0, \bar v_0, \bar w_0)$, and $(\rho_0, v_0, w_0)$ to be defined on $\T_R$. Then, there exists a $\frac{1}{2}$-entropy renormalized weak solution $(\rho_R, v_R, w_R)(t,x)$ to \eqref{eq:2DNS_T_R} (by \Cref{prop:1/2entropyrenormalizedexistenceT_R}) and a strong planar solution $(\bar\rho_R, \bar v_R, \bar w_R)(t,x_1)$ to \eqref{eq:2DNS_T_R} (by \Cref{prop:1dexist+bds}). For fixed $T>0$, we prove uniform-in-$R$ bounds on $||(\bar\rho_R, \bar u_R)_{x_1}||_{L^\infty([0,T];H^2(\T_R))}$ and $\inf_{(t,x) \in [0,T] \times \T_R}\overline{\rho}_R(t,x)$. Then, by \Cref{prop:1/2entropyrenormalizedexistenceT_R} and \Cref{def:suitableinitial}, $\mathcal{E}_{\frac{1}{2}}((\rho_R,v_R,w_R)|(\bar\rho_R,\bar v_R, \bar w_R))(t)$ is uniformly bounded in $R$. Then, for $T$ finite, in \Cref{sec:Rlimit}, we show that this implies we may take the limit $R \to \infty$ to obtain a weak solution to \eqref{eq:2DNS}.
\begin{notation}
In this section, we again denote by $C$ a constant which may change line to line, and we emphasize that it does not depend on $R$.
\end{notation}
\par By the definition of suitable initial data (\Cref{def:suitableinitial}), there exist constants $\delta_1, \delta_2, \delta_3 > 0$ so that
\begin{align}\label{eq:initialbds}
&\mathcal{E}_{\frac{1}{2}}(\rho_0,v_0, w_0|\tilde{\rho}, \tilde{v}, \tilde{w}) \leq \delta_1, \\ 
&\mathcal{E}_{\frac{1}{2}}(\rho_0,v_0, w_0|\bar\rho_0,\bar v_0, \bar{w}_0) \leq \delta_2, \\
&||(\bar\rho_0,\bar u_0)-(\tilde{\rho}, \tilde{u})||_{H^6(\R)} \leq \delta_3, \\
& \bar{\rho_0}(x_1) \geq c_0 > 0.
\end{align}

\subsubsection{Periodic reference function}\label{sec:new1}
First, we define the periodic reference function. Consider the interval $[-R,2R]$. Then, define
\begin{align}\label{eq:reflectdef3}
(\tilde{\rho}_R, \tilde{u}_{R,2})(x_1):=
\begin{cases}
(\tilde{\rho}, \tilde{u}_{R,2})(x_1), & x_1 \in [-R, \frac{R}{3}], \\[1mm]
(\rho_+, u_{2,+}) & x_1 \in [\frac{R}{3}, \frac{2R}{3}], \\[1mm]
(\tilde{\rho}, \tilde{u}_{R,2})(R-x_1) & x_1 \in [\frac{2R}{3}, \frac{4R}{3}], \\[1mm]
(\rho_-, u_{2,-}) & x_1 \in [\frac{4R}{3}, 2R].
\end{cases}
\end{align}
Note that these are smooth as the reference functions $\tilde{\rho}, \tilde{u}_{R,2}$ are constant outside of $\{x_1 \in [-1,1]\}$. Some more care needs to taken to define the periodized velocity component $\tilde{u}_{R,1}$. This is because the equation \eqref{eq:2DNS_T_R} verifies an anti-symmetry property in the first component of the velocity for planar solutions; formally, if $(\rho(t, x_1), u_1(t, x_1), u_2(t, x_1))$ solves \eqref{eq:2DNS_T_R}, then $(\rho(t, -x_1), -u_1(t,-x_1), u_2(t,-x_1))$ remains a solution to the equation. This will be essential in the proof of \Cref{lem:appclem}. Let $\theta$ be a smooth function verifying
\begin{align}\label{eq:thetaprops}
0 \leq \theta \leq 1, \ \theta(s)=0 \text{ for $s \leq 0$,} \ \theta(s)=1 \text{ for $s \geq 1$}.
\end{align}
We define
\begin{align}\label{eq:reflectdef4}
\tilde{u}_{R,1}(x_1):=
\begin{cases}
\tilde{u}_{1}(x_1), & x_1 \in [-R, \frac{R}{3}], \\[1mm]
u_{1,+}(1-2\theta(\frac{3x_1}{R}-1)) & x_1 \in [\frac{R}{3}, \frac{2R}{3}], \\[1mm]
-\tilde{u}_{1}(R-x_1) & x_1 \in [\frac{2R}{3}, \frac{4R}{3}], \\[1mm]
u_{1,-}(2\theta(\frac{3x_1}{R}-4)-1) & x_1 \in [\frac{4R}{3}, \frac{5R}{3}], \\[1mm]
u_{1,-} & x \in [\frac{5R}{3}, 2R].
\end{cases}
\end{align}
Then, we define $(\tilde{\rho}_R, \tilde{u}_R)$ on $\T_R$ by periodicity. 

\subsubsection{Periodic planar strong solution}\label{sec:new2}
Next, we define the initial data for the periodic strong solution. Let $\chi$ be an even, smooth bump function verifying
\begin{align}\label{eq:chiprops}
0 \leq \chi \leq 1, \ \chi(s)=1 \text{ for $|s| \leq \frac{1}{4}$,} \ \chi(s)=0 \text{ for $|s| \geq \frac{1}{3}$,} \ |\chi'(s)|^2 \leq C\chi(s)(1-\chi(s)).
\end{align}
Let us also define
\begin{align}\label{eq:chiRdef}
\chi_R(x_1)=\chi(\frac{x_1}{R}).
\end{align}
Then, define the periodic planar initial data on $[-R, 2R]$:
\begin{align}\label{eq:reflectdef2}
&(\bar{\rho}_R)_0(x_1)=\tilde{\rho}_R(x_1)+\chi_R(x_1)(\bar{\rho}_0(x_1)-\tilde{\rho}(x_1))+\chi_R(R-x_1)(\bar{\rho}_0(R-x_1)-\tilde{\rho}(R-x_1)), \\
&(\bar{u}_{R,1})_{0}(x_1)=\tilde{u}_{R,1}(x_1)+\chi_R(x_1)((\bar{u}_1)_0(x_1)-\tilde{u}_1(x_1))-\chi_R(R-x_1)((\bar{u}_1)_0(R-x_1)-\tilde{u}_1(R-x_1)), \\
&(\bar{u}_{R,2})_{0}(x_1)=\tilde{u}_{R,2}(x_1)+\chi_R(x_1)((\bar{u}_2)_0(x_1)-\tilde{u}_2(x_1))+\chi_R(R-x_1)((\bar{u}_2)_0(R-x_1)-\tilde{u}_2(R-x_1)), \\
\end{align}
and extend periodically as above. Note that these fields verify the representation
\begin{align}\label{eq:piecewiserep}
&(\bar{\rho}_R)_0, (\bar{u}_{R,1})_{0}, (\bar{u}_{R,1})_{0})(x_1)= \\
&\indent \begin{cases}
((\bar{\rho}_0, (\bar{u}_{1})_0, (\bar{u}_{2})_0)(x_1) & |x_1| \leq \frac{R}{4}, \\
((\bar{\rho}_0, -(\bar{u}_{1})_0, (\bar{u}_{2})_0)(R-x_1) & |x_1-R| \leq \frac{R}{4}, \\
\chi_R(x_1)(\bar{\rho}_0, (\bar{u}_{1})_0, (\bar{u}_{2})_0)(x_1)+(1-\chi_R(x_1))(\tilde{\rho},\tilde{u}_{1}, \tilde{u}_2)(x_1) & x_1 \in [-\frac{R}{3}, -\frac{R}{4}] \cup [\frac{R}{4}, \frac{R}{3}], \\
\chi_R(R-x_1)(\bar{\rho}_0, -(\bar{u}_{1})_0, (\bar{u}_{2})_0)(R-x_1)+(1-\chi_R(R-x_1))(\tilde{\rho},-\tilde{u}_{1}, \tilde{u}_2)(R-x_1) & x_1 \in [\frac{2R}{3}, \frac{3R}{R}] \cup [\frac{5R}{4}, \frac{4R}{3}], \\
(\tilde{\rho}_R, \tilde{u}_{R,1}, \tilde{u}_{R,2})(x_1) & \text{otherwise}.
\end{cases}
\end{align}
Note that the planar periodic initial data is either equal to the (reflected) original periodic initial data—where the $u_1$ component is also negated--, exactly equal to the periodic reference fields, or a convex combination of the original data and the original reference function. Due to this, it is clear that
\begin{align*}
||((\bar{\rho}_R)_0, (\bar{u}_R)_0)-(\tilde{\rho}_R, \tilde{u}_R)||_{H^6(\T_R)} \leq C\delta_3.
\end{align*}
Further, it may be proved that
\begin{align}\label{eq:reflectedslopebd}
|\partial_{x_1}(\bar{u}_{R,1})_0)| \leq \frac{1}{2\nu}(\bar{\rho}_R)_0^{\gamma-1},
\end{align}
for $R$ sufficiently large. Indeed, \eqref{eq:reflectedslopebd} holds by definition when $|x_1| \leq \frac{R}{4}$ or $|x_1-R| \leq \frac{R}{4}$. When $(\bar{u}_{R,1})_0=\tilde{u}_{R,1}$, note that from the definition \eqref{eq:reflectdef4} on these regions it holds that
\begin{align*}
\partial_{x_1}(\bar{u}_{R,1})_0=\partial_{x_1}\tilde{u}_{R,1} \leq \frac{C}{R},
\end{align*}
which, combined with $(\bar{\rho}_R)_0 \geq c > 0$, grants \eqref{eq:reflectedslopebd} for $R$ sufficiently large.
Finally, we consider the regions where $(\bar{u}_{R,1})_0$ is a convex combination of the original data and reference function. For simplicity, let us consider only the region $[\frac{R}{4}, \frac{R}{3}]$. In this region, we have
\begin{align*}
\partial_{x_1}(\bar{u}_{R,1})_0=\chi_R'(x_1)[(\bar{u}_1)_0-u_{1,+}]+\chi_R(x_1)\partial_{x_1}(\bar{u}_1)_0.
\end{align*}
Let
\begin{align}\label{eq:omegaR}
\omega_R:=||(\bar{\rho}_0, \bar{u}_0)-(\tilde{\rho}, \tilde{u})||_{H^5(\{|x_1| \geq \frac{R}{5}\})}.
\end{align}
Then, $\omega_R \to 0$ as $R \to \infty$ by \eqref{eq:initialbds}. This immediately gives 
\begin{align*}
|\partial_{x_1}(\bar{u}_{R,1})_0| \leq C \omega_R \to 0,
\end{align*}
giving \eqref{eq:reflectedslopebd} for $R$ sufficiently large.
\par Then, the existence of a planar strong solution to \eqref{eq:2DNS_T_R} with initial data $((\bar{\rho}_R)_0, (\bar{u}_{R})_{0})$ follows from the following proposition.
\begin{proposition}\label{prop:1dexist+bds}
Let $(\rho_0,v_0, w_0)$ be the suitable initial data in the sense of \Cref{def:suitableinitial}.  Let $T > 0$ be fixed. Then, there exists a planar strong solution $(\bar\rho_R,\bar v_R, \bar w_R)(t,x_1)$ to \eqref{eq:2DNS_T_R} with initial data $((\bar \rho_R)_0, (\bar v_R)_0, (\bar{w}_R)_0)$, where $\bar{v}_R=\bar u_R+\nu \nabla \ln(\bar \rho_R)$, and $\bar{w}_R=\nu \nabla \ln(\bar \rho_R)$. Further, there exists a constant $C=C(T,\nu,\delta_3, c_0) > 0$ such that the following holds:
\begin{align*}
&\inf_{(t,x_1) \in [0,T] \times \T_R}\bar\rho_R(t,x_1)\geq C^{-1}, \\[1mm]
&\sup_{t \in [0,T]}\mathcal{E}_{\frac{1}{2},0}(\bar{\rho}_R, \bar{v}_R, \bar{w}_R|\tilde{\rho}_R,\tilde{v}_R,\tilde{w}_R)(t)+||(\bar\rho_R,\bar u_{R})-(\tilde \rho_R, \tilde u_R)||_{L^\infty([0,T] \times \T_R)} \\
&+||(\bar\rho_R,\bar u_{R})-(\tilde \rho_R, \tilde u_R)||_{L^\infty([0,T];H^6(\T_R))}+||\bar u_{R}-\tilde u_R||_{L^2([0,T];H^7(\T_R))} \leq C.
\end{align*}
The constant $C(T,\nu,\delta_3, c_0)$ does not depend on $R$.
\end{proposition}
\Cref{prop:1dexist+bds} is proven in \Cref{sec:appendix1dexistence}. Note that similar analyses on $\T$ and $\R$ have been performed in \cite{MR4049919}, \cite{MR4113341}. However, in our situation there are two notable differences. Firstly, the equation for $\bar{u}_{R,2}$ is present. Secondly, we are interested in obtaining Sobolev estimates uniform in $R$, so we carefully do the estimates to make sure they are independent of $R$.

\subsubsection{Periodic 2D weak solution}\label{sec:new3}
Now, we appropriately define the periodic initial data for the 2D weak solution. Let $m_0:=\rho_0u_0$. Then, define
\begin{align}\label{eq:reflectdef1}
(\rho_R)_0(x_1,x_2)&=\tilde{\rho}_R(x_1)+\chi_R(x_1)(\rho_0(x_1,x_2)-\tilde{\rho}(x_1))+\chi_R(R-x_1)(\rho_0(R-x_1,x_2)-\tilde{\rho}(R-x_1)), \\[1mm]
(m_{R,1})_{0}(x_1,x_2)&=\tilde{\rho}_R(x_1)\tilde{u}_{R,1}(x_1)+\chi_R(x_1)((m_1)_0(x_1,x_2)-\tilde{\rho}(x_1)\tilde{u}_1(x_1)) \\
&-\chi_R(R-x_1)((m_1)_0(R-x_1)-\tilde{\rho}(x_1)\tilde{u}_1(R-x_1)), \\[1mm]
(m_{R,2})_{0}(x_1)&=\tilde{\rho}_R(x_1)\tilde{u}_{R,2}(x_1)+\chi_R(x_1)((m_2)_0(x_1,x_2)-\tilde{\rho}(x_1)\tilde{u}_2(x_1)) \\
&+\chi_R(R-x_1)((m_2)_0(R-x_1)-\tilde{\rho}(R-x_1)\tilde{u}_2(R-x_1)).
\end{align}
Similarly to \eqref{eq:piecewiserep}, the periodized initial data are either equal to the (reflected) original initial data, the periodized reference fields, or a convex combination of the original data and the original reference fields.
\par Extend $((\rho_R)_0,(m_R)_0)$  to $\T_R$ periodically in $x_1$. Further, we may define $(v_R)_0$ and $(w_R)_0$ on $\T_R$ via the formulas 
\begin{align}\label{eq:reflectdef5}
\sqrt{(\rho_R)_0}(w_R)_0&=2\nu \nabla \sqrt{(\rho_R)_0}, \\
\sqrt{(\rho_R)_0}(v_R)_0&=\frac{(m_R)_0}{\sqrt{(\rho_R)_0}}+\sqrt{(\rho_R)_0}(w_R)_0.
\end{align}
The following lemma shows that $((\rho_R)_0, (v_R)_0, (w_R)_0)$ is valid initial data to give rise to a $2D$ weak solution.
\begin{lemma}\label{lem:newlem1}
Let $(\rho_R)_0, (v_R)_0, (w_R)_0)$ be the initial data defined in \eqref{eq:reflectdef1}, \eqref{eq:reflectdef5}. Then
\begin{align}\label{eq:reflectedinitproperty}
\mathcal{E}_{\frac{1}{2},0}((\rho_R)_0, (v_R)_0, (w_R)_0|\tilde \rho_R, \tilde v_R, \tilde w_R) \leq C \delta_1+C(\tau_R+\frac{1}{R}),
\end{align}
and
\begin{align}\label{eq:reflectedinitproperty2}
\mathcal{E}_{\frac{1}{2},0}((\rho_R)_0, (v_R)_0, (w_R)_0|(\bar \rho_R)_0, (\bar v_R)_0, (\bar w_R)_0) \leq C \delta_2+C(\tau_R+\omega_R+\frac{1}{R}),
\end{align}
where
\begin{align*}
\tau_R:=\int_{\{|x_1|> \frac{R}{4}\} \times \T}\rho_0 \big(\frac{|v_0-\tilde{v}|^2+|w_0-\tilde{w}|^2}{2}\big)+F(\rho_0|\tilde{\rho})dx \to 0,
\end{align*}
and $\omega_R$ is defined in \eqref{eq:omegaR}
\end{lemma}
\begin{proof}
We first prove \eqref{eq:reflectedinitproperty}. We have
\begin{align*}
\mathcal{E}_{\frac{1}{2},0}((\rho_R)_0, (v_R)_0, (w_R)_0|\tilde \rho_R, \tilde v_R, \tilde w_R)&=\int_{\{|x_1| \leq \frac{R}{4}\}}(\rho_R)_0\big(\frac{|(v_R)_0-\tilde{v_R}|^2}{2}+\frac{|(w_R)_0-\tilde{w_R}|^2}{2}\big)+F((\rho_R)_0|\tilde{\rho}_R)dx \\
&+\int_{\{|x_1-R| \leq \frac{R}{4}\}}(\rho_R)_0\big(\frac{|(v_R)_0-\tilde{v_R}|^2}{2}+\frac{|(w_R)_0-\tilde{w_R}|^2}{2}\big)+F((\rho_R)_0|\tilde{\rho}_R)dx \\
&+\int_{\{\frac{R}{3} \leq x_1 \leq \frac{2R}{3}\}}(\rho_R)_0\big(\frac{|(v_R)_0-\tilde{v_R}|^2}{2}+\frac{|(w_R)_0-\tilde{w_R}|^2}{2}\big)+F((\rho_R)_0|\tilde{\rho}_R)dx \\
&+\int_{\{\frac{4R}{3} \leq x_1 \leq \frac{-R}{3}\}}(\rho_R)_0\big(\frac{|(v_R)_0-\tilde{v_R}|^2}{2}+\frac{|(w_R)_0-\tilde{w_R}|^2}{2}\big)+F((\rho_R)_0|\tilde{\rho}_R)dx \\
&+\int_{\{0 < \chi_R(x_1) < 1\}}(\rho_R)_0\big(\frac{|(v_R)_0-\tilde{v_R}|^2}{2}+\frac{|(w_R)_0-\tilde{w_R}|^2}{2}\big)+F((\rho_R)_0|\tilde{\rho}_R)dx \\
&+\int_{\{0 < \chi_R(R-x_1) < 1\}}(\rho_R)_0\big(\frac{|(v_R)_0-\tilde{v_R}|^2}{2}+\frac{|(w_R)_0-\tilde{w_R}|^2}{2}\big)+F((\rho_R)_0|\tilde{\rho}_R)dx.
\end{align*}
The first and second terms are bounded by $\delta_1$, while the third and fourth terms are zero. So it remains to control the last two. For simplicity, we will consider only $[\frac{R}{4}, \frac{R}{3}]$, as the other regions are identical. Using the identities \eqref{eq:reflectdef5} and the fact that $\tilde{\rho}_R, \tilde{u}_R$ is constant on this interval, we have
\begin{align*}
(\rho_R)_0\big(\frac{|(v_R)_0-\tilde{v_R}|^2}{2}+\frac{|(w_R)_0-\tilde{w_R}|^2}{2}\big)+F((\rho_R)_0|\tilde{\rho}_R)=\frac{|(m_R)_0+\nu \nabla (\rho_R)_0-(\rho_R)_0u_\pm|^2+\nu^2|\nabla (\rho_R)_0|^2}{2(\rho_R)_0}+F((\rho_R)_0|\rho_{\pm}).
\end{align*}
Further note that on this interval, we have representations
\begin{align*}
(\rho_R)_0&=\chi_R \rho_0+(1-\chi_R)\rho_{+}, \\
\nabla (\rho_R)_0&=\chi_R \nabla \rho_0+(\rho_0-\rho_+)\nabla \chi_R, \\
(m_R)_0&=\chi_R m_0+(1-\chi_R)\rho_{+}u_{+}.
\end{align*}
So, noting that the function $(\rho, m, \nabla \rho) \mapsto \frac{|m+\nu \nabla \rho-\rho u_{+}|^2+\nu^2|\nabla \rho|^2}{2\rho}+F(\rho|\rho_{+})$ is convex in $(\rho, m, \nabla \rho)$, we obtain via the convexity inequality
\begin{align}\label{eq:midstep5}
&\int_{[\frac{R}{4},\frac{R}{3}]}\frac{|(m_R)_0+\nu \nabla (\rho_R)_0-(\rho_R)_0u_+|^2+\nu^2|\nabla (\rho_R)_0|^2}{2(\rho_R)_0}+F((\rho_R)_0|\rho_{+}) \\
&\leq \int_{[\frac{R}{4},\frac{R}{3}]}\frac{|m_0+\nu \nabla \rho_0-\rho_0u_+|^2+\nu^2|\nabla \rho_0|^2}{2\rho_0}+F(\rho_0|\rho_{+})+C\int_{[\frac{R}{4}, \frac{R}{3}]}\frac{(\rho_0-\rho_+)^2|\nabla \chi_R|^2}{\chi_R\rho_0+(1-\chi_R)\rho_+}dx \\
&\leq \tau_R+C\int_{[\frac{R}{4}, \frac{R}{3}]}\frac{(\rho_0-\rho_+)^2|\nabla \chi_R|^2}{\chi_R\rho_0+(1-\chi_R)\rho_+}dx.
\end{align}
To control the second term, using \eqref{eq:chiprops} and the identity
\begin{align*}
\frac{\chi_R(1-\chi_R)(\rho_0-\rho_+)^2}{\chi_R \rho_0+(1-\chi_R)\rho_{+}}\leq (1-\chi_R)\rho_0+\chi_R \rho_+ \leq \rho_0+\rho_+,
\end{align*}
we obtain 
\begin{align*}
\int_{[\frac{R}{4}, \frac{R}{3}]}\frac{(\rho_0-\rho_+)^2|\nabla \chi_R|^2}{\chi_R\rho_0+(1-\chi_R)\rho_+}dx &\leq \frac{C}{R^2}\int_{[\frac{R}{4}, \frac{R}{3}]}(\rho_0+\rho_+)dx \\
&\leq \frac{C}{R^2}\int_{[\frac{R}{4},\frac{R}{3}]}(1+F(\rho_0|\rho_+)+\rho_+)dx \leq C\frac{(1+\delta_1)}{R},
\end{align*}
where we have used \eqref{eq:controlbyrelent} in the second inequality. This proves \eqref{eq:reflectedinitproperty}.
\par To prove \eqref{eq:reflectedinitproperty2}, it again suffices to consider only the intervals where $\chi_R \not \in \{0,1\}$. We again restrict to the interval $[\frac{R}{4}, \frac{R}{3}]$ for simplicity. Here, we need to control
\begin{align}\label{eq:midstep}
\int_{[\frac{R}{4},\frac{R}{3}]}(\rho_R)_0\big(\frac{|(v_R)_0-(\bar v_R)_0|^2}{2}+\frac{|(w_R)_0-(\bar w_R)_0|^2}{2}\big)+F((\rho_R)_0|(\bar \rho_R)_0)dx.
\end{align}
For the first term in \eqref{eq:midstep}, we have
\begin{align}\label{eq:midstep2}
\int_{[\frac{R}{4},\frac{R}{3}]}(\rho_R)_0\big(\frac{|(v_R)_0-(\bar v_R)_0|^2}{2}+\frac{|(w_R)_0-(\bar w_R)_0|^2}{2}\big)&\leq \int_{[\frac{R}{4},\frac{R}{3}]}(\rho_R)_0\big(|(v_R)_0-u_+|^2+|(w_R)_0|^2\big) \\
&+\int_{[\frac{R}{4},\frac{R}{3}]}(\rho_R)_0\big(|(\bar v_R)_0-u_+|^2+|(\bar w_R)_0|^2\big).
\end{align}
The first term on the right-hand side of \eqref{eq:midstep2} is controlled by $\tau_R+\frac{C}{R}$ as above. For the second term, using \eqref{eq:controlbyrelent}, we have
\begin{align}\label{eq:midstep3}
\int_{[\frac{R}{4},\frac{R}{3}]}(\rho_R)_0\big(|(\bar v_R)_0-u_+|^2+|(\bar w_R)_0|^2\big) &\leq C\int_{[\frac{R}{4},\frac{R}{3}]}(1+F((\rho_R)_0|\rho_+))\big(|(\bar v_R)_0-u_+|^2+|(\bar w_R)_0|^2\big) \\
&\leq C(\omega_R+\omega_R^2\int_{[\frac{R}{4},\frac{R}{3}]}F((\rho_R)_0|\rho_+)) \\
& \leq C(\omega_R+\omega_R^2\tau_R).
\end{align}
So, we obtain 
\begin{align}\label{eq:midstep4}
\int_{[\frac{R}{4},\frac{R}{3}]}(\rho_R)_0\big(\frac{|(v_R)_0-(\bar v_R)_0|^2}{2}+\frac{|(w_R)_0-(\bar w_R)_0|^2}{2}\big) \leq C(\tau_R+\omega_R+\omega_R^2\tau_R+\frac{1}{R}).
\end{align}
For the second term in \eqref{eq:midstep}, we compute
\begin{align*}
\int_{[\frac{R}{4},\frac{R}{3}]}F((\rho_R)_0|(\bar \rho_R)_0)&=\int_{[\frac{R}{4},\frac{R}{3}]}\big(F((\rho_R)_0|\rho_+)+(F'(\rho_+)-F'((\bar \rho_R)_0)((\rho_R)_0-\rho_+)+F(\rho_+|(\bar \rho_R)_0)\big) \\
&=\sum_{j=1}^3R_j.
\end{align*}
We have computed already in \eqref{eq:midstep5} that $R_1 \leq \tau_R$. For $R_2$, we have
\begin{align}\label{eq:midstep6}
\int_{[\frac{R}{4},\frac{R}{3}]}(F'(\rho_+)-F'((\bar \rho_R)_0)((\rho_R)_0-\rho_+) &\leq C\int_{[\frac{R}{4},\frac{R}{3}]}(\rho_+-(\bar \rho_R)_0)(\sqrt{F((\rho_R)_0|\rho_+)}+F((\rho_R)_0|\rho_+)^{\frac{1}{\gamma}}) \\
&\leq C(\omega_R\sqrt{\tau_R}+\omega_R\tau_R^{\frac{1}{\gamma}}).
\end{align}
For $R_3$ we have
\begin{align}\label{eq:midstep7}
\int_{[\frac{R}{4},\frac{R}{3}]}F(\rho_+|(\bar \rho_R)_0) \leq C\omega_R^2.
\end{align}
So, combining \eqref{eq:midstep5}, \eqref{eq:midstep6} and \eqref{eq:midstep7}, we have
\begin{align}\label{eq:midstep8}
\int_{[\frac{R}{4},\frac{R}{3}]}F((\rho_R)_0|(\bar \rho_R)_0) \leq C(\tau_R+\omega_R+\omega_R^2).
\end{align}
Combined with \eqref{eq:midstep4} and plugging into \eqref{eq:midstep}, this gives
\begin{align*}
\int_{[\frac{R}{4},\frac{R}{3}]}(\rho_R)_0\big(\frac{|(v_R)_0-(\bar v_R)_0|^2}{2}+\frac{|(w_R)_0-(\bar w_R)_0|^2}{2}\big)+F((\rho_R)_0|(\bar \rho_R)_0)dx \leq C(\tau_R+\omega_R+\frac{1}{R}).
\end{align*}
This completes the proof of \eqref{eq:reflectedinitproperty2}.
\end{proof}
Then, by \eqref{eq:reflectedinitproperty} and \Cref{prop:1/2entropyrenormalizedexistenceT_R}, there exists a $\frac{1}{2}$-entropic renormalized weak solution $(\rho_R, v_R, w_R)$ to $\eqref{eq:2DNS_T_R}$ with initial data $((\rho_R)_0, (v_R)_0, (w_R)_0)$.

\subsubsection{Uniform-in-$R$ bounds on 2D weak solution}\label{estimate-R}
Fix $T>0$. By uniform-in-$R$ bounds on the planar strong solution in \Cref{prop:1dexist+bds} and the uniform-in-$R$ relative $\frac{1}{2}$-entropy estimate in \Cref{prop:1/2entropyrenormalizedexistenceT_R}, we have the following uniform-in-$R$ bounds on 2D weak solution $(\rho_R, v_R, w_R)$:
\begin{equation}\label{eq:uniform-R}
\begin{array}{ll}
\di \sup_{t\in[0,T]}\mathcal{E}_{\frac{1}{2},0}((\rho_R,v_R,w_R)|(\bar\rho_R,\bar v_R, \bar w_R))(t)\\
\di\quad  \di +\nu \int_0^T\int_{\T_R}|\mathcal{T}(u_R)-\sqrt{\rho_R}\nabla \bar u_R|^2dxdt \\[3mm]
\di\quad  +\nu \int_0^T\int_{\T_R}|\frac{2}{\sqrt{\gamma}}\nabla \rho_R^{\frac{\gamma}{2}}-\sqrt{\rho_Rp'(\rho_R)}\nabla \ln(\bar \rho_R)|^2dxdt \leq C,
\end{array}
\end{equation}
which implies that 
\begin{align}\label{eq:uniforminRestimates}
&||\sqrt{\rho_R}(u_R-\bar u_R)||_{L^\infty(0,T; L^2([-R,R/2]\times \mathbb{T}))} \leq C;\nonumber\\[3mm]
& ||p(\rho_R|\bar \rho_R)||_{L^\infty(0,T; L^1([-R,R/2]\times \mathbb{T}))} \leq C;\nonumber\\[3mm]
&||\sqrt{\rho_R}\ (\frac{w_R}{\nu}-\nabla \ln\bar \rho_R)||_{L^\infty(0,T; L^2([-R,R/2]\times \mathbb{T}))} \leq C;
\\[3mm] 
&||\frac{2}{\sqrt{\gamma}} \nabla \rho_R^{\frac{\gamma}{2}}-\sqrt{\rho_Rp'(\rho_R)}\nabla \ln(\bar \rho_R)||_{L^2([0,T]\times [-R,R/2]\times \mathbb{T})} \leq C;\\[3mm]
&||\mathcal{T}(u_R)-\sqrt{\rho_R}\nabla \bar u_R||_{L^2([0,T]\times [-R,R/2]\times \mathbb{T})} \leq C;
\end{align}
By \Cref{prop:1dexist+bds}, noting that $\forall (t,x_1)\in [0,T]\times [-R,R/2]\times \mathbb{T},$
$$
C^{-1} \leq \bar \rho_R(t,x_1)\leq C,
$$
we can do the estimate as
\begin{equation}{\label{I}}
\begin{array}{ll}
\di \int_{-R}^{R/2}\int_{\T}(\sqrt{\rho_R}-\sqrt{\bar \rho_R})^2 dx\\[4mm]
\di =\int_{-R}^{R/2}\int_{\T}(\sqrt{\rho_R}-\sqrt{\bar \rho_R})^2\big(\mathbf{1}_{\{\frac 12 C^{-1}(T, \nu)\leq \rho_R\leq \frac 32 C(T, \nu)\}}  \\[2mm]
\di \qquad \qquad \qquad \qquad \qquad \qquad  +\mathbf{1}_{\{0\leq \rho_R< \frac 12 C^{-1}(T, \nu)\}} +\mathbf{1}_{\{\rho_R>\frac 32 C(T, \nu)\}}\big)dx\\[2mm]
\di :=I_1+I_2+I_3.
\end{array}
\end{equation}
For $I_1$, we have
\begin{equation}\label{I1}
\begin{array}{ll}
\di I_1=\int_{-R}^{R/2}\int_{\T}(\sqrt{\rho_R}-\sqrt{\bar \rho_R})^2\mathbf{1}_{\{\frac 12 C^{-1}(T, \nu)\leq \rho_R\leq \frac 32 C(T, \nu)\}}dx\\[3mm]
\quad \di \leq C \int_{-R}^{R/2}\int_{\T}(\rho_R-\bar \rho_R)^2\mathbf{1}_{\{\frac 12 C^{-1}(T, \nu)\leq \rho_R\leq \frac 32 C(T, \nu)\}}dx\\[3mm]
\di \quad \leq C \int_{-R}^{R/2}\int_{\T}  p(\rho_R|\bar \rho_R) dx<\infty.
\end{array}
\end{equation}
Then 
\begin{equation}\label{I2}
\begin{array}{ll}
\di I_2=\int_{-R}^{R/2}\int_{\T}(\sqrt{\rho_R}-\sqrt{\bar \rho_R})^2\mathbf{1}_{\{0\leq \rho_R< \frac 12 C^{-1}(T, \nu)\}}dx \\[3mm]
\quad \di \leq C \int_{-R}^{R/2}\int_{\T}p(\rho_R|\bar \rho_R)\mathbf{1}_{\{0\leq \rho_R< \frac 12 C^{-1}(T, \nu)\}}dx<\infty,
\end{array}
\end{equation}
and
\begin{equation}\label{I3}
\begin{array}{ll}
\di I_3=\int_{-R}^{R/2}\int_{\T}(\sqrt{\rho_R}-\sqrt{\bar \rho_R})^2\mathbf{1}_{\{\rho_R>\frac 32 C(T, \nu)\}}dx\\[3mm]
\quad \di \leq C \int_{-R}^{R/2}\int_{\T}p(\rho_R|\bar \rho_R)\mathbf{1}_{\{\rho_R>\frac 32 C(T, \nu)\}}dx<\infty.
\end{array}
\end{equation}
Substituting \eqref{I1}-\eqref{I3} into \eqref{I} yields the following uniform-in-$R$ estimate
\begin{equation}\label{eq:L2}
||\sqrt{\rho_R}-\sqrt{\bar \rho_R}||_{L^\infty(0,T; L^2([-R,R/2]\times \mathbb{T}))} \leq C.
\end{equation}
Similarly, the following uniform-in-$R$ estimate holds true.
\begin{equation}\label{eq:L2-g}
||\rho_R^{\frac \gamma 2}-\bar \rho_R^{\frac \gamma 2}||_{L^\infty(0,T; L^2([-R,R/2]\times \mathbb{T}))} \leq C.
\end{equation}
On the other hand, \eqref{eq:L2-g} implies that
$$
\sup_{t\in[0,T]}\int_{-R}^{R/2}\int_{\T}(\rho_R+\rho_R^\gamma)|\nabla \ln \bar \rho_R|^2dx \leq C,
$$
which combined with \eqref{eq:uniforminRestimates} and \eqref{eq:vandwinitdatachannel} gives
\begin{equation}\label{eq:der-L2}
\sup_{t\in[0,T]}\int_{-R}^{R/2}\int_{\T} |\nabla (\sqrt{\rho_R})|^2dx \leq C,
\end{equation}
and
\begin{equation}\label{eq:der-L2-g}
\int_0^T\int_{\T_R} |\nabla (\rho_R^{\frac \gamma 2})|^2dx \leq C.
\end{equation}
Therefore, \eqref{eq:L2}, \eqref{eq:der-L2} and \eqref{eq:L2-g}, \eqref{eq:der-L2-g} imply the following uniform-in-$R$ estimates
\begin{equation}\label{H1}
\begin{array}{ll}
 \di   ||\sqrt{\rho_R}-\sqrt{\bar \rho_R}||_{L^\infty(0,T; H^1([-R,R/2]\times \mathbb{T}))} \leq C,\\[3mm]
\di ||\rho_R^{\frac \gamma 2}-\bar \rho_R^{\frac \gamma 2}||_{L^2(0,T; H^1([-R,R/2]\times \mathbb{T}))} \leq C. 
\end{array}
\end{equation}

\subsubsection{Passing into the limit $R \to \infty$}\label{sec:Rlimit}

$\forall M>0,$ set $\Omega_M:=[-M, M]\times \mathbb{T}$. If $R$ is sufficiently large, then $\Omega_M\subset\subset [-R,R/2]\times \T$.

Based on the uniform-in-$R$ estimates in Subsection \ref{estimate-R}, we have the following weak/strong convergence.

\underline{Step 1}. Convergence of $\rho_R$.

\

First, 
$$
\begin{array}{ll}
\rho_Ru_R&\di =\sqrt{\rho_R}(u_R-\bar u_R)\sqrt{\rho_R}+\rho_R\bar u_R\\[2mm]
&\di =\sqrt{\rho_R}(u_R-\bar u_R)(\sqrt{\rho_R}-\sqrt{\bar\rho_R})+\sqrt{\bar\rho_R}\sqrt{\rho_R}(u_R-\bar u_R)\\[2mm]
&\di \quad +(\sqrt{\rho_R}-\sqrt{\bar\rho_R})^2\bar u_R+2\sqrt{\bar\rho_R}\bar u_R(\sqrt{\rho_R}-\sqrt{\bar\rho_R})+\bar \rho_R\bar u_R,
\end{array}
$$
is uniformly bounded in $L^\infty(0,T; L^{p}(\Omega_M))$ for any $1\leq p<2$.
Then by the continuity equation, we have 
$$
\partial_t \rho_R\ {\rm is}\ {\rm uniformly}\ {\rm bounded}\ {\rm in}\ 
L^\infty(0,T; W^{-1, p}(\Omega_M)),
$$
for any $1\leq p<2$.
On the other hand, 
$$\rho_R=\big((\sqrt{\rho_R}-\sqrt{\bar\rho_R})+\sqrt{\bar\rho_R}\big)^2=(\sqrt{\rho_R}-\sqrt{\bar\rho_R})^2+2\sqrt{\bar\rho_R}(\sqrt{\rho_R}-\sqrt{\bar\rho_R})+\bar \rho_R,$$
is uniformly bounded in $L^\infty(0,T; L^{p}(\Omega_M))$ for any $1\leq p<\infty$.
And
$$
\nabla \rho_R=2 \sqrt{\rho_R} \nabla \sqrt{\rho_R}=2(\sqrt{\rho_R}-\sqrt{\bar\rho_R})\nabla \sqrt{\rho_R} +2\sqrt{\bar \rho_R} \nabla \sqrt{\rho_R},
$$
is uniformly bounded in $L^\infty(0,T; L^{p}(\Omega_M))$ for any $1\leq p<2$.
Thus, $\rho_R$ is uniformly bounded in $L^\infty(0,T; W^{1,p}(\Omega_M))$ for any $1\leq p<2$.

By Aubin-Lions Lemma, $\rho_R$ is precompact in $C([0,T];L^{q}(\Omega_M))$ for any $2\leq q<\infty$. Up to a subsequence, 
\begin{equation}\label{eq:rho-con}
    \rho_R\rightarrow \rho \ {\rm in }\ C([0,T];L^{q}(\Omega_M)),
\end{equation}
as $R\rightarrow +\infty$ for any $2\leq q<\infty$.

\underline{Step 2}. Convergence of pressure $p_R$.

\

By Step 1, up to a subsequence, $\rho_R^\gamma-\bar \rho_R^\gamma\rightarrow \rho^\gamma-\bar \rho^\gamma,$ a. e. in $[0,T]\times \Omega_M.$
By \eqref{H1}, we have
$$
\rho_R^{\frac\gamma2}-\bar \rho_R^{\frac\gamma2},
$$ 
is uniformly bounded in $L^2(0,T; H^1(\Omega_M))$.
Therefore,
$$
\rho_R^\gamma-\bar \rho_R^\gamma,
$$ 
is uniformly bounded in $L^1(0,T; L^q(\Omega_M))\cap L^\infty(0,T; L^1(\Omega_M))$ and then is uniformly bounded in $L^r([0,T]\times \Omega_M)$ for any $r\in [1,2)$. Therefore, Egorov Theorem implies that 
$$
\rho_R^\gamma-\bar \rho_R^\gamma\rightarrow \rho^\gamma-\bar\rho^\gamma, \quad {\rm strongly}\ {\rm in}\ L^1([0,T]\times \Omega_M).
$$

\underline{Step 3}. Convergence of the momentum.

\

Since
\begin{equation}
\begin{array}{ll}
    \displaystyle \nabla(\rho_R u_R)\\[2mm]
   \displaystyle  =\big((\sqrt{\rho_R}-\sqrt{\bar\rho_R})+\sqrt{\bar \rho_R}\big)\mathcal{T}(u_R)+2\big(\sqrt{\rho_R}(u_R-\bar u_R)+\sqrt{\rho_R}\bar u_R\big)\nabla\sqrt{\rho_R},
    \end{array}
\end{equation}
is uniformly bounded in $L^2(0,T; L^1(\Omega_M))$, and by Step 1, $\rho_Ru_R$ is uniformly bounded in $L^\infty(0,T; L^{p}(\Omega_M))$ for any $1\leq p<2$, we have that $\rho_Ru_R$ is uniformly bounded in $L^2(0,T; W^{1,1}(\Omega_M))$. On the other hand, by the momentum  equation $\eqref{eq:2DNS_T_R}_2$, it holds that
$$
\partial_t(\rho_Ru_R),
$$
is uniformly bounded in $L^2(0,T; H^{-N}(\Omega_M))$ for sufficiently large $N>1$.

Therefore, by Aubin-Lions Lemma, the momentum $\rho_R u_R$ is pre-compact in $L^2(0, T; L^q(\Omega_M))$ for any $q\in[1,2)$. Up to a subsequence,
$$
\rho_Ru_R\rightarrow m,
$$
strongly in $L^2(0, T; L^q(\Omega_M))$ for any $q\in[1,2)$.

\

\underline{Step 4}. Convergence of $\mathcal{T}(u_R)$.

\

By \eqref{eq:uniforminRestimates}, we have that  both $\mathcal{T}(u_R)$ and $\mathcal{S}(u_R)$ is uniformly bounded in $L^2([0,T]\times \Omega_M)$, and then 
\begin{equation}\label{eq:diffusionconv}
\mathcal{T}(u_R) \rightharpoonup \ \mathcal{T},\qquad 
\mathcal{S}(u_R) \rightharpoonup \mathcal{S},
\end{equation}
for some tensors $\mathcal{T}, \mathcal{S} \in L^2([0,T]\times \Omega_M)$.

\

\underline{Step 5}. Convergence of the renormalized quantity.

\

Let $H\in W^{2, \infty}(\mathbb{R^d})$ and the constant $\alpha\in (0, \frac {\gamma q}2)$ with $q\in[1, 4)$. Denote the non-vacuum region $Q_{T, M}:=\{(t,x)\in [0, T]\times \Omega_M| \rho(t, x)>0\}$ and the vacuum region $Q_{T, M}^c:=\{(t,x)\in [0, T]\times \Omega_M| \rho(t, x)=0\}=[0, T]\times \Omega_M \backslash Q_{T, M}.$ Define the velocity
$$
u(t,x):=\left\{
\begin{array}{ll}
\di \frac{m(t,x)}{\rho(t,x)},&\displaystyle (t,x )\in Q_{T, M},\\[3mm]
0,&\displaystyle (t,x )\in Q_{T, M}^c.
\end{array}
\right.
$$
Then the momentum $m(t,x)=\rho(t,x) u(t,x)$ for all $(t,x)\in [0, T]\times \Omega_M$.
Now we want to obtain the convergence of the renormalized quantity $\rho_RH(u_R)$.

$\forall (t,x)\in Q_{T, M},$ we have
$\rho_R(t,x)>0$, otherwise, if $\rho_R(t,x)=0$, then $\rho(t,x)=0$. Since $\rho_R\rightarrow \rho$ and $\rho_R u_R\rightarrow \rho u$ a. e. in $[0, T]\times \Omega_M$, then $\forall (t,x)\in Q_{T, M},$
\begin{equation}
   u_R(t,x)=\frac{\rho_R(t,x) u_R(t,x)}{\rho_R(t,x)}\rightarrow \frac{\rho (t,x)u(t,x)}{\rho(t,x)}=u(t,x),
\end{equation}
a. e. in $Q_{T, M}$. Therefore, 
\begin{equation}
  (\rho_R(t,x))^\alpha H(u_R)(t,x)\rightarrow  (\rho(t,x))^\alpha H(u)(t,x),
\end{equation}
a. e. in $Q_{T, M}$.

On the other hand, $\forall (t,x)\in Q^c_{T, M},$ we have
\begin{equation}
|(\rho_R(t,x))^\alpha H(u_R)(t,x)|\leq C(\rho_R(t,x))^\alpha\rightarrow C(\rho(t,x))^\alpha=0.
\end{equation}
Thus, $\forall (t,x)\in Q^c_{T, M},$ 
$$
(\rho_R(t,x))^\alpha H(u_R)(t,x)\rightarrow 0= (\rho(t,x))^\alpha H(u)(t,x).
$$
Consequently, 
$$
(\rho_R(t,x))^\alpha H(u_R)(t,x)\rightarrow (\rho(t,x))^\alpha H(u)(t,x).
$$
a. e. in $[0, T]\times \Omega_M$. By Step 2, 
$$
\rho_R^{\frac\gamma2}-\bar \rho_R^{\frac\gamma2},
$$
is uniformly bounded in $L^\infty(0,T; L^2(\Omega_M))\cap L^2(0,T; H^1(\Omega_M))$, and therefore in $L^q([0, T]\times \Omega_M)$ for any $1\leq q<4$.

Since
\begin{equation}
\begin{array}{ll}
\displaystyle|(\rho_R(t,x))^\alpha H(u_R)(t,x)|^p \leq C(\rho_R(t,x))^{\alpha p}&= C\big((\rho_R(t,x))^{\alpha p}-(\bar\rho_R(t,x))^{\alpha p}\big) \\
&+C(\bar\rho_R(t,x))^{\alpha p},
\end{array}
\end{equation}
if $1\leq \alpha p<\frac{\gamma q}{2}$ for any $1\leq q<4$, then by Egorov Theorem
\begin{equation}\label{eq:H-con}
(\rho_R)^\alpha H(u_R) \rightarrow \rho^\alpha H(u),
\end{equation}
strongly in $L^p([0,T]\times \Omega_M)$ for any $1\leq p<\frac{\gamma q}{2\alpha}$ with $1\leq q<4$.

\par Now, we are ready to pass into the limit $R \to \infty$. Recall that for any $\phi \in W^{2,\infty}(\R^2)$ compactly supported, there exist two measures $(\Xi_{\phi,R}, \tilde{\Xi}_{\phi,R})$ such that $(\rho_R, v_R, w_R)$ verifies 
\begin{align}\label{eq:Rrenormalizedeqn}
&\int_0^\infty \int_{\T_R}(\rho_R \psi_t+\rho_R u_R \cdot \nabla \psi) dx dt =0, \\
&\int_0^\infty \int_{\T_R}\Big[\rho_R \phi(u_R) \psi_t+\big(\rho_R \phi(u_R) u_R-(2\nu \sqrt{\rho_R}\mathcal{S}(u_R)\phi'(u_R)\big)\cdot \nabla \psi\\
&\qquad\qquad\quad  (-2\rho_R^{\frac{\gamma}{2}}\nabla (\rho_R)^{\frac{\gamma}{2}})\phi'(u_R)\psi\Big]dx dt=\left\langle \Xi_{\phi,R}, \psi\right\rangle,
\end{align}
where $\mathcal{S}(u_R)$ is the symmetric part of $\mathcal{T}(u_R)$ which verifies
\begin{equation}\label{eq:Rrenormalizedeqndiffusion}
\nu \sqrt{\rho_R}\phi^\prime(u_R)[\mathcal{T}(u_R)]_{jk}:=\nu\partial_j\big(\rho_R \phi^\prime(u_R) u_{R,k}\big)-2\nu\sqrt{\rho}u_{R,k}\phi^\prime(u_R)\partial_j\sqrt{\rho_R}+\widetilde\Xi _{\phi,R}.
\end{equation}
In fact, we have more information on these measures. By \cite[Remark 2.1]{MR3490496}, it holds
\begin{align}
\tilde{\Xi}_{\phi,R}=-\phi''(u_R)\mathcal{T}(u_R)\sqrt{\rho_R}u_R,
\end{align}
and by \Cref{prop:1/2entropyrenormalizedexistenceT_R}, we have a bound
\begin{align*}
||\Xi_{\phi,R}||_{\mathcal{M}([0,T] \times \Omega_M)} \leq C||\phi''||_{L^\infty},
\end{align*}
where the bound is uniform in $R$ thanks to the uniform bounds on $||\rho_R||_{L^\infty(0,T;L^1(\Omega_M))}$ and $||\nabla \overline{u}_R||_{L^\infty}$. Similarly, thanks to \eqref{eq:uniforminRestimates}, we obtain
\begin{align*}
||\tilde{\Xi}_{\phi,R}||_{\mathcal{M}([0,T] \times \Omega_M)} \leq C||\phi''||_{L^\infty},
\end{align*}
Then by Banach–Alaoglu theorem, there exist two Radon measures $\Xi_{\phi}, \widetilde \Xi_{\phi} \in \mathcal{M}([0,T] \times \Omega_M)$ such that
\begin{equation}
\Xi_{\phi,R} \rightharpoonup  \Xi_{\phi}, \qquad   \widetilde \Xi_{\phi,R} \rightharpoonup   \widetilde \Xi_{\phi}  
\end{equation}
weakly-$*$ in $\mathcal{M}([0,T] \times \Omega_M)$.

For the renormalized momentum, it holds that
\begin{equation}\label{renorm-m}
\rho_R \phi(u_R)=\sqrt{\rho_R}\phi(u_R)\sqrt{\rho_R}.
\end{equation}
By \eqref{eq:H-con}, we have
\begin{equation}\label{eq:renorm-con}
   \sqrt{\rho_R}\phi(u_R)\rightarrow \sqrt{\rho}\phi(u),
\end{equation}
strongly in $L^p([0,T]\times \Omega_M)$  for any $1\leq p<\gamma q$ with $1\leq q<4$, which together with \eqref{eq:rho-con} in Step 1 implies that
the renormalized momentum in \eqref{renorm-m} converges strongly to $\rho \phi(u)$ in $L^2([0,T]\times \Omega_M)$.

For the renormalized convection term, it holds that
\begin{equation}\label{eq:renorm-ad}
\rho_R u_R\phi(u_R)=\sqrt{\rho_R}u_R \sqrt{\rho_R}\phi(u_R).
\end{equation}
By \eqref{eq:uniforminRestimates}, we have
\begin{equation}\label{m-con}
   \sqrt{\rho_R}u_R \rightharpoonup \sqrt{\rho}u,
\end{equation}
weakly in $L^2([0,T]\times \Omega_M)$, which together with \eqref{eq:renorm-con} implies that
the renormalized convection term in \eqref{eq:renorm-ad} converges to $\rho  u\phi(u)$ in $\mathcal{D}^\prime([0,T]\times \Omega_M)$.

For the renormalized diffusion, \eqref{eq:diffusionconv} along with \eqref{eq:H-con} implies that 
\begin{align*}
\sqrt{\rho_R}\mathcal{S}(u_R)\phi'(u_R) \to \sqrt{\rho}\mathcal{S}\phi'(u),
\end{align*}
in $\mathcal{D}'([0,T] \times \Omega_M)$.

To pass to the limit for the renormalized pressure, by \eqref{eq:H-con} it holds that
\begin{equation}
   (\rho_R)^{\frac\gamma 2}\phi^\prime(u_R)\rightarrow \rho^{\frac\gamma 2}\phi^\prime(u),
\end{equation}
strongly in $L^p([0,T]\times \Omega_M)$ for any $p\in[1,4)$. 

By \eqref{eq:uniforminRestimates}, 
\begin{equation}
\nabla (\rho_R)^{\frac\gamma 2}\rightharpoonup \nabla \rho^{\frac\gamma 2}
\end{equation}
weakly in $L^2([0,T]\times \Omega_M)$.

Then it holds that 
\begin{equation}
2\rho_R^{\frac{\gamma}{2}}\nabla (\rho_R)^{\frac{\gamma}{2}}\phi'(u_R)  \rightarrow  2\rho^{\frac{\gamma}{2}}\nabla (\rho)^{\frac{\gamma}{2}}\phi'(u)
\end{equation}
in $\mathcal{D}^\prime([0,T]\times \Omega_M)$.
Taking $R\rightarrow+\infty$ in \eqref{eq:Rrenormalizedeqn}, we have that for any $\phi \in W^{2,\infty}(\R^2)$ compactly supported, there exist measures $(\Xi_{\phi}, \tilde{\Xi}_{\phi})$ verifying 
\begin{align*}
||\Xi_{\phi}, \tilde{\Xi}_{\phi}||_{\mathcal{M}([0,T] \times \Omega_M)} \leq C||\phi''||_{L^\infty},
\end{align*}
such that $(\rho, v, w)$ verifies 
\begin{align}\label{eq:Mrenormalizedeqn}
&\int_0^T \int_{\Omega_M}(\rho \psi_t+\rho u \cdot \nabla \psi) dx dt =0, \\
&\int_0^T \int_{\Omega_M}\Big[\rho \phi(u) \psi_t+\big(\rho \phi(u) u-(2\nu \sqrt{\rho}\mathcal{S}(u)\phi'(u)\big)\cdot \nabla \psi\\
&\qquad\qquad\quad  (-2\rho^{\frac{\gamma}{2}}\nabla (\rho)^{\frac{\gamma}{2}})\phi'(u)\psi\Big]dx dt=\left\langle \Xi_{\phi,R}, \psi\right\rangle,
\end{align}
for any test function $\psi \in C_c^\infty([0,T] \times \Omega_M)$.
Similarly, we may pass into the limit $R \to \infty$ in \eqref{eq:Rrenormalizedeqndiffusion} to obtain
\begin{equation}\label{eq:Mrenormalizedeqndiffusion}
\nu \sqrt{\rho}\phi^\prime(u)[\mathcal{T}]_{jk}:=\nu\partial_j\big(\rho \phi^\prime(u) u_{k}\big)-2\nu\sqrt{\rho}u_{k}\phi^\prime(u)\partial_j\sqrt{\rho}+\widetilde\Xi _{\phi},
\end{equation}
in $\mathcal{D}'([0,T] \times \Omega_M)$. Using the standard argument taking a sequence of $\phi_n$ such that $\phi_n(y)$ converges to $y_i$ for $i=1,2$ (see \cite[Section 4]{MR3801754}) we obtain that 
that for any $\phi \in W^{2,\infty}(\R^2)$ compactly supported, there exist measures $(\Xi_{\phi}, \tilde{\Xi}_{\phi})$ verifying 
\begin{align*}
||\Xi_{\phi}, \tilde{\Xi}_{\phi}||_{\mathcal{M}([0,T] \times \Omega_M)} \leq C||\phi''||_{L^\infty},
\end{align*}
such that $(\rho, v, w)$ verifies 
\begin{align}\label{eq:Meqn}
&\int_0^T\int_{\Omega_M}(\rho \psi_t+\rho u \cdot \nabla \psi) dx dt =0, \\
&\int_0^T \int_{\Omega_M}\Big[\rho u \psi_t+\big(\rho u \times u-(2\nu \sqrt{\rho} \mathcal{S}\big)\cdot \nabla \psi\\
&\qquad\qquad\quad  (-2\rho^{\frac{\gamma}{2}}\nabla (\rho)^{\frac{\gamma}{2}}) \psi\Big]dx dt=0,
\end{align}
for any test function $\psi \in C_c^\infty([0,T] \times \Omega_M)$.
and
\begin{equation}\label{eq:Meqndiffusion}
\sqrt{\rho}\mathcal{T}\nabla(\rho u)-2\sqrt{\rho} u \otimes \nabla \sqrt{\rho}, 
\end{equation}
in $\mathcal{D}'([0,T] \times \Omega_M)$. Therefore, $\mathcal{T}=\mathcal{T}(u)$, and so $\mathcal{S}=\mathcal{S}(u)$. 

\par Therefore, we proved that $\forall M>1,$ the limit $(\rho, u)$ is still a  weak solution to the $2$-D compressible Navier-Stokes equations in $[0,T]\times \Omega_M$. By the standard diagonal arguments, taking $T,M \to \infty$, we can obtain that the limit $(\rho, u)$ is a weak solution to \eqref{eq:2DNS}, i.e. it verifies \eqref{eq:uequationonchannel} and \eqref{eq:uequationinitdata}.
\par Further, by \eqref{eq:uniforminRestimates}, $\sqrt{\rho_R}w_R$ and $\sqrt{\rho_R}v_R$ are uniformly bounded in $L^2([0,T] \times \Omega_M)$, so they converge up to a subsequence to $\sqrt{\rho}w$ and $\sqrt{\rho}v$. Taking a diagonal sequence again as $T,M \to \infty$, we obtain $\sqrt{\rho} w$, $\sqrt{\rho} v$ which verify \eqref{eq:vandwrelationsonchannel}. The equation \eqref{eq:vandweqnintegratedchannel} and \eqref{eq:vandwinitdatachannel} follow directly from \eqref{eq:uequationonchannel}, \eqref{eq:uequationinitdata}, and \eqref{eq:vandwinitdatachannel}.

\par To obtain the relative entropy estimate \eqref{eq:weakstrongestimatechannel}, we also need to pass into the limit $R \to \infty$ for the planar strong solutions $(\bar \rho_R, \bar v_R, \bar w_R)$. By \Cref{prop:1dexist+bds}, we have uniform-in-$R$ bounds
\begin{align}\label{eq:higherorder}
&\inf_{(t,x) \in [0,T] \times \Omega_M}\bar \rho_R(t,x) \geq C^{-1}, \indent ||\bar \rho_R||_{L^\infty([0,T];H^6(\Omega_M))} \leq C, \\
&||\bar u_R||_{L^\infty([0,T];H^6(\Omega_M))}+|| \bar{u}_R||_{L^2([0,T];H^7(\Omega_M))} \leq C.
\end{align}
Using the equation \eqref{eq:2DNS} to control the time derivatives, this implies the following bounds
\begin{align}\label{eq:convtot}
||\bar \rho_R||_{H^3([0,T] \times \Omega_M)} \leq C, \\
||\bar u_R||_{H^3([0,T] \times \Omega_M)} \leq C, \\
||\nabla \bar u_R||_{H^3([0,T] \times \Omega_M)} \leq C.
\end{align}
Indeed, let us show just one step of the argument. The spatial derivatives of $\bar{\rho}_R$ are controlled already, so we need only to control the derivatives in time. For the first, we have
\begin{align*}
(\bar{\rho}_{R})_t=-(\bar \rho_R \bar u_{R,1})_{x_1} \in L^2([0,T] \times \Omega_M),
\end{align*}
uniformly in $R$ by \eqref{eq:higherorder}. For the second time derivative, we differentiate the continuity equation and use the continuity and momentum equations to replace the time derivatives with spatial ones:
\begin{align}\label{eq:rhott}
 (\bar{\rho_R})_{tt}=((\bar \rho_R \bar u_{R,1})_{x_1}\bar \rho_R)_{x_1}-(\bar{u}_{R,1})_t(\bar{\rho}_{R})_{x_1}-(\bar{u}_{R,1})_{tx_{1}}\bar \rho_R.
\end{align}
Using the momentum equation to obtain
\begin{align*}
(\bar{u}_{R,1})_t=\bar \rho_R^{-1}\Big((2\nu \bar \rho_R(\bar{u}_{R,1})_{x_1})_{x_1}-(\bar \rho_R \bar{u}_{R,1}^2+p(\bar \rho_R))_{x_1}+(\bar \rho_R (\bar{u}_{R,1})_{x_1})_{x_1}\bar u_{R,1}\Big),
\end{align*}
plugging this into \eqref{eq:rhott}, and using \eqref{eq:higherorder}, we obtain
\begin{align*}
(\bar{\rho}_{R})_{tt} \in L^2([0,T] \times \Omega_M),
\end{align*}
uniformly in $R$. By the same process with more time derivatives and doing similarly with $\bar u_R$, we may have \eqref{eq:convtot}. So, by passing to a subsequence, we obtain the convergence
\begin{align}\label{eq:localconvergence}
    \bar \rho_R \to \bar \rho \text{ in }H^2([0,T] \times \Omega_M), \\
    \bar u_R \to \bar u \text{ in }H^2([0,T] \times \Omega_M), \\
    \nabla \bar u_R \to \nabla \bar u \text{ in }H^1([0,T] \times \Omega_M).
\end{align}
 This allows also to obtain at the limit
\begin{align*}
 \overline{v}_R \to \overline{v}=\overline{u}+\nu\nabla \ln(\overline{\rho})  \text{ in }H^1([0,T] \times \Omega_M), \\
 \overline{w}_R \to \overline{w}=\nu\nabla \ln(\overline{\rho})  \text{ in }H^1([0,T] \times \Omega_M). \\
\end{align*}
By a diagonal argument as $T,M \to \infty$ we obtain a planar strong solution $(\bar \rho,\bar u)$ to \eqref{eq:2DNS} verifying 
\begin{align}\label{eq:strongsolbds}
\sup_{t \in [0,T]}\mathcal{E}_{\frac{1}{2}}(\bar \rho, \bar v, \bar w|\tilde{\rho}, \tilde{v}, \tilde{w})+||(\bar \rho, \bar u)-(\tilde{\rho}, \tilde{u})||_{H^3([0,T] \times \R)} \leq C, 
\end{align}
for all $T>0$. In fact, this is the unique planar strong solution with initial data $(\bar \rho^0, \bar u^0)$ (uniqueness may be proven by a relative entropy method, see e.g. a small modification of \cite[Proposition 5.1]{MR2368905}). Further, it may be proven that
\begin{lemma}\label{lem:appclem}
Let $(\bar \rho_R, \bar u_R)$ be the planar strong solutions to \eqref{eq:2DNS_T_R} with initial data given by \eqref{eq:reflectdef2} constructed in \Cref{prop:1dexist+bds}. Let $(\bar{\rho}, \bar{u})$ be the unique planar strong solution to \eqref{eq:2DNS} with initial data $(\bar \rho^0, \bar u^0)$. Then we have the convergence
\begin{align*}
&||\nabla \bar v_R||_{L^\infty([0,T] \times \T_R)} \to ||\nabla \bar v||_{L^\infty([0,T] \times \R \times \T)}, \\
&||\nabla \bar w_R||_{L^\infty([0,T] \times \T_R)} \to ||\nabla \bar w||_{L^\infty([0,T] \times \R \times \T)}, \\
&||\Delta \ln(\bar \rho_R)||_{L^\infty([0,T] \times \T_R)} \to ||\Delta \ln(\bar \rho)||_{L^\infty([0,T] \times \R \times \T)}, \\
&||\nabla \ln(\bar \rho_R)||_{L^\infty([0,T] \times \T_R)} \to ||\nabla \ln(\bar \rho)||_{L^\infty([0,T] \times \R \times \T)}.
\end{align*}
\end{lemma}
A proof of \Cref{lem:appclem} is given in \Cref{sec:normsapp}.

By convexity, strong convergence of $(\bar{\rho}_R, \bar{v}_R, \bar{w}_R)$, \eqref{eq:reflectedinitproperty2}, and \Cref{lem:appclem}, we may pass into the limit $R \to \infty$ in \eqref{eq:weakstrongestimateonTR} to obtain
\begin{align*}
&\sup_{t \in [0,T]}\int_{\Omega_M}\rho\big(\frac{|v-\bar v|^2+|w-\bar w|^2}{2}\big)+F(\rho|\bar \rho)dx \leq \\
&Ce^{\big(||\nabla \overline{v}||_{L^\infty([0,T] \times \R)}+||\nabla \overline{w}||_{L^\infty([0,T] \times \R)}+||\Delta \ln(\overline{\rho})||_{L^\infty([0,T] \times \R)}+||\nabla \ln(\overline{\rho})||_{L^\infty([0,T] \times \R)}^2\big)t} \\
& \times \mathcal{E}_{\frac{1}{2}}(\rho_0, v_0, w_0|\bar{\rho}_0, \bar{v}_0, \bar{w}_0) \Big).
\end{align*}
Passing to the limit $M \to \infty$, we obtain exactly \eqref{eq:weakstrongestimatechannel}.
\par Finally we aim to prove that the limit $(\rho, v,w)$ has the right far-fields $(\rho_\pm, u_\pm)$ in the channel domain $\mathbb{R}\times\mathbb{T}$, i.e. we want to prove that \eqref{eq:correctff} holds. Note that we may write 
\begin{align*}
\mathcal{E}_{\frac{1}{2}}(\rho,v,w|\tilde{\rho},\tilde{v},\tilde{w})(t) &\leq C\Big( \mathcal{E}_{\frac{1}{2}}(\rho,v,w|\bar{\rho},\bar{v},\bar{w})(t)+\mathcal{E}_{\frac{1}{2}}(\bar \rho,\bar v,\bar w|\tilde{\rho},\tilde{v},\tilde{w})(t)\Big) \\
&+\int_{\R \times \T}(\rho-\bar \rho)\big(\frac{|\bar v-\tilde{v}|^2+|\bar w-\tilde{w}|^2}{2}\big)dx(t)-\int_{\R \times \T}(F'(\tilde{\rho})-F'(\bar \rho))(\rho-\bar \rho)dx(t).
\end{align*}
Note that the first two terms are controlled by \eqref{eq:weakstrongestimatechannel} and \eqref{eq:strongsolbds}. It remains to control the latter two. For the first, noting that $\bar \rho \geq c > 0$, and using \eqref{eq:strongsolbds}, we use Young's inequality
\begin{align*}
&\int_{\R \times \T}(\rho-\bar \rho)\big(\frac{|\bar v-\tilde{v}|^2+|\bar w-\tilde{w}|^2}{2}\big)dx(t) \\
&\leq C\Big( \mathcal{E}_{\frac{1}{2}}(\rho, v, w|\bar \rho, \bar v, \bar w)(t)+||(\bar v-\tilde{v}, \bar w-\tilde{w})||_{L^\infty([0,T];L^4(\R \times \T))}^4\Big) < \infty.
\end{align*}
For the second, using the fact that $\tilde{\rho}, \bar \rho$ are both bounded from above and below away from zero, and using \Cref{lem:relquantityests}
\begin{align*}
\int_{\R \times \T}(F'(\tilde{\rho})-F'(\bar \rho))(\rho-\bar \rho)dx(t) &\leq C\Big(||\tilde{\rho}-\bar{\rho}||_2\mathcal{E}_{\frac{1}{2}}(\rho, v, w|\bar{\rho},\bar{v},\bar{w})(t)^{\frac{1}{2}}\\
&+||\tilde{\rho}-\bar{\rho}||_{\frac{\gamma}{\gamma-1}}\mathcal{E}_{\frac{1}{2}}(\rho, v, w|\bar{\rho},\bar{v},\bar{w})(t)^{\frac{1}{\gamma}}\Big).
\end{align*}
This completes the proof of \Cref{thm:intro1}.

\section{Vanishing viscosity limit of 2D compressible Navier-Stokes equations}\label{sec:vvlimit}
\renewcommand{\theequation}{\arabic{section}.\arabic{equation}}
\setcounter{equation}{0}
The rest of the paper is devoted to proving \Cref{thm:intro2}. Let $(\rho^\nu_0, v^\nu_0, w^\nu_0)$ be sequence of initial data for \eqref{eq:2DNS} as in \Cref{thm:intro2}. Let $(\rho^\nu, v^\nu, w^\nu)$ be a sequence of weak solutions to \eqref{eq:2DNS} with initial data $(\rho^\nu_0, v^\nu_0, w^\nu_0)$ as given by \Cref{thm:intro1}. Further, let $(\bar{\rho}^\nu, \bar{v}^\nu, \bar{w}^\nu)$ be the unique planar solutions to \eqref{eq:2DNS} with initial data $(\overline{\rho}_0^\nu, \overline{v}_0^\nu, \overline{w}_0^\nu)$ given by \Cref{prop:1dexist+bds}. 
\par We have already proven the required estimate on the relative entropy between the $2D$ weak solution $(\rho^\nu, v^\nu, w^\nu)$ and the planar strong solution  $(\bar{\rho}^\nu, \bar{v}^\nu, \bar{w}^\nu)$. The next step is to obtain an estimate on the relative entropy between the planar strong solution $(\bar{\rho}^\nu, \bar{v}^\nu, \bar{w}^\nu)$ and the inviscid planar shock $(\hat{\rho}, \hat{u})$. For this part, we first estimate the distance between the planar strong solution and the viscous shock associated to $(\hat{\rho}, \hat{u})$. We use the $a$-contraction estimate proven in \cite[Theorem 1.2]{KV-Inven}. However, the $a$-contraction estimate is given in the Lagrangian coordinates, so we must first consider the planar strong solution in the Lagrangian coordinates. 
\par Without loss of generality, let us assume that $(\hat{\rho}, \hat{u})$ is a $1-$shock, so that $\rho_- < \rho_+$.

\subsection{1-D uniform-in-$\nu$ estimates: from Lagrangian to Eulerian}
Given the planar strong solution $(\bar{\rho}^\nu, \bar{v}^\nu, \bar{w}^\nu)$ to \eqref{eq:2DNS}, consider the Euler-Lagrange change of variables in the first coordinate
\begin{align}\label{eq:ELchangeofcoords}
 y_1(t,x_1)= \int_{(0,0)}^{(t,x_1)}\overline{\rho}^\nu dx_1- \overline{\rho}^\nu \overline{u}^\nu dt.
\end{align}
Under this change of coordinates, we obtain a strong solution $(\bar{s}^\nu(t,y_1), \bar{u}^\nu(t,y_1))$ to the equation
\begin{equation}\label{1DNS_Lag}
\left\{
\begin{array}{ll}
    (\bar s^\nu)_t -(\bar u^\nu_1)_{y_1}=0, \qquad (t,y_1,x_2) \in \R^+ \times \mathbb{R} \times \T, \\[2mm]
     (\bar u^\nu_{1})_t+p^L(\bar s^\nu)_{y_1}=(2 \nu \frac{(\bar u^\nu_1)_{y_1}}{(\bar s^\nu)^2})_{y_1}, \\[2mm]
      (\bar u^\nu_2)_{t}=(\nu \frac{(\bar u^\nu_2)_{y_1}}{(\bar s^\nu)^2})_{y_1} ,\\[2mm]
    % (\bar\rho_R(0,x_1), (\bar\rho_R \bar u_R)(0,x_1))=(\bar\rho^0_R, \bar m^0_R)(x_1),
\end{array}
\right.
\end{equation}
where $\overline{s}^\nu(t,y_1)=\frac{1}{\overline{\rho}^\nu(t,x_1)}$ and $p^L(s)=s^{-\gamma}$. Note that the third equation for $\bar u_2^\nu$ is decoupled with the first two equations for $(\bar s^\nu, \bar u_1^\nu)$. 
\par Also corresponding to the inviscid Eulerian shock $(\hat{\rho}, \hat{u})$ is an inviscid Lagrangian shock $(\hat{s}, \hat{u})(y_1-\sigma't)$ given by the equation
\begin{align}\label{eq:inviscidlagrangianshock}
(\hat{s}, \hat{u})(y_1)=\begin{cases}
    (s_-, u_{1,-}, 0) & y_1 < 0, \\
    (s_+, u_{1,+}, 0) & y_1 > 0,
\end{cases}
\end{align}
where
\begin{align*}
s_-:=\frac{1}{\rho_-}, \\
s_+:=\frac{1}{\rho_+},
\end{align*}
and$(\hat{s}, \hat{u})(y_1-\sigma't)$ is a solution to \eqref{1DNS_Lag} with $\nu=0$ for an appropriate speed $\sigma' \in \R$ given by the Rankine-Hugoniot conditions. %The speed $\sigma'$ will not be important for us, so we will neglect to define it.
\par Finally, associated to $\eqref{eq:inviscidlagrangianshock}$ are the corresponding planar viscous profiles $(\hat{s}^\nu, \hat{u}^\nu)(y_1-\sigma't)$ which verify, for the same speed $\sigma'$ (see e.g. \cite{MR2785875}),
\begin{align}\label{eq:viscouslagrangianshock}
-\sigma'(\hat{s}^\nu)'-(\hat{u}_1)'&=0, \\
-\sigma'(\hat{u}_1^\nu)'+p^L(\hat{s}^\nu)'&=2\nu(\frac{1}{(\hat{s}^\nu)^2}(\hat{u}_1^\nu)')', \\
\hat{u}_2^\nu&=0, \\
\lim_{\xi \to \infty}(\hat{s}^\nu, \hat{u}^\nu)(\xi)&=(\frac{1}{\rho_{\pm}}, u_{1,\pm}, 0).
\end{align}
The $a$-contraction estimate in \cite{KV-Inven} allows to estimate the difference in an appropriate relative functional $E(\cdot, \cdot)$, uniformly in $\nu$, between the solution $(\overline{s}^\nu, \overline{u}_1^\nu)$ and the viscous profile $(\hat{s}^\nu, \hat{u}_1^\nu)$ up to suitable time-dependent shift $X^\nu(t)$. Here, the relative entropy in the Lagrangian frame is
\begin{align*}
E_\nu((s_1,u_1)|(s_2,v_2)):=\frac{1}{2}\big(u_1+2\nu(s_1^{-1})_{y_1}-u_2-2\nu(s_2^{-1})_{y_1})^2+Q(s_1|s_2),
\end{align*}
where 
\begin{align*}
Q(s):=\frac{s^{-\gamma+1}}{\gamma-1}, \ s > 0,
\end{align*}
and $Q(s_1|s_2)$ is the relative quantity (recall \eqref{eq:relativequantity}). 
\par First, we need to verify that the initial values $(\bar{s}^\nu_0, \bar{u}^\nu_0)$ are suitable to apply the $a$-contraction estimate. This is content of the following lemma.
\begin{lemma}\label{lem:initialdataapprox}
Let $(\overline{\rho}_0^\nu, \overline{v}_0^\nu, \overline{w}_0^\nu)$ be initial data as in \Cref{thm:intro2}, giving rise to planar strong solutions $(\bar{\rho}^\nu, \bar{v}^\nu, \bar{w}^\nu)$ to \eqref{eq:2DNS}. Let $(\bar{s}^\nu_0, \bar{u}^\nu_0)$ be the corresponding initial data for the strong solution to \eqref{1DNS_Lag}, given by the formula
\begin{align*}
(\bar{s}^\nu_0(y_1), \bar{u}^\nu_0(y_1))=(\frac{1}{\bar{\rho}_0^\nu(x_1)}, \bar{u}^\nu_0(x_1)),
\end{align*}
under the change of coordinates corresponding to \eqref{eq:ELchangeofcoords} at $t=0$:
\begin{align*}
\frac{dy_1}{dx_1}(0,x_1)=\bar{\rho}_0^\nu(x_1), \ y_1(0,0)=0.
\end{align*}
Finally, let $(\hat{s}^\nu, \hat{u}^\nu)$ be the viscous profile solving \eqref{eq:viscouslagrangianshock}. Then, there exists a constant $C$ depending only on $\hat{s}$ and $\gamma$ so that (up to a subsequence)
\begin{align*}
\lim_{\nu \to 0}\int_{\R}E_\nu(\overline{s}^\nu_0,(\overline{u}^\nu_{1})_0|\hat{s}, \hat{u}^\nu_1)dy_1 \leq CE_0.
\end{align*}
\end{lemma}
\begin{proof}
Thanks to \eqref{eq:nearplanar}, making the change of variables \eqref{eq:ELchangeofcoords} gives, up to a subsequence
\begin{align}\label{eq:lem4step1}
\lim_{\nu \to 0}\int_{\R}\frac{|(\overline{u}^\nu_{1})_0-\hat{u}_1|^2}{2}+Q(\bar{s}_0^\nu|\hat{s}) \leq \frac{1}{\gamma-1}E_0.
\end{align}
First, we replace the inviscid shock $(\hat{s}, \hat{u})$ by the viscous profile in the above estimate. We have
\begin{align*}
Q(\bar{s}_0^\nu|\hat{s})-Q(\bar{s}_0^\nu|\hat{s}^\nu)=Q(\hat{s}^\nu|\hat{s})+(Q'(\hat{s}^\nu)-Q'(\hat{s}))[(\bar{s}^\nu_0-\hat{s})+(\hat{s}-\hat{s}^\nu)].
\end{align*}
Then, we compute
\begin{align*}
\lim_{\nu \to 0}\int_{\R}|Q(\bar{s}_0^\nu|\hat{s})-Q(\bar{s}_0^\nu|\hat{s}^\nu)|dy_1&\leq \lim_{\nu \to 0}\int_{\R}|Q(\hat{s}^\nu|\hat{s})+(Q'(\hat{s}^\nu)-Q'(\hat{s}))(\hat{s}-\hat{s}^\nu)|dy_1 \\
&+\lim_{\nu \to 0}\int_{\R}|(Q'(\hat{s}^\nu)-Q'(\hat{s}))(\bar{s}^\nu_0-\hat{s})|dy_1.
\end{align*}
The first term is equal to $0$ thanks to the fact that $||\hat{s}^\nu-\hat{s}||_{L^p(\R)}=\nu ||\hat{s}^1-\hat{s}||_{L^p(\R)}$, for any $1 \leq p < \infty$ and the fact that $\hat{s}^\nu, \hat{s}$ are bounded above and below away from $0$ uniformly in $\nu$. For the second term, we have
\begin{align*}
\lim_{\nu \to 0}\int_{\R}(Q'(\hat{s}^\nu)-Q'(\hat{s}))(\bar{s}^\nu_0-\hat{s})dy_1=&\lim_{\nu \to 0}\int_{\R}(Q'(\hat{s}^\nu)-Q'(\hat{s}))(\bar{s}^\nu_0-\hat{s})\mathbbm{1}_{\{\bar{s}^\nu_0 \leq \hat{s} \}}dy_1 \\
&+\lim_{\nu \to 0}\int_{\R}(Q'(\hat{s}^\nu)-Q'(\hat{s}))(\bar{s}^\nu_0-\hat{s})\mathbbm{1}_{\{\bar{s}^\nu_0 > \hat{s} \}}dy_1.
\end{align*}
The first term above equals zero again for the same reasons as above. For the second term, using the fact that 
\begin{align*}
\lim_{\bar{s}_0^\nu \to \infty}\frac{(\bar{s}^\nu_0-\hat{s})}{Q(\bar{s}_0^\nu|\hat{s})}=\frac{1}{-Q'(\hat{s})},
\end{align*}
we have 
\begin{align*}
\int_{\R}|(Q'(\hat{s}^\nu)-Q'(\hat{s}))(\bar{s}^\nu_0-\hat{s})\mathbbm{1}_{\{\bar{s}^\nu_0 > \hat{s} \}}|dy_1 &\leq C||Q'(\hat{s}^\nu)-Q'(\hat{s})||_{L^\infty(\R)}||Q(\bar{s}_0^\nu|\hat{s})||_{L^1(\R)} \\
& \leq CE_0.
\end{align*}
Combining all the above contributions gives (up to a subsequence)
\begin{align}\label{eq:replaceQ}
\lim_{\nu \to 0}\int_{\R}|Q(\bar{s}_0^\nu|\hat{s})-Q(\bar{s}_0^\nu|\hat{s}^\nu)|dy_1 \leq CE_0.
\end{align}
By a similar analysis, we have
\begin{align}\label{eq:replaceu}
\lim_{\nu \to 0}\int_{\R}[\frac{|(\overline{u}^\nu_{1})_0-\hat{u}_1^\nu|^2}{2}-\frac{|(\overline{u}^\nu_{1})_0-\hat{u}_1|^2}{2}]dy_1=0.
\end{align}
Plugging \eqref{eq:replaceQ} and \eqref{eq:replaceu} into \eqref{eq:lem4step1} gives
\begin{align}\label{eq:lem4step2}
\lim_{\nu \to 0}\int_{\R}\frac{|(\overline{u}^\nu_{1})_0-\hat{u}_1^\nu|^2}{2}+Q(\bar{s}_0^\nu|\hat{s}^\nu)dy_1 \leq CE_0.
\end{align}
Finally, we have
\begin{align*}\label{eq:lem4finalstep}
\int_{\R}\frac{|(\overline{u}^\nu_{1})_0+2\nu (\frac{1}{\bar{s}^\nu_0})_{y_1}-\hat{u}_1^\nu-2\nu (\frac{1}{\hat{s}^\nu})_{y_1}|^2}{2}dy_1 &\leq \int_{\R}|(\overline{u}^\nu_{1})_0-\hat{u}_1^\nu|^2dy_1 \\
&+2\int_{\R}|\nu (\frac{1}{\bar{s}^\nu_0})_{y_1}-\nu(\frac{1}{\hat{s}^\nu})_{y_1})|^2dy_1.
\end{align*}
Note that the first term is controlled already, while the second converges to $0$ as $\nu \to 0$ thanks to \eqref{eq:nearplanar}. Thus, up to a subsequence we have
\begin{align*}
\lim_{\nu \to 0}\int_{\R}\frac{|(\overline{u}^\nu_{1})_0+2\nu (\frac{1}{\bar{s}^\nu_0})_{y_1}-\hat{u}_1^\nu-2\nu (\frac{1}{\hat{s}^\nu})_{y_1}|^2}{2}+Q(\bar{s}_0^\nu|\hat{s}^\nu)dy_1 \leq CE_0,
\end{align*}
as desired.
\end{proof}

Now, we may apply the $a$-contraction estimate. Let $\eps_0$ in \Cref{thm:intro2} be small enough so that $|p^L(\frac{1}{\rho_-})-p^L(\frac{1}{\rho_+})|$ is sufficiently small for \cite[Theorem 1.2]{KV-Inven} to be applied. Then, applying  \cite[Theorem 1.2]{KV-Inven} to the functions
\begin{align*}
&\underline{s}(t,y_1)=\bar{s}^\nu(\nu t, \nu y_1), \ \underline{u}(t,y_1)=\bar{u}^\nu(\nu t, \nu y_1), \ \underline{\hat{s}}(t,y_1)=\hat{s}^\nu(\nu y_1), \\
& \underline{\hat{u}}(t,y_1)=\hat{u}^\nu(\nu y_1),
\end{align*}
and changing variables $t \mapsto \frac{t}{\nu}, y_1 \mapsto \frac{y_1}{\nu}$ as in \cite[page 120]{KV-Inven}, we obtain 
\begin{equation}\label{KV-est_0}
\begin{array}{ll}
\di \sup_{t\in[0,T]}\int_{\mathbb{R}} \Big[|\bar h_1-\hat{h}_1^{\nu, X_1^\nu(\cdot)}|^2+Q(\bar s^\nu| \hat{s}^{\nu, X_1^\nu(\cdot)})\Big](t,y_1)d y_1 \leq C\int_{\R}E_\nu(\overline{s}^\nu_0,(\overline{u}^\nu_{1})_0|\hat{s}^\nu, \hat{u}^\nu_1)dy_1,
\end{array}
\end{equation}
for some suitable shift function $X^\nu_1(t)$, where
$$\bar h_1^\nu:=\bar u_1^\nu+2\nu ((\bar{s}^\nu)^{-1})_{y_1},\qquad \hat{h}^\nu_1:=\hat{u}_1^\nu+2\nu ((\hat{s}^\nu)^{-1})_{y_1},
$$
and we have used the notation
\begin{align}\label{eq:shiftednotation}
f(t,y_1)^{X^\nu_1(t)}:=f(t,y_1-X^\nu_1(t)).
\end{align}
By \Cref{lem:initialdataapprox}, this implies that there exists $\nu^*$ such that for any $0 < \nu < \nu^*$
\begin{equation}\label{KV-est}
\begin{array}{ll}
\di \sup_{t\in[0,T]}\int_{\mathbb{R}} \Big[|\bar h_1^\nu-\hat{h}_1^{\nu, X_1^\nu(\cdot)}|^2+Q(\bar s^\nu| \hat{s}^{\nu, X_1^\nu(\cdot)})\Big](t,y_1)d y_1 \leq CE_0.
\end{array}
\end{equation}
Further, multiplying the third equation in \eqref{1DNS_Lag} by $\bar{u}^\nu_2$ and integrating in $y_1$ and $t$, we obtain
\begin{equation}\label{u2-est}
\begin{array}{ll}
\di \sup_{t\in[0,T]}\int_{\mathbb{R}} |\bar u^\nu_2|^2(t,y_1)dy_1 +\nu \int_0^{+\infty}\int_{\mathbb{R}}\frac{|(\bar u^\nu_2)_{y_1}|^2}{ (\bar s^\nu)^{2}} dy_1dt
&\di \leq C\int_\R|(\bar{u}^\nu_2)_0|^2dy_1 \\
&\di =C\int_\R\bar{\rho}^\nu_0|(\bar{u}^\nu_2)_0|^2dx_1 \\
&\leq CE_0,
\end{array}
\end{equation}
where the last inequality follows from \eqref{eq:nearplanar}.
Summing \eqref{KV-est} and \eqref{u2-est} gives
\begin{align}\label{eq:KV+u2-est}
\sup_{t\in[0,T]}\int_{\mathbb{R}} \Big[|\bar h_1^\nu-\hat{h}_1^{\nu, X_1^\nu(\cdot)}|^2+ |\bar u^\nu_2|^2+Q(\bar s^\nu| \hat{s}^{\nu, X_1^\nu(\cdot)})\Big](t,y_1)d y_1 \leq CE_0.
\end{align}
\par Now, we would like to replace the viscous profile in \eqref{eq:KV+u2-est} with the inviscid Lagrangian shock $(\hat{s}, \hat{u})$. This is needed to apply the Euler-Lagrange change-of-variables to attain an analogous estimate in the Eulerian frame. Note that we cannot apply the change of variables to \eqref{eq:KV+u2-est} directly because under the change of variables associated to $(\bar{s}^\nu, \bar{u}^\nu)$, the image of the viscous Lagrangian shock $(\hat{s}^\nu, \hat{u}^\nu)$ has no interpretation. The change-of-variables only behaves nicely for the piecewise constant inviscid shock $(\hat{s}, \hat{u})$. 

\par By the properties of the viscous shock, it holds that
\begin{equation}\label{ss}
\begin{array}{ll}
\di \int_{\mathbb{R}} |\hat{h_1^\nu}^{X_1^\nu(\cdot)}-\hat{u_{1}}^{X_1^\nu(\cdot)}|^2(t,y_1)d y_1 \\[3mm]
\di \leq 2\int_{\mathbb{R}} |\hat{u_1}^{\nu, X_1^\nu(\cdot)}-\hat{u_{1}}^{X_1^\nu(\cdot)}|^2(t,y_1)d y_1 +4\nu^2\int_{\mathbb{R}} |(\frac{1}{\hat{s}^{\nu, X_1^\nu(\cdot)}})_{y_1}|^2(t,y_1)d y_1\\[3mm]
\di \leq C\nu,
\end{array}
\end{equation}
and 
\begin{equation}\label{ss-1}
\begin{array}{ll}
\di \int_{\mathbb{R}} Q(\hat{s}^{\nu, X_1^\nu(\cdot)}| \hat{s}^{X_1^\nu(\cdot)})(t,y_1)d y_1 \\[3mm]
\di \leq C\int_{\mathbb{R}} |\hat{s}^{\nu, X_1^\nu(\cdot)}-\hat{s}^{X_1^\nu(\cdot)}|^2(t,y_1)d y_1  \leq C\nu.
\end{array}
\end{equation}
By \eqref{KV-est} and \eqref{ss}, we have
\begin{equation}\label{KV-est-2}
\begin{array}{ll}
\di \sup_{t \in [0,T]}\di \int_{\mathbb{R}}|\bar h_1^\nu-\hat{h}_1^{X_1^\nu(\cdot)}|^2(t,y_1)dy_1 \leq C\Big(E_0+\nu\Big).
\end{array}
\end{equation}
Since
\begin{equation}\label{KV-est-21}
\begin{array}{ll}
\di Q(\bar s^\nu| \hat{s}^{X_1^\nu(t)})\\[3mm]
\di =Q(\bar s^\nu| \hat{s}^{\nu,X_1^\nu(t)})+Q( \hat{s}^{\nu,X_1^\nu(t)}|\hat{s}^{X_1^\nu(t)})+\Big(Q^\prime(\hat{s}^{\nu,X_1^\nu(t)})-Q^\prime(\hat{s}^{X_1^\nu(t)})\Big)(\bar s^\nu- \hat{s}^{\nu,X_1^\nu(t)}).
\end{array}
\end{equation}
While
\begin{equation}\label{KV-est-22}
\begin{array}{ll}
\di \int_{\mathbb{R}}
\Big(Q^\prime(\hat{s}^{\nu,X_1^\nu(\cdot)})-Q^\prime(\hat{s}^{X_1^\nu(\cdot)})\Big)(\bar s^\nu- \hat{s}^{\nu,X_1^\nu(t\cdot})(t,y_1)d y_1\\[3mm]
\di =\int_{\mathbb{R}}
\Big(Q^\prime(\hat{s}^{\nu,X_1^\nu(\cdot)})-Q^\prime(\hat{s}^{X_1^\nu(\cdot)})\Big)(\bar s^\nu- \hat{s}^{\nu,X_1^\nu(\cdot)})\Big[\mathbbm{1}_{\{0\leq \bar s^\nu\leq \frac{s_+}{2}\}}\\[3mm]
\di \hspace{6cm}+\mathbbm{1}_{\{\frac{s_+}{2}\leq \bar s^\nu\leq 2s_-\}}+\mathbbm{1}_{\{\bar s^\nu\geq 2s_-\}}\Big](t,y_1)d y_1\\[3mm]
\di =I_1+I_2+I_3.
\end{array}
\end{equation}
Since 
$$
\lim_{\bar s^\nu\rightarrow0}\frac{|\bar s^\nu- \hat{s}^{\nu,X_1^\nu(t)}|}{Q(\bar s^\nu| \hat{s}^{\nu,X_1^\nu(t)})}=0,
$$
we have
\begin{equation}\label{KV-est-23}
\begin{array}{ll}
\di |I_1|\leq C\|Q^\prime(\hat{s}^{\nu,X_1^\nu(t)})-Q^\prime(\hat{s}^{X_1^\nu(t)})\|_{L^\infty(\mathbb{R})}\|(\bar s^\nu- \hat{s}^{\nu,X_1^\nu(t)})\mathbbm{1}_{\{0\leq \bar s^\nu\leq \frac{s_+}{2}\}}\|_{L^1(\mathbb{R})}\\[3mm]
\di \qquad \leq C\|Q(\bar s^\nu| \hat{s}^{\nu,X_1^\nu(t)})\|_{L^1(\mathbb{R})}.
\end{array}
\end{equation}
\begin{equation}\label{KV-est-24}
\begin{array}{ll}
\di |I_2|\leq C\|Q^\prime(\hat{s}^{\nu,X_1^\nu(t)})-Q^\prime(\hat{s}^{X_1^\nu(t)})\|_{L^2(\mathbb{R})}\|(\bar s^\nu- \hat{s}^{\nu,X_1^\nu(t)})\mathbbm{1}_{\{\frac{s_+}{2}\leq \bar s^\nu\leq 2s_-\}}\|_{L^2(\mathbb{R})}\\[3mm]
\di \qquad \leq C\|\hat{s}^{\nu,X_1^\nu(t)}-\hat{s}^{X_1^\nu(t)}\|_{L^2(\mathbb{R})}\|Q(\bar s^\nu| \hat{s}^{\nu,X_1^\nu(t)})\|_{L^1(\mathbb{R})}^{\frac12}\\[3mm]
\di \qquad \leq C\nu^{\frac12}\|Q(\bar s^\nu| \hat{s}^{\nu,X_1^\nu(t)})\|_{L^1(\mathbb{R})}^{\frac12}.
\end{array}
\end{equation}
Since 
$$
\lim_{\bar s^\nu\rightarrow+\infty}\frac{|\bar s^\nu- \hat{s}^{\nu,X_1^\nu(t)}|}{Q(\bar s^\nu| \hat{s}^{\nu,X_1^\nu(t)})}=-\frac{1}{Q^\prime(\hat{s}^{\nu,X_1^\nu(t)})},
$$
which  is positive and bounded below and above, then it holds that
\begin{equation}\label{KV-est-25}
\begin{array}{ll}
\di |I_3|\leq C\|Q^\prime(\hat{s}^{\nu,X_1^\nu(t)})-Q^\prime(\hat{s}^{X_1^\nu(t)})\|_{L^\infty(\mathbb{R})}\|(\bar s^\nu- \hat{s}^{\nu,X_1^\nu(t)})\mathbbm{1}_{\{\bar s^\nu\geq 2s_-\}}\|_{L^1(\mathbb{R})}\\[3mm]
\di \qquad \leq C\|Q(\bar s^\nu| \hat{s}^{\nu,X_1^\nu(t)})\|_{L^1(\mathbb{R})}.
\end{array}
\end{equation}
By \eqref{KV-est}, \eqref{ss-1}, \eqref{KV-est-21}-\eqref{KV-est-25}, we have
\begin{equation}\label{KV-est-26}
\begin{array}{ll}
\di \int_{\mathbb{R}} Q(\bar s^\nu| \hat{s}^{X_1^\nu(\cdot)})(t,y_1)d y_1\\[3mm]
\di =\int_{\mathbb{R}} \Big[Q(\bar s^\nu| \hat{s}^{\nu,X_1^\nu(\cdot)})+Q( \hat{s}^{\nu,X_1^\nu(\cdot)})+\Big(Q^\prime(\hat{s}^{\nu,X_1^\nu(\cdot)})-Q^\prime(\hat{s}^{X_1^\nu(\cdot)})\Big)(\bar s^\nu- \hat{s}^{\nu,X_1^\nu(\cdot)})\Big](t,y_1)d y_1\\[3mm]
\leq C\Big(E_0+\nu\Big).
\end{array}
\end{equation}
Collecting \eqref{KV-est}, \eqref{u2-est},  \eqref{KV-est-2} and \eqref{KV-est-26}, we have
\begin{equation}\label{Lag-est}
\begin{array}{ll}
\di \sup_{t\in[0,T]}\int_{\mathbb{R}} \Big[|\bar h_1^\nu-\hat{u_1}^{X_1^\nu(\cdot)}|^2+|\bar u_2^\nu|^2+Q(\bar s^\nu| \hat{s}^{X_1^\nu(\cdot)})\Big](t,y_1)d y_1 \leq C\Big(E_0+\nu\Big).
\end{array}
\end{equation}

Turning back to Eulerian coordinate transformation (recall \eqref{eq:ELchangeofcoords})
\begin{equation}\label{E-L}
x_1(t, y_1):=\oint_{(0,0)}^{(t, y_1)} \bar s^\nu dy_1+\bar u_1^\nu dt, 
\end{equation}
we have
\begin{equation}\label{Euler-est}
\begin{array}{ll}
\di \sup_{t\in[0,T]}\int_{\mathbb{R}} \Big[\bar \rho^\nu|\bar v_1^\nu+\bar{w}_1^\nu-\hat{u}_1^{X_E^\nu(\cdot)}|^2+\bar \rho^\nu|\bar u_2^\nu|^2+F(\bar \rho^\nu| \hat{\rho}^{X_E^\nu(\cdot)})\Big](t,x_1)d x_1 \leq C\Big(E_0+\nu\Big),
\end{array}
\end{equation}
where
\begin{align}\label{h1-E}
\bar v_1^\nu&:=\bar u_1^\nu+\nu \frac{\bar \rho^\nu_{x_1}}{\bar\rho^\nu}, \\
\bar w_1^\nu&:=\nu \frac{\bar \rho^\nu_{x_1}}{\bar\rho^\nu},
\end{align}
and $X_E^\nu(t)$ is the image of $X^\nu_1(t)$ under the change of coordinates \eqref{E-L}.

\subsection{Estimates on the shift $X^\nu_E(t)$}
Now, we want to show that as $\nu\rightarrow 0$,
\begin{equation}\label{shift-E}
|X_E^\nu(t)-\sigma t| \leq \frac{C}{|\rho_+-\rho_-|^2} \Big(\big(E_0+\nu\big)^{\frac12}+(E_0+\nu)\Big), \quad \forall t\in [0, T],
\end{equation}
where $\sigma=\frac{\rho_-u_{1,-}-\rho_+u_{1,+}}{\rho_--\rho_+}$ is the shock speed in Eulerian coordinates determined by Rankine-Hugoniot condition \eqref{eq:rhcond}. Recall that $(\bar{\rho}^\nu, \bar{u}^\nu)$ is a solution to the system
\begin{equation}\label{1DNS_E}
\left\{
\begin{array}{ll}
    (\bar \rho^\nu)_t +(\bar\rho^\nu \bar u_{1}^\nu)_{x_1}=0, \qquad\quad (t,x_1) \in \R^+ \times \mathbb{R}, \\[2mm]
    (\bar\rho^\nu \bar u_{1}^\nu)_{t}+\big(\bar \rho^\nu (\bar u_1^\nu)^2+\bar p(\bar \rho^\nu) \big)_{x_1}=(2 \nu\bar\rho^\nu (\bar u_1^\nu)_{x_1})_{x_1}, \\[2mm]
      (\bar\rho^\nu \bar u_2^\nu)_{t}+\big(\bar \rho^\nu \bar u_1^\nu \bar u_2^\nu \big)_{x_1}=(\nu \bar\rho^\nu (\bar u_2^\nu)_{x_1})_{x_1} .\\[2mm]
    % (\bar\rho_R(0,x_1), (\bar\rho_R \bar u_R)(0,x_1))=(\bar\rho^0_R, \bar m^0_R)(x_1),
\end{array}
\right.
\end{equation}
To prove \eqref{shift-E}, we use the mass equation $\eqref{1DNS_E}_1$. Define the test function $\varphi_M(x_1)\in W^{2, \infty}(\mathbb{R})$ with the support $\text{supp}( \varphi_M)=[-2M, 2M]$ and $M>0$ being a sufficiently large constant to be determined
\begin{equation}\label{test-func}
\varphi_M(x_1):=\left\{
\begin{array}{ll}
 \di 0,    & \di x_1\in (-\infty, -2M], \\[2mm]
\di -2(\frac{x_1}{M}+1)^3-3(\frac{x_1}{M}+1)^2+1, & \di x_1\in [-2M, -M], \\[2mm]
\di 1,  & \di x_1\in [-M, M], \\[2mm]
\di 2(\frac{x_1}{M}-1)^3-3(\frac{x_1}{M}-1)^2+1,   & \di x_1\in [M, 2M], \\[2mm]
\di 0   & \di x_1\in [2M, +\infty),
\end{array}
\right.
\end{equation}
and we can compute that 
\begin{equation}\label{test-func-1}
\varphi^\prime_M(x_1)=\frac{-6(x_1+M)(x_1+2M)}{M^3} \mathbbm{1}_{[-2M, -M]}+\frac{6(x_1-M)(x_1-2M)}{M^3}\mathbbm{1}_{[M, 2M]}\in W^{1,\infty}(\mathbb{R}),
\end{equation}
and
\begin{equation}\label{test-func-2}
\varphi^{\prime\prime}_M(x_1)=\frac{-6(2x_1+3M)}{M^3} \mathbbm{1}_{[-2M, -M]}+\frac{6(2x_1-3M)}{M^3}\mathbbm{1}_{[M, 2M]}\in L^\infty(\mathbb{R}).
\end{equation}
By the mass equation  $\eqref{1DNS_E}_1$ and the relation \eqref{h1-E}, we have
\begin{equation}\label{mass-new}
 (\bar \rho^\nu)_t +(\bar\rho^\nu \bar h_{1}^\nu)_{x_1}=2\nu \bar\rho^\nu_{x_1x_1}.
\end{equation}
Multiplying  equation \eqref{mass-new} by the test function $\varphi_M(x_1)$ gives that
\begin{equation}\label{mass-new-test}
 (\bar{\rho}^\nu \varphi_M)_t +(\bar\rho\bar{h}^\nu_{1}\varphi_M)_{x_1}=2\nu (\bar{\rho}^\nu\varphi_M)_{x_1x_1}-4\nu(\bar{\rho}^\nu  \varphi^\prime_M)_{x_1}+\varphi_M^\prime \bar{\rho}^\nu\bar{h}^\nu_{1}+2\nu \varphi_M^{\prime \prime}\bar{\rho}^\nu.
\end{equation}
Integrating the equation \eqref{mass-new-test} with respect to $(\tau, x_1)$ over $[0,t]\times \mathbb{R}$ with $t\in[0,T]$ implies that
\begin{equation}\label{eq1}
\begin{array}{ll}
\di\int_{\mathbb{R}}\bar{\rho}^\nu (t, x_1)\varphi_M(x_1)dx_1- \int_{\mathbb{R}}\bar{\rho}^\nu_0(x_1)\varphi_M(x_1)dx_1+\int_0^t(\bar{\rho}^\nu\bar{h}^\nu_{1}\varphi_M)\Big|^{x_1=+\infty}_{x_1=-\infty}d\tau\\[3mm]
\di =2\nu\int_0^t(\bar{\rho}^\nu \varphi_M)_{x_1}\Big|^{x_1=+\infty}_{x_1=-\infty}d\tau-4\nu\int_0^t\bar{\rho}^\nu \varphi_M^\prime\Big|^{x_1=+\infty}_{x_1=-\infty}d\tau\\[3mm]
\di \quad +\int_0^t\int_{\mathbb{R}}\bar{\rho}^\nu\bar{h}^\nu_{1}\varphi^\prime_Mdx_1 d\tau+2\nu \int_0^t\int_{\mathbb{R}}\bar{\rho}^\nu\varphi^{\prime\prime}_Md x_1 d\tau.
\end{array}
\end{equation}
Due to the compact support of both $\varphi_M$ and $\varphi_M^\prime$, all the boundary terms in \eqref{eq1} vanish. Therefore, we can rewrite \eqref{eq1} as
\begin{equation}\label{eq2}
\begin{array}{ll}
\di\int_{\mathbb{R}}(\bar{\rho}^\nu-\hat{\rho}^{X_E^\nu(t)})\varphi_M(x_1)dx_1- \int_{\mathbb{R}}(\bar{\rho}^\nu_0-\hat{\rho})\varphi_M(x_1)dx_1+\int_{\mathbb{R}}\varphi_M(x_1)(\hat{\rho}^{X_E^\nu(t)}-\hat{\rho})dx_1\\[3mm]
\di =\int_0^t\int_{\mathbb{R}}\big(\bar{\rho}^\nu\bar{h}^\nu_{1}-\hat{\rho}^{X_E^\nu(t)} \hat{u}^{X_E^\nu(t)}_1\big)\varphi^\prime_Mdx_1 d\tau+\int_0^t\int_{\mathbb{R}}\hat{\rho}^{X_E^\nu(t)} \hat{u}^{X_E^\nu(t)}_1\varphi^\prime_Mdx_1 d\tau\\[3mm]
\di\quad  +2\nu \int_0^t\int_{\mathbb{R}}\big(\bar{\rho}^\nu-\hat{\rho}^{X_E^\nu(t)}\big)\varphi^{\prime\prime}_Md x_1 d\tau+2\nu\int_0^t\int_{\mathbb{R}}\hat{\rho}^{X_E^\nu(t)}\varphi^{\prime\prime}_Md x_1 d\tau.
\end{array}
\end{equation}
First, we a priori assume that
\begin{equation}\label{M-ass}
    |X_E^\nu(t)|\leq M, \ \  \forall t\in[0,T].
\end{equation}
Then we can compute that
\begin{equation}\label{RH1}
\begin{array}{ll}
\di    \int_{\mathbb{R}}\varphi_M(x_1)(\hat{\rho}^{X_E^\nu(t)}-\hat{\rho})dx_1 \\[3mm]
\di =\rho_-\int_{-\infty}^{X_E^\nu(t)}\varphi_M(x_1) dx_1+\rho_+\int_{X_E^\nu(t)}^{+\infty}\varphi_M(x_1) dx_1\\[4mm]
\di \quad-\rho_-\int_{-\infty}^{0}\varphi_M(x_1) dx_1-\rho_+\int_{0}^{+\infty}\varphi_M(x_1) dx_1\\[3mm]
\di = \rho_-\int_{0}^{X_E^\nu(t)}\varphi_M(x_1) dx_1-\rho_+\int_0^{X_E^\nu(t)}\varphi_M(x_1) dx_1\\[3mm]
\di =(\rho_--\rho_+)X_E^\nu(t),
\end{array}
\end{equation}
and 
\begin{equation}\label{RH2}
\begin{array}{ll}
\di    \int_0^t\int_{\mathbb{R}}\varphi^\prime_M(x_1)\hat{\rho}^{X_E^\nu(t)}\hat{u}^{X_E^\nu(t)}_1dx_1d\tau \\[3mm]
\di =-\int_0^t\int_{-2M}^{-M}\rho_-u_{1,-}\frac{6(x_1+M)(x_1+2M)}{M^3}dx_1d\tau \\[4mm]
\di \quad+\int_0^t\int_{M}^{2M}\rho_+u_{1,+}\frac{6(x_1-M)(x_1-2M)}{M^3}dx_1d\tau \\[3mm]
\di =-\rho_-u_{1,-} t\int_{-2M}^{-M}\frac{6(x_1+M)(x_1+2M)}{M^3}dx_1\\[4mm]
\di \quad +\rho_+u_{1,+} t\int_{M}^{2M}\frac{6(x_1-M)(x_1-2M)}{M^3}dx_1\\[3mm]
\di =(\rho_-u_{1,-}-\rho_+u_{1,+})t.
\end{array}
\end{equation}
Substituting \eqref{RH1} and \eqref{RH2} into \eqref{eq2} yields that
\begin{equation}\label{shift-2}
    \begin{array}{ll}
       \di   \big(X_E^\nu(t)-\sigma t\big)(\rho_--\rho_+)\\[3mm]
       \di =-\int_{\mathbb{R}}(\bar{\rho}^\nu-\hat{\rho}^{X_E^\nu(t)})\varphi_M(x_1)dx_1+\int_{\mathbb{R}}(\bar{\rho}^\nu_0-\hat{\rho})\varphi_M(x_1)dx_1\\[3mm]
        \di \quad +\int_0^t\int_{\mathbb{R}}\big(\bar{\rho}^\nu\bar{h}^\nu_{1}-\hat{\rho}^{X_E^\nu(t)} \hat{u}^{X_E^\nu(t)}_1\big)\varphi^\prime_Mdx_1 d\tau\\[3mm]
        \di \quad+2\nu \int_0^t\int_{\mathbb{R}}\big(\bar{\rho}^\nu-\hat{\rho}^{X_E^\nu(t)}\big)\varphi^{\prime\prime}_Md x_1 d\tau+2\nu\int_0^t\int_{\mathbb{R}}\hat{\rho}^{X_E^\nu(t)}\varphi^{\prime\prime}_Md x_1 d\tau\\[3mm]
        \di :=\sum_{i=1}^5 J_i.
    \end{array}
\end{equation}
We will estimate $J_i~(i=1, ..., 5)$ as follows.
First, 
\begin{equation}\label{J1}
    \begin{array}{ll}
       \di J_1=-\int_{\mathbb{R}}(\bar{\rho}^\nu-\hat{\rho}^{X_E^\nu(t)})\varphi_M(x_1)dx_1\\[4mm]
        \di\quad\  =-\int_{-2M}^{2M}(\bar{\rho}^\nu-\hat{\rho}^{X_E^\nu(t)})\varphi_M(x_1)\Big[\mathbbm{1}_{\{0\leq \bar{\rho}^\nu\leq \frac{\rho_-}{2}\}}+\mathbbm{1}_{\{\frac{\rho_-}{2}\leq \bar{\rho}^\nu\leq 2\rho_+\}}+\mathbbm{1}_{\{\bar{\rho}^\nu\geq 2\rho_+\}}\Big]dx_1\\[4mm]
        \di\quad \  :=J_{11}+J_{12}+J_{13}.
    \end{array}
\end{equation}
Since 
$$
\lim_{\bar{\rho}^\nu\rightarrow0}\frac{|\bar{\rho}^\nu- \hat{\rho}^{X_E^\nu(t)}|}{F(\bar{\rho}^\nu| \hat{\rho}^{X_E^\nu(t)})}=\frac{1}{(\hat{\rho}^{X_E^\nu(t)})^{\gamma-1}},
$$
which is positively bounded below and above, 
we have
\begin{equation}\label{J11}
\begin{array}{ll}
\di |J_{11}|\leq \|(\bar{\rho}^\nu-\hat{\rho}^{X_E^\nu(t)})\mathbbm{1}_{\{0\leq \bar{\rho}^\nu\leq \frac{\rho_-}{2}\}}\|_{L^1(\mathbb{R})}\\[3mm]
\di \qquad\  \leq C\|F(\bar{\rho}^\nu| \hat{\rho}^{X_E^\nu(t)})\|_{L^1(\mathbb{R})}.
\end{array}
\end{equation}
And
\begin{equation}\label{J12}
\begin{array}{ll}
\di |J_{12}|\leq \|(\bar{\rho}^\nu-\hat{\rho}^{X_E^\nu(t)})\mathbbm{1}_{\{\frac{\rho_-}{2}\leq \bar{\rho}^\nu\leq 2\rho_+\}}\|_{L^2(\mathbb{R})}\|\varphi_M\|_{L^2[-2M, 2M]}\\[3mm]
\di \qquad\  \leq C M^{\frac12}\|F(\bar{\rho}^\nu| \hat{\rho}^{X_E^\nu(t)})\|_{L^1(\mathbb{R})}^{\frac12}.
\end{array}
\end{equation}
Since 
$$
\lim_{\bar{\rho}^\nu\rightarrow+\infty}\frac{|\bar{\rho}^\nu- \hat{\rho}^{X_E^\nu(t)}|}{F(\bar{\rho}^\nu| \hat{\rho}^{X_E^\nu(t)})}=0,
$$
we have
\begin{equation}\label{J13}
\begin{array}{ll}
\di |J_{13}|\leq \|(\bar{\rho}^\nu-\hat{\rho}^{X_E^\nu(t)})\mathbbm{1}_{\{\bar{\rho}^\nu\geq 2\rho_+\}}\|_{L^1(\mathbb{R})}\\[3mm]
\di \qquad\  \leq C\|F(\bar{\rho}^\nu| \hat{\rho}^{X_E^\nu(t)})\|_{L^1(\mathbb{R})}.
\end{array}
\end{equation}
Combining the above estimates \eqref{J11}-\eqref{J13} implies that
\begin{equation}\label{J1-e}
|J_1|\leq C\Big(\|F(\bar{\rho}^\nu| \hat{\rho}^{X_E^\nu(t)})\|_{L^1(\mathbb{R})}+M^{\frac12}\|F(\bar{\rho}^\nu| \hat{\rho}^{X_E^\nu(t)})\|_{L^1(\mathbb{R})}^{\frac12}\Big).
\end{equation}
Similarly, 
\begin{equation}\label{J2-e}
|J_2|\leq C\Big(\|F(\bar{\rho}^\nu_0| \hat{\rho})\|_{L^1(\mathbb{R})}+M^{\frac12}\|F(\bar{\rho}^\nu_0| \hat{\rho})\|_{L^1(\mathbb{R})}^{\frac12}\Big).
\end{equation}
Now we estimate $J_3$ as follows.

\begin{equation}\label{J3}
\begin{array}{ll}
\di J_{3}=\int_0^t\int_{\mathbb{R}}\bar{\rho}^\nu\big(\bar{h}^\nu_{1}- \hat{u}^{X_E^\nu(t)}_1\big)\varphi^\prime_Mdx_1 d\tau+\int_0^t\int_{\mathbb{R}} \hat{u}^{X_E^\nu(t)}_1\big(\bar{\rho}^\nu- \hat{\rho}^{X_E^\nu(t)}\big)\varphi^\prime_Mdx_1 d\tau\\[3mm]
\di\quad\  =\int_0^t\int_{\mathbb{R}}\sqrt{\bar{\rho}^\nu}\big(\bar{h}^\nu_{1}- \hat{u}^{X_E^\nu(t)}_1\big)\big(\sqrt{\bar{\rho}^\nu}- \sqrt{\hat{\rho}^{X_E^\nu(t)}}\big)\varphi^\prime_Mdx_1 d\tau\\[3mm]
\di \qquad \ +\int_0^t\int_{\mathbb{R}}\sqrt{\bar{\rho}^\nu}\big(\bar{h}^\nu_{1}- \hat{u}^{X_E^\nu(t)}_1\big)\sqrt{\hat{\rho}^{X_E^\nu(t)}}\varphi^\prime_Mdx_1 d\tau\\[3mm]
\di \qquad \ +\int_0^t\int_{\mathbb{R}} \hat{u}^{X_E^\nu(t)}_1\big(\bar{\rho}^\nu- \hat{\rho}^{X_E^\nu(t)}\big)\varphi^\prime_Mdx_1 d\tau\\[3mm]
\quad \ :=J_{31}+J_{32}+J_{33}.
\end{array}
\end{equation}
Then we have
\begin{equation}\label{J31}
\di |J_{31}|\leq \int_0^t\|\sqrt{\bar{\rho}^\nu}\big(\bar{h}^\nu_{1}- \hat{u}^{X_E^\nu(t)}_1\big)\|_{L^2(\mathbb{R})}\|\sqrt{\bar{\rho}^\nu}- \sqrt{\hat{\rho}^{X_E^\nu(t)}}\|_{L^2(\mathbb{R})}d\tau.
\end{equation}
while
\begin{equation}\label{J31-1}
\begin{array}{ll}
\di \|\sqrt{\bar{\rho}^\nu}- \sqrt{\hat{\rho}^{X_E^\nu(t)}}\|_{L^2(\mathbb{R})}\\[3mm]
\di =\left\|\Big(\sqrt{\bar{\rho}^\nu}- \sqrt{\hat{\rho}^{X_E^\nu(t)}}\Big)\Big[\mathbbm{1}_{\{0\leq \bar{\rho}^\nu\leq \frac{\rho_-}{2}\}}+\mathbbm{1}_{\{\frac{\rho_-}{2}\leq \bar{\rho}^\nu\leq 2\rho_+\}}+\mathbbm{1}_{\{\bar{\rho}^\nu\geq 2\rho_+\}}\Big]\right\|_{L^2(\mathbb{R})}\\[3mm]
\di \leq C\|F(\bar{\rho}^\nu| \hat{\rho}^{X_E^\nu(t)})\|_{L^1(\mathbb{R})}^{\frac12},
\end{array}
\end{equation}
due to the facts that
$$
\lim_{\bar{\rho}^\nu\rightarrow0}\frac{\Big(\sqrt{\bar{\rho}^\nu}- \sqrt{\hat{\rho}^{X_E^\nu(t)}}\Big)^2}{F(\bar{\rho}^\nu| \hat{\rho}^{X_E^\nu(t)})}=\frac{1}{\big(\hat{\rho}^{X_E^\nu(t)}\big)^{\gamma-1}},
$$
and
$$
\lim_{\bar{\rho}^\nu\rightarrow+\infty}\frac{\Big(\sqrt{\bar{\rho}^\nu}- \sqrt{\hat{\rho}^{X_E^\nu(t)}}\Big)^2}{F(\bar{\rho}^\nu| \hat{\rho}^{X_E^\nu(t)})}=0.
$$
Substituting \eqref{J31-1} into \eqref{J31} yields that
\begin{equation}\label{J31-e}
\di |J_{31}|\leq C\int_0^t\|\sqrt{\bar{\rho}^\nu}\big(\bar{h}^\nu_{1}- \hat{u}^{X_E^\nu(t)}_1\big)\|_{L^2(\mathbb{R})}\|F(\bar{\rho}^\nu| \hat{\rho}^{X_E^\nu(t)})\|_{L^1(\mathbb{R})}^{\frac12}d\tau.
\end{equation}
$J_{32}$ can be estimated as
\begin{equation}\label{J32-e}
\begin{array}{ll}
\di J_{32}= -\int_0^t\int_{-2M}^{-M}\sqrt{\bar{\rho}^\nu}\big(\bar{h}^\nu_{1}- \hat{u}^{X_E^\nu(t)}_1\big)\sqrt{\hat{\rho}^{X_E^\nu(t)}}\frac{6(x_1+M)(x_1+2M)}{M^3}dx_1 d\tau\\[4mm]
\di \qquad\ \   +\int_0^t\int_{M}^{2M}\sqrt{\bar{\rho}^\nu}\big(\bar{h}^\nu_{1}- \hat{u}^{X_E^\nu(t)}_1\big)\sqrt{\hat{\rho}^{X_E^\nu(t)}}\frac{6(x_1-M)(x_1-2M)}{M^3}dx_1 d\tau\\[4mm]
\di\quad\ \  \leq \frac{CT}{M^3}\sup_{t\in[0,T]}\|\sqrt{\bar{\rho}^\nu}\big(\bar{h}^\nu_{1}- \hat{u}^{X_E^\nu(t)}_1\big)\|_{L^2(\mathbb{R})}\Big[\|(x_1+M)(x_1+2M)\|_{L^2[-2M, -M]}\\[3mm]
\di \qquad \quad +\|(x_1-M)(x_1-2M)\|_{L^2[M, 2M]}\Big]\\[3mm]
\di \quad \ \ \leq C_T M^{-\frac12} \sup_{t\in[0,T]}\|\sqrt{\bar{\rho}^\nu}\big(\bar{h}^\nu_{1}- \hat{u}^{X_E^\nu(t)}_1\big)\|_{L^2(\mathbb{R})}.
\end{array}
\end{equation}
Then 
\begin{equation}\label{J33-e}
\begin{array}{ll}
\di J_{33}=\int_0^t\int_{\mathbb{R}} \hat{u}^{X_E^\nu(t)}_1\big(\bar{\rho}^\nu- \hat{\rho}^{X_E^\nu(t)}\big)\Big[\mathbbm{1}_{\{0\leq \bar{\rho}^\nu\leq \frac{\rho_-}{2}\}}+\mathbbm{1}_{\{\frac{\rho_-}{2}\leq \bar{\rho}^\nu\leq 2\rho_+\}}+\mathbbm{1}_{\{\bar{\rho}^\nu\geq 2\rho_+\}}\Big]\varphi^\prime_Mdx_1 d\tau\\[3mm]
\di \qquad \leq C\int_0^t \|\big(\bar{\rho}^\nu- \hat{\rho}^{X_E^\nu(t)}\big)\Big[\mathbbm{1}_{\{0\leq \bar{\rho}^\nu\leq \frac{\rho_-}{2}\}}+\mathbbm{1}_{\{\bar{\rho}^\nu\geq 2\rho_+\}}\Big]\|_{L^1({\mathbb{R})}} \|\varphi_M^\prime\|_{L^\infty[-2M, 2M]} d\tau\\[3mm]
\di \qquad\ \  +C\int_0^t\|\big(\bar{\rho}^\nu- \hat{\rho}^{X_E^\nu(t)}\big)\mathbbm{1}_{\{\frac{\rho_-}{2}\leq \bar{\rho}^\nu\leq 2\rho_+\}}\|_{L^2({\mathbb{R})}} \|\varphi_M^\prime\|_{L^2[-2M, 2M]} d\tau\\[3mm]
\di \qquad \leq CTM^{-1}\sup_{t\in[0,T]}\|p(\bar{\rho}^\nu| \hat{\rho}^{X_E^\nu(t)})\|_{L^1(\mathbb{R})}+ CTM^{-\frac 12}\sup_{t\in[0,T]}\|p(\bar{\rho}^\nu| \hat{\rho}^{X_E^\nu(t)})\|_{L^1(\mathbb{R})}^{\frac12}.
\end{array}
\end{equation}
On the other hand,
\begin{equation}\label{J4-e}
\begin{array}{ll}
\di J_4=
2\nu \int_0^t\int_{\mathbb{R}}\big(\bar{\rho}^\nu-\hat{\rho}^{X_E^\nu(t)}\big)\Big[\mathbbm{1}_{\{0\leq \bar{\rho}^\nu\leq \frac{\rho_-}{2}\}}+\mathbbm{1}_{\{\frac{\rho_-}{2}\leq \bar{\rho}^\nu\leq 2\rho_+\}}+\mathbbm{1}_{\{\bar{\rho}^\nu\geq 2\rho_+\}}\Big]\varphi^{\prime\prime}_Md x_1 d\tau\\[4mm]
\di \quad\ =-2\nu\int_0^t\int_{-2M}^{-M}\big(\bar{\rho}^\nu-\hat{\rho}^{X_E^\nu(t)}\big)\Big[\mathbbm{1}_{\{0\leq \bar{\rho}^\nu\leq \frac{\rho_-}{2}\}}+\mathbbm{1}_{\{\frac{\rho_-}{2}\leq \bar{\rho}^\nu\leq 2\rho_+\}}\\[4mm]
\di\qquad \qquad +\mathbbm{1}_{\{\bar{\rho}^\nu\geq 2\rho_+\}}\Big]\frac{6(2x_1+3M)}{M^3} d x_1 d\tau+2\nu\int_0^t\int_{M}^{2M}\big(\bar{\rho}^\nu-\hat{\rho}^{X_E^\nu(t)}\big)\Big[\mathbbm{1}_{\{0\leq \bar{\rho}^\nu\leq \frac{\rho_-}{2}\}}\\[4mm]
\di \qquad\qquad +\mathbbm{1}_{\{\frac{\rho_-}{2}\leq \bar{\rho}^\nu\leq 2\rho_+\}}+\mathbbm{1}_{\{\bar{\rho}^\nu\geq 2\rho_+\}}\Big]\frac{6(2x_1-3M)}{M^3} d x_1 d\tau\\[3mm]
\di \quad \leq 2\nu\int_0^t \|\big(\bar{\rho}^\nu-\hat{\rho}^{X_E^\nu(t)}\big)\Big[\mathbbm{1}_{\{0\leq \bar{\rho}^\nu\leq \frac{\rho_-}{2}\}}+\mathbbm{1}_{\{\bar{\rho}^\nu\geq 2\rho_+\}}\Big]\|_{L^1(\mathbb{R})}\Big[\|\frac{6(2x_1-3M)}{M^3}\|_{L^\infty[-2M, -M]}\\
\di \qquad +\|\frac{6(2x_1+3M)}{M^3}\|_{L^\infty[M, 2M]}\Big]d\tau+ 2\nu\int_0^t\|\big(\bar{\rho}^\nu-\hat{\rho}^{X_E^\nu(t)}\big)\mathbbm{1}_{\{\frac{\rho_-}{2}\leq \bar{\rho}^\nu\leq 2\rho_+\}}\|_{L^2(\mathbb{R})}\cdot\\[3mm]
\di \qquad \Big[\|\frac{6(2x_1-3M)}{M^3}\|_{L^2[-2M, -M]}+\|\frac{6(2x_1+3M)}{M^3}\|_{L^2[M, 2M]}\Big]d\tau \\[3mm]
\di \quad \leq C\nu  M^{-2} T\sup_{t\in[0,T]}\|F(\bar{\rho}^\nu| \hat{\rho}^{X_E^\nu(t)})\|_{L^1(\mathbb{R})} +C\nu M^{-\frac32} T  \sup_{t\in[0,T]}\|F(\bar{\rho}^\nu| \hat{\rho}^{X_E^\nu(t)})\|_{L^1(\mathbb{R})}^{\frac12}.
\end{array}
\end{equation}
\begin{equation}\label{J5-e}
\begin{array}{ll}
\di J_5=
2\nu\int_0^t\int_{\mathbb{R}}\hat{\rho}^{X_E^\nu(t)}\varphi^{\prime\prime}_Md x_1 d\tau\\[3mm]
\di \quad =-2\nu \int_0^t\int_{-2M}^{-M} \hat{\rho}^{X_E^\nu(t)}  \frac{6(2x_1+3M)}{M^3} dx_1d\tau+2\nu \int_0^t\int_{M}^{2M} \hat{\rho}^{X_E^\nu(t)}  \frac{6(2x_1-3M)}{M^3} dx_1d\tau\\[3mm]
\di \quad\leq C\nu \int_{-2M}^{-M}  \frac{6|2x_1+3M|}{M^3} dx_1d\tau+C\nu \int_0^t\int_{M}^{2M} \frac{6|2x_1-3M|}{M^3} dx_1d\tau\\[3mm]
\di \quad\leq C\nu M^{-1}T.
\end{array}
\end{equation}
Collecting all the above estimates together, we have $\forall t\in[0,T]$
\begin{equation}\label{shift-3}
 |X_E^\nu(t)-\sigma t|\leq \frac{C_T}{|\rho_+-\rho_-|} \Big[\big(E_0+\nu\big)+\big(E_0+\nu\big)^{\frac12}M^{\frac12}+\big(E_0+\nu\big)^{\frac12}\Big]. 
\end{equation}
if we take $M>1$.
Therefore, we have 
$\forall t\in[0,T]$
\begin{equation}\label{shift4}
 |X_E^\nu(t)|\leq |\sigma|T+ \frac{C_T\big(E_0+\nu\big)^{\frac12}}{|\rho_+-\rho_-|} \Big(M^{\frac12}+1\Big). 
\end{equation}
If we choose 
\begin{equation}\label{M-l}
 M=\max\left\{1, 2|\sigma|T+ 2\frac{C_T\big(E_0+\nu\big)^{\frac12}}{|\rho_+-\rho_-|} + \frac{|C_T|^2\big(E_0+\nu\big)}{(\rho_+-\rho_-)^2}\right\},
\end{equation}
then the a priori assumption \eqref{M-ass} is valid for $\forall t\in [0,T].$
By \eqref{shift-3}, 
\begin{equation}\label{shift-5}
 |X_E^\nu(t)-\sigma t|\leq \frac{C_T}{|\rho_+-\rho_-|^2} \Big(\big(E_0+\nu\big)^{\frac12}+(E_0+\nu)\Big). 
\end{equation}

\subsection{The relative entropy between 2D weak solution and the shifted shock}
By \eqref{eq:weakstrongestimatechannel}, we have
\begin{equation}\label{eq:2D-estimatebefore}
\begin{array}{ll}
&\di \sup_{t\in[0,T]}\int_{\mathbb{R}\times \mathbb{T}}\big[\rho^\nu(v_1^\nu-\bar v^\nu_1)^2+\rho^\nu(w_1^\nu-\bar w^\nu_1)^2+\rho^\nu |v_2^\nu-\bar u_2^\nu|^2+\rho^\nu |w_2^\nu|^2 +F(\rho^\nu|\bar \rho^\nu)\big]dx_1dx_2\\[3mm]
%+\nu\int_0^T \int_{\mathbb{R}\times \mathbb{T}}\rho [\nabla(u-\bar u)]^2 dx_1dx_2 dt
&\di \leq Ce^{\big(||\partial_{x_1}\bar{v}^\nu||_{L^\infty([0,T] \times \R)}+||\partial_{x_1} \bar{w}^\nu||_{L^\infty([0,T] \times \R)}+||\partial_{x_1x_1}\ln(\bar{\rho}^\nu)||_{L^\infty([0,T] \times \R)}+||\partial_{x_1}\ln(\bar{\rho}^\nu)||_{L^\infty([0,T] \times \R)}^2\big)T} \\
& \di \times \mathcal{E}_{\frac{1}{2}}(\rho_0^\nu, v_0^\nu, w_0^\nu|\bar{\rho}^\nu_0, \bar{v}^\nu_0, \bar{w}^\nu_0).
\end{array}
\end{equation}

Further, we have the following estimate on the norms of the planar strong solutions:
\begin{proposition}\label{prop:1doptimalbdsforuse}
There exists $C>0$ and $j \in \mathbb{N}$ so that the following holds. Let $(\rho^\nu_0, v^\nu_0, w^\nu_0)$ be as in \Cref{thm:intro2}. Let $(\bar{\rho}^\nu, \bar{u}^\nu)$ be the unique planar strong solution to \eqref{eq:2DNS} with initial data $(\bar{\rho}^\nu_0, \bar{u}^\nu_0)$. Then:
\begin{align*}
&\sup_{(t,x) \in [0,T] \times \R}\frac{1}{\bar{\rho}^\nu(t,x)}+||\bar{\rho}^\nu-\tilde{\rho}||_{L^\infty([0,T];H^3(\R))} \notag \\ 
&+||(\bar{u}_1^\nu, \bar{u}_2^\nu)-(\tilde{u}_1, \tilde{u}_2)||_{L^\infty([0,T];H^2(\R))} \leq \frac{C}{\nu^j}.
\end{align*}
\end{proposition}
A proof of \Cref{prop:1doptimalbdsforuse} is given in \Cref{sec:optimalbds}. Plugging the estimate of \Cref{prop:1doptimalbdsforuse} into \eqref{eq:2D-estimatebefore} gives 
\begin{equation}\label{2D-estimate}
\begin{array}{ll}
&\di \sup_{t\in[0,T]}\int_{\mathbb{R}\times \mathbb{T}}\big[\rho^\nu(v_1^\nu-\bar v^\nu_1)^2+\rho^\nu(w_1^\nu-\bar w^\nu_1)^2+\rho^\nu |v_2^\nu-\bar u_2^\nu|^2+\rho^\nu |w_2^\nu|^2+F(\rho^\nu|\bar \rho^\nu)\big]dx_1dx_2\\[3mm]
%+\nu\int_0^T \int_{\mathbb{R}\times \mathbb{T}}\rho [\nabla(u-\bar u)]^2 dx_1dx_2 dt
&\di \leq e^{\frac{C}{\nu^j}} \times \mathcal{E}_{\frac{1}{2}}(\rho_0^\nu, v_0^\nu, w_0^\nu|\bar{\rho}^\nu_0, \bar{v}^\nu_0, \bar{w}^\nu_0),
\end{array}
\end{equation}
for some integer $j > 0$.

By \eqref{Euler-est}, we have
\begin{equation}\label{Euler-est-1}
\begin{array}{ll}
\di \sup_{t\in[0,+\infty)}\int_{\mathbb{R}} \Big[\bar \rho^\nu|\bar v^\nu_1+\bar w^\nu_1-\hat{u}^{X_E^\nu(\cdot)}_1|^2+\bar \rho^\nu|\bar u_2^\nu|^2+F(\bar \rho^\nu| \hat{\rho}^{X_E^\nu(\cdot)})\Big](t,x_1)d x_1\leq  C\Big(E_0+\nu\Big).
\end{array}
\end{equation}
We would like to combine \eqref{2D-estimate} and \eqref{Euler-est-1} to obtain an estimate on
\begin{equation}\label{2D-e}
    \di \sup_{t\in[0,T]}\int_{\mathbb{R}\times \mathbb{T}}\big[\rho^\nu(v_1^\nu+w_1^\nu-\hat{u}^{X_E^\nu(t)}_1)^2+\rho^\nu |v_2^\nu|^2+\rho^\nu |w_2^\nu|^2 +F(\rho^\nu|\hat{\rho}^{X_E^\nu(t)})\big]dx_1dx_2.
\end{equation}
First, we prove a couple of estimates on the relative quantity $F(v|w)$ that will be crucial in the following proofs.
\begin{lemma}\label{lem:relquantityests}
For given constants $\gamma, M >1$, there exists a constant $C>0$ and $k_0 > 1$ such that the following inequalities hold.
\begin{enumerate}
    \item For any $k \geq 3M$ and $w \in (M^{-1},M)$,
    \begin{align}\label{eq:rellem1}
        \big[\max\big((k^{-1}-v)_+,(v-k)_+\big)\big]^\gamma \leq C\,F(v|w), \qquad \text{for any } v >0.
    \end{align}
    \item For any $w_1, w_2 \in (M^{-1}, M)$,
    \begin{align}\label{eq:rellem2}
        F(v|w_1) \leq C\,F(v|w_2), \qquad \text{for any } v \in (0, k_0^{-1}) \cup (k_0, \infty).
    \end{align}
    \item Further, for $w \in (M^{-1}, M)$, we have the pointwise estimate
    \begin{align}\label{eq:ptwiserelentest}
    |v-w| \leq C\big(\sqrt{F(v|w)}+F(v|w)^{\frac{1}{\gamma}}\big), \qquad \text{for any } v >0.
    \end{align}
    \item In particular, \eqref{eq:ptwiserelentest} implies
    \begin{align}\label{eq:controlbyrelent}
    v \leq C(1+F(v|w)), \qquad \text{for any } v >0.
    \end{align}
\end{enumerate}
Finally, in all of the above, the constant $C$ verifies
\begin{align}\label{eq:lemCest}
C \leq M^j, \text{for some } j \in \mathbb{N}.
\end{align}
\end{lemma}

\begin{proof}
Since $F(v) = v^\gamma/(\gamma-1)$, we have $F'(v) = \gamma v^{\gamma-1}/(\gamma-1)$, and hence
\begin{align*}
    F(v|w) = F(v) - F(w) - F'(w)(v-w) = \frac{v^\gamma - \gamma w^{\gamma-1} v + (\gamma-1)w^\gamma}{\gamma-1}.
\end{align*}
Note that \eqref{eq:rellem1} holds trivially for $v \in [k^{-1},k]$. So first assume $v < k^{-1}$. Then
\begin{align*}
    \big[\max\big((k^{-1}-v)_+,(v-k)_+\big)\big]^\gamma=(k^{-1}-v)^\gamma< k^{-\gamma}, \qquad 0 < v < k^{-1},
\end{align*}
and, using the identity above,
\begin{align*}
    F(v|w) = \frac{v^\gamma - \gamma w^{\gamma-1} v + (\gamma-1)w^\gamma}{\gamma-1} \geq M^{-\gamma} - \frac{\gamma M^{\gamma-1}}{(\gamma-1)k},
\end{align*}
for $w \in (M^{-1},M)$ and $v < k^{-1}$. Hence, taking
\begin{align*}
    C >\frac{2\big(2\gamma M^{2\gamma-1}\big)^{-\gamma}}{(\gamma-1)^2 M^{-\gamma}},
\end{align*}
we obtain \eqref{eq:rellem1} in this case. 

Next, assume $v > k$, so that
\begin{align*}
    \big[\max\big((k^{-1}-v)_+,(v-k)_+\big)\big]^\gamma = (v-k)^\gamma.
\end{align*}
Define $Z(v) = C F(v|w) - (v-k)^\gamma$. We show $Z(v) \geq 0$ on $[k,\infty)$ for $C$ chosen appropriately. At $v=k$, $Z(k) = C F(k|w) \geq 0$. Next, compute
\begin{align*}
    Z'(v) = \frac{C\gamma}{\gamma-1}\big(v^{\gamma-1} - w^{\gamma-1}\big) - \gamma (v-k)^{\gamma-1}.
\end{align*}
Using the fact that $v > k \geq 3M > 3w$,
\begin{align*}
    Z'(v) \geq \gamma v^{\gamma-1}\left(\frac{C}{\gamma-1}\big(1 - 3^{1-\gamma}\big) - 1\right).
\end{align*}
So, picking
\begin{align*}
    C > \max\left(\frac{2\big(2\gamma M^{2\gamma-1}\big)^{-\gamma}}{(\gamma-1)^2 M^{-\gamma}}, \ \frac{\gamma-1}{1-3^{1-\gamma}}\right),
\end{align*}
we obtain \eqref{eq:rellem1} in all cases.
\par Next, we prove \eqref{eq:rellem2}. We break the analysis into two regions. First, assume $v > k_0$. Then
\begin{align*}
    F(v|w) = \frac{v^\gamma}{\gamma-1}\left[1 - \gamma\left(\frac{w}{v}\right)^{\gamma-1} + (\gamma-1)\left(\frac{w}{v}\right)^\gamma\right].
\end{align*}
As $k_0 \to \infty$, the bracketed term converges to $1$ uniformly in $w \in (M^{-1},M)$. So, by taking $k_0$ sufficiently large (depending on $M$), we have
\begin{align*}
    \frac{1}{2}\cdot\frac{v^\gamma}{\gamma-1} \leq F(v|w) \leq \frac{3}{2}\cdot\frac{v^\gamma}{\gamma-1}, \qquad \text{for any } w \in (M^{-1},M),\ v > k_0.
\end{align*}
This gives \eqref{eq:rellem2} in this range.

Now assume $0 < v < k_0^{-1}$. As before,
\begin{align*}
    F(v|w) \geq M^{-\gamma} - \frac{\gamma M^{\gamma-1}}{(\gamma-1)k_0},
\end{align*}
for $w \in (M^{-1},M)$ and $v < k_0^{-1}$. So, taking $k_0$ sufficiently large, we have for all $w \in (M^{-1},M)$
\begin{align*}
    F(v|w) \geq \frac{1}{2}M^{-\gamma}.
\end{align*}
Further, noting that $F(v|w) \leq C_2$ for some constant $C_2$ (depending on $\gamma, M, k_0$) when $w \in (M^{-1},M)$ and $v < k_0^{-1}$, we have
\begin{align*}
    F(v|w_1) \leq C_2 \leq 2C_2M^\gamma\, F(v|w_2).
\end{align*}
This completes the proof of \eqref{eq:rellem2}.
\par Finally, we prove \eqref{eq:ptwiserelentest}. First, assume that $0 < v < \frac{1}{6M}$. Then, picking $C \geq 6M^2$ gives
\begin{align*}
|v-w|=w-v \leq M-v \leq \frac{C}{6M^2}(M-v)<\frac{C}{6M}\leq C\Big(\frac{1}{3M}-v\Big).
\end{align*}
Picking $k=3M$ and using \eqref{eq:rellem1}, we obtain \eqref{eq:ptwiserelentest} in this case. Next, assume that $v > 6M$. Then, for $C \geq 3$,
\begin{align*}
|v-w| \leq v+M \leq C\frac{v}{2} \leq C(v-3M).
\end{align*}
Picking $k=3M$ and using \eqref{eq:rellem1}, we obtain \eqref{eq:ptwiserelentest} in this case.
\par Finally, assume that $v \in [\frac{1}{6M}, 6M]$. Then, we have
\begin{align*}
F(v|w) \geq \inf_{z \in [\frac{1}{6M}, 6M]}\frac{F''(z)}{2}|v-w|^2.
\end{align*}
Re-arranging and taking $\di C \geq \big(\inf_{z \in [\frac{1}{6M}, 6M]}\big)^{-1}$, we obtain \eqref{eq:ptwiserelentest} in all cases. 
\par Finally, it is evident from the proof that \eqref{eq:lemCest} is verified.
\end{proof}

\par Now, let us begin the estimate of \eqref{2D-e}. First, 
\begin{equation}\label{2D-e1}
\begin{array}{ll}
    \di \rho^\nu(v_1^\nu+w_1^\nu-\hat{u}^{X_E^\nu(t)}_1)^2\leq 2\rho^\nu\Big[(v_1^\nu+w_1^\nu-\bar v^\nu_1-\bar w^\nu_1)^2+(\bar v^\nu_1+\bar w^\nu_1-\hat{u}^{X_E^\nu(t)}_1)^2\Big]\\[4mm]
    \di \leq 2\Big[\rho^\nu(v_1^\nu-\bar v^\nu_1)^2+\rho^\nu(w_1^\nu-\bar w^\nu_1)^2+\bar \rho^\nu(\bar v^\nu_1+\bar w^\nu_1-\hat{u}^{X_E^\nu(t)}_1)^2+(\rho^\nu-\bar \rho^\nu)(\bar v^\nu_1+\bar w^\nu_1-\hat{u}^{X_E^\nu(t)}_1)^2\Big].
\end{array}
\end{equation}
We need to estimate the last term in \eqref{2D-e1}. By \Cref{prop:1doptimalbdsforuse} and \Cref{lem:relquantityests}, we have
\begin{align*}
&\di \int_{\mathbb{R}\times \mathbb{T}}(\rho^\nu-\bar \rho^\nu)(\bar v^\nu_1+\bar w^\nu_1-\hat{u}^{X_E^\nu(t)}_1)^2dx_1dx_2\\[3mm]
&\leq C\nu^{-j}\int_{\mathbb{R}\times \mathbb{T}}\big(\sqrt{F(\rho^\nu|\bar \rho^\nu)}+F(\rho^\nu|\bar \rho^\nu)^\frac{1}{\gamma}\big)(\bar v^\nu_1+\bar w^\nu_1-\hat{u}^{X_E^\nu(t)}_1)^2dx_1dx_2.
\end{align*}
By Holder's inequality
\begin{align*}
&\int_{\mathbb{R}\times \mathbb{T}}\big(\sqrt{F(\rho^\nu|\bar \rho^\nu)}+F(\rho^\nu|\bar \rho^\nu)^\frac{1}{\gamma}\big)(\bar v^\nu_1+\bar w^\nu_1-\hat{u}^{X_E^\nu(t)}_1)^2dx_1dx_2 \\[3mm]
&\leq C\nu^{-j}\Big(||\sqrt{F(\rho^\nu|\bar \rho^\nu)}||_{L^2(\R \times \T)}||(\bar v^\nu_1+\bar w^\nu_1-\hat{u}^{X_E^\nu(t)}_1)^2||_{L^2(\R \times \T)}\\
&\qquad \qquad  +||F(\rho^\nu|\bar \rho^\nu)^\frac{1}{\gamma}||_{L^\gamma(\R \times \T)}||(\bar v^\nu_1+\bar w^\nu_1-\hat{u}^{X_E^\nu(t)}_1)^2||_{L^{\frac{\gamma}{\gamma-1}}(\R \times \T)}\Big).
\end{align*}
For the first term, we compute 
\begin{align*}
||\sqrt{F(\rho^\nu|\bar \rho^\nu)}||_{L^2(\R \times \T)}||(\bar v^\nu_1+\bar w^\nu_1-\hat{u}^{X_E^\nu(t)}_1)^2||_{L^2(\R \times \T)} &=||F(\rho^\nu|\bar \rho^\nu)||_{L^1(\R \times \T)}^{\frac12}||\bar v^\nu_1+\bar w^\nu_1-\hat{u}^{X_E^\nu(t)}_1||_{L^4(\R \times \T)}^2.
\end{align*}
Then, using \Cref{prop:1doptimalbdsforuse} and the fact that $\hat{u}_1^\nu$ is bounded, we obtain
$$
\begin{array}{ll}
\di \|\bar v^\nu_1+\bar w^\nu_1-\hat{u}^{X_E^\nu(t)}_1\|^2_{L^4(\mathbb{R})}\leq C\Big[\|\bar v^\nu_1+\bar w^\nu_1-\hat{u}_1^{\nu, X_E^\nu(t)}\|^2_{L^4(\mathbb{R})}+\|\hat{u}_1^{\nu, X_E^\nu(t)}-\hat{u}^{X_E^\nu(t)}_1\|^2_{L^4(\mathbb{R})}\Big]\\[3mm]
\di \qquad \leq C\nu^{-j}\Big[\|\bar v^\nu_1+\bar w^\nu_1-\hat{u}_1^{\nu, X_E^\nu(t)}\|_{L^2(\mathbb{R})}+\nu^{\frac12}\Big]\\[3mm]
\di \qquad \leq C\nu^{-j}(E_0+\nu)^{\frac12}.
\end{array}
$$
This gives 
$$
||\sqrt{F(\rho^\nu|\bar \rho^\nu)}||_{L^2(\R \times \T)}||(\bar v^\nu_1+\bar w^\nu_1-\hat{u}^{X_E^\nu(t)}_1)^2||_{L^2(\R \times \T)} \leq C\nu^{-j}(E_0+\nu)^{\frac12}||F(\rho^\nu|\bar \rho^\nu)||_{L^1(\R \times \T)}^{\frac12}.
$$
In a similar fashion, we estimate
\begin{align*}
||F(\rho^\nu|\bar \rho^\nu)^\frac{1}{\gamma}||_{L^\gamma(\R \times \T)}||(\bar v^\nu_1+\bar w^\nu_1-\hat{u}^{X_E^\nu(t)}_1)^2||_{L^{\frac{\gamma}{\gamma-1}}(\R \times \T)}&=\|F(\rho^\nu|\bar \rho^\nu)\|^{\frac1\gamma}_{L^1(\mathbb{R}\times\mathbb{T})}\|\bar v^\nu_1+\bar w^\nu_1-\hat{u}^{X_E^\nu(t)}_1\|^2_{L^{\frac{2\gamma}{\gamma-1}}(\mathbb{R})} \\
&\leq C\nu^{-j}(E_0+\nu)^{1-\frac{1}{\gamma}}\|F(\rho^\nu|\bar \rho^\nu)\|_{L^1(\R \times \T)}^{\frac1\gamma}.
\end{align*}
Altogether, we obtain
\begin{align*}
&\int_{\mathbb{R}\times \mathbb{T}}\big(\sqrt{F(\rho^\nu|\bar \rho^\nu)}+F(\rho^\nu|\bar \rho^\nu)^\frac{1}{\gamma}\big)(\bar v^\nu_1+\bar w^\nu_1-\hat{u}^{X_E^\nu(t)}_1)^2dx_1dx_2 \\[3mm]
& \leq C\nu^{-j}\max((E_0+\nu)^{\frac12},(E_0+\nu)^{1-\frac1\gamma})\max(\|F(\rho^\nu|\bar \rho^\nu)\|_{L^1(\R \times \T)}^{\frac12}, \|F(\rho^\nu|\bar \rho^\nu)\|_{L^1(\R \times \T)}^{\frac1\gamma}),
\end{align*}
which, by plugging this back into \eqref{2D-e1}, gives
\begin{equation}\label{rho-4}
\begin{array}{ll}
&\di \int_{\mathbb{R}\times \mathbb{T}}\rho^\nu |v_1^\nu+w_1^\nu-\hat{u}_1^{X^\nu_E(t)}|^2dx_1dx_2 \\[3mm]
&\di\leq C\int_{\mathbb{R}\times \mathbb{T}}\Big[\rho^\nu |v_1^\nu-\bar v_1^\nu|^2+\rho^\nu |w_1^\nu-\bar w_1^\nu|^2+\bar{\rho}^\nu
|\bar v_1^\nu+\bar w_1^\nu-\hat{u}_1^{X^\nu_E(t)}|^2\Big]dx_1dx_2 \\[3mm]
&\quad + C\nu^{-j}\max((E_0+\nu)^{\frac12},(E_0+\nu)^{1-\frac1\gamma})\max(\|F(\rho^\nu|\bar \rho^\nu)\|_{L^1(\R \times \T)}^{\frac12}, \|F(\rho^\nu|\bar \rho^\nu)\|_{L^1(\R \times \T)}^{\frac1\gamma}).
\end{array}    
\end{equation}
Similarly, we have
\begin{equation}\label{rho-5}
\begin{array}{ll}
\di \int_{\mathbb{R}\times \mathbb{T}}\rho^\nu |v_2^\nu|^2dx_1dx_2\leq C\int_{\mathbb{R}\times \mathbb{T}}\Big[\rho^\nu |v_2^\nu-\bar u_2^\nu|^2+\bar{\rho}^\nu
|\bar u_2^\nu|^2+(\rho^\nu-\bar{\rho}^\nu)|\bar u_2^\nu|^2 \Big]dx_1dx_2\\[3mm]
\di\leq C\int_{\mathbb{R}\times \mathbb{T}}\Big[\rho^\nu |v_2^\nu-\bar u_2^\nu|^2+\bar{\rho}^\nu
|\bar u_2^\nu|^2\Big]dx_1dx_2 \\[3mm]
\quad+C\nu^{-j}\max((E_0+\nu)^{\frac12},(E_0+\nu)^{1-\frac1\gamma})\max(\|F(\rho^\nu|\bar \rho^\nu)\|_{L^1(\R \times \T)}^{\frac12}, \|F(\rho^\nu|\bar \rho^\nu)\|_{L^1(\R \times \T)}^{\frac1\gamma}).
\end{array}
\end{equation}
For the term with the relative pressure, we have the relation
\begin{equation}\label{p1}
  F(\rho^\nu|\hat{\rho}^{X_E^\nu(t)})=F(\rho^\nu|\bar \rho^\nu)+F(\bar{\rho}^\nu|\hat{\rho}^{X_E^\nu(t)})+(\rho^\nu-\bar{\rho}^\nu)\big(F^\prime(\bar \rho^\nu)-F^\prime(\hat{\rho}^{X_E^\nu(t)})\big).  
\end{equation}
Again, the last term in \eqref{p1} needs to be estimated. Using \Cref{lem:relquantityests} twice, alongside the bounds in \Cref{prop:1doptimalbdsforuse} and $C^{-1} \leq |\hat{\rho}| \leq C$ gives
\begin{align*}
&\int_{\R \times \T}(\rho^\nu-\bar{\rho}^\nu)\big(F^\prime(\bar \rho^\nu)-F^\prime(\hat{\rho}^{X_E^\nu(t)})\big)dx_1dx_2 \\[3mm]
&\leq C\nu^{-j}\int_{\R \times \T}\big(\sqrt{F(\rho^\nu|\bar \rho^\nu)}+F(\rho^\nu|\bar \rho^\nu)^{\frac1\gamma}\big)\big(\sqrt{F(\bar \rho^\nu|\hat{\rho}^{X^\nu_E(t)})}+F(\bar \rho^\nu|\hat{\rho}^{X^\nu_E(t)})^{\frac1\gamma}\big)dx_1dx_2.
\end{align*}
For the first term, we estimate
\begin{align}\label{cont1}
&\int_{\R \times \T}\sqrt{F(\rho^\nu|\bar \rho^\nu)}\sqrt{F(\bar \rho^\nu|\hat{\rho}^{X^\nu_E(t)})}dx_1dx_2 \\[3mm]
&\leq ||F(\rho^\nu|\bar \rho^\nu)||_{L^1(\R \times \T)}^{\frac12}||F(\bar \rho^\nu|\hat{\rho}^{X^\nu_E(t)})||_{L^1(\R \times \T)}^{\frac12} \\[3mm]
&\leq (E_0+\nu)^{\frac12}||F(\rho^\nu|\bar \rho^\nu)||_{L^1(\R \times \T)}^{\frac12},
\end{align}
by \eqref{Euler-est-1}. For the second
\begin{align}\label{cont2}
&\int_{\R \times \T}\sqrt{F(\rho^\nu|\bar \rho^\nu)}F(\bar \rho^\nu|\hat{\rho}^{X^\nu_E(t)})^{\frac1\gamma}dx_1 dx_2 \\[3mm]
&\leq ||F(\rho^\nu|\bar \rho^\nu)||_{L^1(\R \times \T)}^{\frac12}||F(\bar \rho^\nu|\hat{\rho}^{X^\nu_E(t)})^{\frac2\gamma}||_{L^1(\R \times \T)}^{\frac12} \\[3mm]
&\leq C\nu^{-j}(E_0+\nu)^{\frac12}||F(\rho^\nu|\bar \rho^\nu)||_{L^1(\R \times \T)}^{\frac12},
\end{align}
where in the last inequality we have used the fact that $\frac2\gamma \geq 1$ for $\gamma \in (1,2]$ and $|\bar \rho^\nu|, |\hat{\rho}| \leq C\nu^{-j}$. For the third:
\begin{align}\label{cont3}
&\int_{\R \times \T}F(\rho^\nu|\bar \rho^\nu)^{\frac1\gamma}\sqrt{F(\bar \rho^\nu|\hat{\rho}^{X^\nu_E(t)})}dx_1 dx_2 \\[3mm]
&\leq ||F(\rho^\nu|\bar \rho^\nu)||_{L^1(\R \times \T)}^{\frac1\gamma}||F(\bar \rho^\nu|\hat{\rho}^{X^\nu_E(t)})^{\frac{\gamma}{2(\gamma-1)}}||_{L^1(\R \times \T)}^{1-\frac1\gamma} \\[3mm]
&\leq C\nu^{-j}(E_0+\nu)^{1-\frac1\gamma}||F(\rho^\nu|\bar \rho^\nu)||_{L^1(\R \times \T)}^{1-\frac1\gamma}, 
\end{align}
as $\frac{\gamma}{2(\gamma-1)} \geq 1$. Finally
\begin{align}\label{cont4}
&\int_{\R \times \T}F(\rho^\nu|\bar \rho^\nu)^{\frac1\gamma}F(\bar \rho^\nu|\hat{\rho}^{X^\nu_E(t)})^{\frac1\gamma}dx_1 dx_2 \\[3mm]
&\leq ||F(\rho^\nu|\bar \rho^\nu)||_{L^1(\R \times \T)}^{\frac1\gamma}||F(\bar \rho^\nu|\hat{\rho}^{X^\nu_E(t)})^{\frac{1}{\gamma-1}}||_{L^1(\R \times \T)}^{1-\frac1\gamma} \\[3mm]
&\leq C\nu^{-j}(E_0+\nu)^{1-\frac1\gamma}||F(\rho^\nu|\bar \rho^\nu)||_{L^1(\R \times \T)}^{1-\frac1\gamma}, 
\end{align}
by identical reasoning. Combining \eqref{cont1}-\eqref{cont4} and using \eqref{p1}, we obtain 
\begin{align}\label{p3}
&\int_{\mathbb{R}\times \mathbb{T}}(\rho^\nu-\bar{\rho}^\nu)\big(F'(\bar \rho^\nu)-F'(\hat{\rho}^{X_E^\nu(t)})\big)dx_1dx_2 \\[3mm]
&\leq C\nu^{-j}\max((E_0+\nu)^{\frac12},(E_0+\nu)^{1-\frac1\gamma})\max(\|F(\rho^\nu|\bar \rho^\nu)\|_{L^1(\R \times \T)}^{\frac12}, \|F(\rho^\nu|\bar \rho^\nu)\|_{L^1(\R \times \T)}^{\frac1\gamma}).
\end{align}
By  \eqref{2D-estimate}, \eqref{Euler-est-1}, \eqref{rho-4}, \eqref{rho-5}, and \eqref{p3}, we have
\begin{equation}\label{2D-est}
\begin{array}{ll}
    \di \sup_{t\in[0,T]}\int_{\mathbb{R}\times \mathbb{T}}\big[\rho^\nu(v_1^\nu+w_1^\nu-\hat{u}^{X_E^\nu(t)}_1)^2+\rho^\nu |v_2^\nu|^2 +F(\rho^\nu|\hat{\rho}^{X_E^\nu(t)})\big]dx_1dx_2\\[3mm]
    \di \leq C\sup_{t\in[0,T]}\int_{\mathbb{R}\times \mathbb{T}}\big[\rho^\nu(v_1^\nu-\bar v^\nu_1)^2+\rho^\nu |v_2^\nu-\bar u_2^\nu|^2+F(\rho^\nu|\bar \rho^\nu)\\[3mm]
    \di \qquad \qquad \qquad \quad\    +\bar \rho^\nu(\bar v^\nu_1-\hat{u}^{X_E^\nu(t)}_1)^2+\bar{\rho}^\nu
|\bar u_2^\nu|^2+p(\bar{\rho}^\nu|\hat{\rho}^{X_E^\nu(t)})\Big]dx_1dx_2\\[3mm]
\di \quad + C\nu^{-j}\max((E_0+\nu)^{\frac12},(E_0+\nu)^{1-\frac1\gamma})\max(\|F(\rho^\nu|\bar \rho^\nu)\|_{L^1(\R \times \T)}^{\frac12}, \|F(\rho^\nu|\bar \rho^\nu)\|_{L^1(\R \times \T)}^{\frac1\gamma})\\[3mm]
\di \leq C(E_0+\nu)+  C\nu^{-j}[\big(1+\max((E_0+\nu)^{\frac12},(E_0+\nu)^{1-\frac1\gamma})\big) \\[3mm]
\times \max(\sqrt{\mathcal{E}_{\frac{1}{2}}(\rho^\nu_0,v^\nu_0, w_0^\nu|\bar{\rho}^\nu_0,\bar v^\nu_0, \bar{w}^\nu_0)e^{\frac{C}{\nu^{j+1}}}}, (\mathcal{E}_{\frac{1}{2}}(\rho^\nu_0,v^\nu_0, w_0^\nu|\bar{\rho}^\nu_0,\bar v^\nu_0, \bar{w}^\nu_0)e^{\frac{C}{\nu^{j+1}}})^{\frac1\gamma}).
\end{array}
\end{equation}

\subsection{Proof of \Cref{thm:intro2}}
Using \eqref{shift-5} and  \eqref{2D-est}, we can take 2D weak inviscid limit as follows. 

\subsubsection{Convergence of $\sqrt{\rho^\nu}$} Now, by \eqref{2D-est}, there exists a constant $C>0$ and an integer $j > 0$ so that if
\begin{align*}
\mathcal{E}_{\frac{1}{2}}(\rho^\nu_0,v^\nu_0, w^\nu_0|\bar\rho^\nu_0,\bar v^\nu_0, \bar{w}^\nu_0) \leq \exp(-C\nu^{-j}),
\end{align*}
then we have 
\begin{equation}\label{2D-controlled}
\begin{array}{ll}
    \di \sup_{t\in[0,T]}\int_{\mathbb{R}\times \mathbb{T}}\big[\rho^\nu|v^\nu+w^\nu-\hat{u}^{E, X_E^\nu(t)}|^2+p(\rho^\nu|\hat{\rho}^{X_E^\nu(t)})\big]dx_1dx_2 \leq C(E_0+\nu)+E_1(\nu),
\end{array}
\end{equation}
where $E_1(\nu) \to 0$ as $\nu \to 0$. Set the function $C^*(\nu)$ in the statement of \Cref{thm:intro2} equal to $\exp(-C\nu^{-j})$ for this $C$ and $j$. 
\par Next, given the two end states $\rho_{\pm}$ of the shock $\hat{\rho}$, fix a constant $M >1$ so that
\begin{align*}
    (\min(\rho_-,\rho_+),\max(\rho_-,\rho_+)) \subset (M^{-1},M).
\end{align*}
Then, fix the constant $k_0 >1$ in \Cref{lem:relquantityests} and set
\begin{align*}
    K=\max(3M, k_0).
\end{align*}
Then, let $\overline{\psi}$ be a continuous function defined by
\begin{align*}
    \overline{\psi}(x)=\begin{cases}
        x, & \text{ if } K^{-1} \leq x \leq K, \\
        K^{-1}, & \text{ if } x < K^{-1}, \\
        K, & \text{ if } x > K.
    \end{cases}
\end{align*}
Then set
\begin{align*}
\underline{\rho}^\nu:=\overline{\psi}(\rho^\nu), \ \rho_e^\nu:=\rho^\nu-\underline{\rho}^\nu.
\end{align*}
Note that the truncation $\underline{\rho}^\nu$ will be used in the proof of \eqref{eq:finalstabilityest} later. Now, since
\begin{align*}
|\rho_e^\nu|=[\max\big((k^{-1}-\rho^\nu)_+,(\rho^\nu-k)_+\big)],
\end{align*}
and $M^{-1} < \hat{\rho}^{X^\nu_E} < M$, we use \eqref{eq:rellem1} to have
\begin{align*}
|\rho_e^\nu|^\gamma \leq CF(\rho^\nu|\hat{\rho}^{X^\nu_E}).
\end{align*}
It follows from \eqref{2D-est} that
\begin{align*}
\int_{\R \times \T}|\rho_e^\nu|^\gamma dx_1dx_2 \leq C(E_0+\nu)+E_1(\nu).
\end{align*}
Therefore, $\{\rho_e^\nu\}_{\nu > 0}$ is bounded in $L^\infty([0,T];L^\gamma(\R \times \T))$. Further, since $\{\underline{\rho}^\nu\}_{\nu > 0}$ is bounded in $L^\infty([0,T] \times \R \times \T) \subset L^\infty([0,T];L^\gamma_{\text{loc}}(\R \times \T))$, we obtain
\begin{align}\label{eq:rhonubd}
\{\sqrt{\rho^\nu}\}_{\nu > 0} \text{ is bounded in } L^\infty([0,T];L^{2\gamma}_{\text{loc}}(\R \times \T)).
\end{align}
Therefore, there exists $\sqrt{\rho_\infty}$ such that, up to a subsequence
\begin{align*}
\sqrt{\rho^\nu} \rightharpoonup \sqrt{\rho_\infty} \text{ in } L^\infty([0,T];L^{2\gamma}_{\text{loc}}(\R \times \T)).
\end{align*}
\subsubsection{Convergence of $\sqrt{\rho^\nu}(v^\nu+w^\nu)$ and $\sqrt{\rho^\nu}u^\nu$}
For convenience, let us denote by 
\begin{align}
h^\nu=v^\nu+w^\nu.
\end{align} 
Take the constant $K$ larger if needed so that 
\begin{align*}
    (\min(u_{1,-},u_{1,+}),\min(u_{1,-},u_{1,+})) \subset (\frac{-K}{2},\frac{K}{2}).
\end{align*}
Then set 
\begin{align*}
\overline{\phi}(x)=\begin{cases}
        x, & \text{ if } -K \leq x \leq K, \\
        -K, & \text{ if } x < -K, \\
        K, & \text{ if } x > K.
    \end{cases}
\end{align*}
and 
\begin{align*}
\underline{h}^\nu:=(\overline{\phi}(h_1^\nu), \overline{\phi}(h_2^\nu)), \ h_e^\nu:=h^\nu-\underline{h}^\nu.
\end{align*}
Again, the truncation $\underline{h}^\nu$ will be used in the proof of \eqref{eq:finalstabilityest} later. Now, since
\begin{align*}
|(h_{e}^\nu)_1|=\max\big((-h_1^\nu-K)_+,(h_1^\nu-K)_+\big),
\end{align*}
and $-K < \hat{u}_1^{\nu,X^\nu_E} < K$, we have
\begin{align*}
|h_e^\nu| \leq |h^\nu-\hat{u}^{\nu,X^\nu_E}|.
\end{align*}
It follows from \eqref{2D-controlled} that
\begin{align*}
\int_{\R \times \T}\rho^\nu |h_e^\nu|^2dx_1 dx_2 \leq C(E_0+\nu)+E_1(\nu).
\end{align*}
Therefore, $\{\sqrt{\rho^\nu}h_e^\nu\}_{\nu > 0}$ is bounded in $L^\infty([0,T];L^2(\R \times \T))$. Further, $\{\sqrt{\rho^\nu}\underline{h}^\nu\}_{\nu > 0}$ is bounded in $ L^\infty([0,T];L^{2\gamma}_{\text{loc}}(\R \times \T)) \hookrightarrow L^\infty([0,T];L^2_{\text{loc}}(\R \times \T))$. So, there exists $u_\infty$ so that up to a subsequence
\begin{align*}
\sqrt{\rho^\nu} h^\nu \stackrel{\ast}{\rightharpoonup} \sqrt{\rho_\infty} u_\infty \text{ in } L^\infty([0,T];L^2_{\text{loc}}(\R \times \T)).
\end{align*}
Finally, by \eqref{eq:rhonubd}, $\{\sqrt{\rho}^\nu\}_{\nu > 0}$ is bounded in $L^2_{\text{loc}}([0,T] \times \R \times \T)$. Using the identity \eqref{eq:vandwrelationsonchannel}, this implies $\sqrt{\rho^\nu}w^\nu=2\nu \nabla \sqrt{\rho^\nu}$ converges to $0$ in $H^{-1}_{\text{loc}}([0,T] \times \R \times \T)$. Using the linearity $\sqrt{\rho^\nu}u^\nu=\sqrt{\rho^\nu}v^\nu-\sqrt{\rho^\nu}w^\nu$ gives 
\begin{align*}
\sqrt{\rho^\nu} u^\nu \stackrel{\ast}{\rightharpoonup} \sqrt{\rho_\infty} u_\infty \text{ in } H^{-1}_{\text{loc}}([0,T] \times \R \times \T).
\end{align*}
This proves \eqref{eq:finalinviscidlimits}.

\subsubsection{Convergence of $X^\nu_E$ and proof of \eqref{eq:finalshiftest}}
By \eqref{shift-5}, $\{X^\nu_E\}_{\nu > 0}$ is uniformly bounded in $L^\infty([0,T])$. So, there exists $X_\infty \in L^\infty([0,T])$ so that, up to a subsequence
\begin{align*}
X^\nu_E  \stackrel{\ast}{\rightharpoonup} X_\infty \text{ in $L^\infty([0,T])$}.
\end{align*}
Further, by the weak-* lower semi-continuity of the $L^\infty$ norm, we immediately obtain \eqref{eq:finalshiftest} from \eqref{shift-5}.

\subsubsection{Proof of \eqref{eq:finalstabilityest}} 
For the truncations $\underline{\rho}^\nu, \underline{h}^\nu$, with $K$ fixed, let 
\begin{align*}
L_\nu(t):=\int_{\R \times \T}\eta(\underline{\rho}^\nu, \underline{h}^\nu|\hat{\rho}^{X^\nu_E}, \hat{u}^{X^\nu_E})dx.
\end{align*}
Using the definitions of the truncations together with \eqref{2D-controlled}, we have
\begin{align*}
||L_\nu||_{L^\infty([0,T])} \leq C(E_0+\nu)+E_1(\nu).
\end{align*}
Then, we have the following lemma.
\begin{lemma}\label{lem:shiftchange}
For the fixed constant $K>1$, let 
\begin{align*}
R_\nu(t):=\int_{\R \times \T}\eta(\underline{\rho}^\nu, \underline{h}^\nu|\hat{\rho}^{X_\infty}, \hat{u}^{X_\infty})dx.
\end{align*}
Then,
\begin{align*}
    \limsup_{\nu \to 0}||L_\nu-R_\nu||_{L^\infty([0,T])} \leq \frac{C}{|\rho_--\rho_+|^2}\Big(E_0+\sqrt{E_0}\Big).
\end{align*}
\end{lemma}
\begin{proof}
Since $\underline{h}^\nu, \hat{u}$ are bounded, we have
\begin{align*}
\Big||\underline{h}^\nu-\hat{u}^{X^\nu_E}|^2-|\underline{h}^\nu-\hat{u}^{X_\infty}|^2\Big| \leq C|\hat{u}^{X^\nu_E}-\hat{u}^{X_\infty}|.
\end{align*}
Thus, using $||\hat{u}^{X^\nu_E}-\hat{u}^{X_\infty}||_{L^1(\R \times \T)}=|u_{1,-}-u_{1,+}||X^\nu_E-X_\infty|$, we have
\begin{align*}
\int_{\R \times \T} \frac{\underline{\rho}^\nu}{2}\Big||\underline{h}^\nu-\hat{u}^{X^\nu_E}|^2-|\underline{h}^\nu-\hat{u}^{X_\infty}|^2\Big|dx \leq C||X^\nu_E-X_\infty||_{L^\infty([0,T])} \leq \frac{C}{|\rho_--\rho_+|^2}\Big(E_0+\sqrt{E_0}\Big).
\end{align*}
For the second part, again using the fact that $\underline{\rho}^\nu$ and $\hat{\rho}$ are bounded
\begin{align*}
&|F(\underline{\rho}^\nu|\hat{\rho}^{X^\nu_E})-F(\underline{\rho}^\nu|\hat{\rho}^{X_\infty})| \\
&\leq |F(\hat{\rho}^{X^\nu_E})-F(\hat{\rho}^{X_\infty})| \\
&+|F'(\hat{\rho}^{X^\nu_E})||\hat{\rho}^{X^\nu_E}-\hat{\rho}^{X_\infty}| \\
&+(|\underline{\rho}^\nu+|\hat{\rho}^{X_\infty}|)|F'(\hat{\rho}^{X^\nu_E})-F'(\hat{\rho}^{X_\infty})| \\
&\leq C|\hat{\rho}^{X^\nu_E}-\hat{\rho}^{X_\infty}|.
\end{align*}
So, we obtain the result by the same computations as above.
\end{proof}
Now, we consider 
\begin{align*}
&\int_{\R \times \T}\eta(\rho^\nu, h^\nu|\hat{\rho}^{X_\infty}, \hat{u}^{X_\infty})dx \\
&=\int_{\rho^\nu \in [K^{-1},K] \cap h^\nu \in [-K,K]} \underline{\rho}^\nu\frac{|\underline{h}^\nu-\hat{u}^{X_\infty}|^2}{2}dx \\
&+\int_{\rho^\nu \not \in [K^{-1},K] \cup h^\nu \not \in [-K,K]} \rho^\nu\frac{|h^\nu-\hat{u}^{X_\infty}|^2}{2}dx \\
&+\int_{\rho^\nu \in [K^{-1},K]}F(\underline{\rho}^\nu|\hat{\rho}^{X_\infty})dx \\
&+\int_{\rho^\nu \not \in [K^{-1},K]}F(\rho^\nu|\hat{\rho}^{X_\infty})dx \\
&=:\sum_{i=1}^4 J_i.
\end{align*}
Note that 
\begin{align*}
J_1+J_3 \leq R_\nu=(R_\nu-L_\nu)+L_\nu \leq (R_\nu-L_\nu)+C(E_0+\nu)+E_1(\nu).
\end{align*}
For $J_2$, since $\hat{u} \in [\frac{-K}{2}, \frac{K}{2}]$, we have
\begin{align*}
|h^\nu-\hat{u}^{X_\infty}| \leq 3 |h^\nu-\hat{u}^{X^\nu_E}| \text{ when } h^\nu \not \in [-K, K].
\end{align*}
So, using \eqref{2D-controlled}
\begin{align*}
J_2 \leq C(E_0+\nu+E_1(\nu)).
\end{align*}
For $J_4$, as $\hat{\rho} \in [M^{-1},M]$, recalling that $K \geq k_0$ and using \eqref{eq:rellem2} gives
\begin{align*}
J_4 \leq C(E_0+\nu+E_1(\nu)).
\end{align*}
In total, we obtain 
\begin{align}\label{eq:2Dreadyforlimit}
\int_{\R \times \T}\Big(\rho^{\nu}\frac{|h^\nu-\hat{u}^{X_\infty}|^2}{2}+F(\rho^\nu|\hat{\rho}^{X_\infty})\Big)dx \leq |R_\nu-L_\nu|+C(E_0+\nu+E_1(\nu)).
\end{align}
It remains to take the limit $\nu \to 0$ in \eqref{eq:2Dreadyforlimit}. By the weak lower semi-continuity of the $L^2$-norm, we may take the limit $\nu \to 0$ in the first term. Further, as $x \mapsto F(x|\hat{\rho}^{X_\infty})$ is convex, the second term is weakly lower semi-continuous as well. So, we have
\begin{align*}
&\int_{\R \times \T}\Big(\rho_\infty \frac{|u_\infty-\hat{u}^{X_\infty}|^2}{2}+F(\rho_\infty|\hat{\rho}^{X_\infty})\Big)dx \\
&\leq \liminf_{\nu \to 0}\int_{\R \times \T}\Big(\rho^{\nu}\frac{|h^\nu-\hat{u}^{X_\infty}|^2}{2}+F(\rho^\nu|\hat{\rho}^{X_\infty})\Big)dx \\
& \leq \frac{C}{|\rho_--\rho_+|^2}\Big(E_0+\sqrt{E_0}\Big),
\end{align*}
using \Cref{lem:shiftchange} and $E_1(\nu) \to 0$ as $\nu \to 0$, giving exactly \eqref{eq:finalstabilityest}. This completes the proof of \Cref{thm:intro2}.

\section{AI acknowledgement Statement}
A completed previous version of this manuscript was given to GPT-6.1 Sol to check for typos and expository imprecision. Afterwards, alongside fixing some minor typos, we added the statement of \Cref{lem:appclem} and the proof (in Appendix C) to clarify the exposition and give more details.

\appendix 
\section{The a priori $\frac{1}{2}$-entropy estimate}\label{sec:halfentropyestimatedetails}
\renewcommand{\theequation}{A.\arabic{equation}}
\setcounter{equation}{0}
In this section, we show how to obtain the $\frac{1}{2}$-entropy inequality \eqref{eq:1/2entropyestimate} for smooth solutions, slightly modifying the a priori estimates proved in \cite[Appendix A]{MR3801754}. Note that existence of weak solutions to \eqref{eq:2DQNS_T_R} has been proved in \cite{MR3490496}. The solution they construct is actually a renormalized solution \cite[Theorem 2.1]{MR3801754}. Their construction can be slightly modified to also obtain at $\frac{1}{2}$-entropy solution by proving the inequality \eqref{eq:1/2entropyestimate} at the level of the Faedo-Galerkin approximation for $\eps_0, \eps_1, \eps_2 >0$ and passing it into the limit. 
\par Consider a smooth solution of \eqref{eq:2DQNS_T_R}. Then, we have the following quantities
\begin{align*}
\mathcal{E}(\rho, u)&=\int_{\T_R}\big(\rho\frac{|u|^2}{2}+F(\rho)+2\eps_2|\nabla \sqrt{\rho}|^2\big)dx, \\
\mathcal{D}_{\mathcal{E}}(u)&=2\nu \int_{\T_R}|\mathcal{S}(u)|^2dx, \\
\mathcal{E}_{BD}(\rho, u)&=\int_{\T_R}\big(\rho \frac{|u+2\nu \nabla \ln(\rho)|^2}{2}+F(\rho)+2\eps_2|\nabla \sqrt{\rho}|^2\big)dx, \\
\mathcal{D}_{BD}(\rho, u)&=\nu \int_{\T_R}\big(\frac{8}{\gamma}|\nabla \rho^{\frac{\gamma}{2}}|^2+\eps_2 \rho|\nabla^2 \ln(\rho)|^2+2|\mathcal{A}(u)|^2\big)dx.
\end{align*}
Multiplying the continuity equation \eqref{eq:2DQNS_T_R}$_1$ by $\frac{\gamma \rho^{\gamma-1}}{\gamma-1}$ and the second equation by $u$, and integrating in space the sum of these quantities gives
\begin{align}\label{eq:appendixAenergy}
\partial_t\mathcal{E}(\rho, u)+\mathcal{D}_{\mathcal{E}}(u)+\eps_0\int_{\T_R}|u|^2dx+\eps_1\int_{\T_R}\rho|u|^4dx=0.
\end{align}
The equation on $k:=u+2\nu \nabla \ln(\rho)$ is 
\begin{align}\label{eq:bdeqn}
\partial_t(\rho k)+{\rm div}(\rho u \otimes k)+\nabla \rho^\gamma-2{\rm div}(\nu \rho \mathbb{A}u+\kappa \rho \mathbb{S}_{\eps_2})=-\eps_0u-\eps_1\rho|u|^2u.
\end{align}
Multiplying the equation \eqref{eq:bdeqn} by $k$, adding the continuity equation multiplied by $\frac{\gamma \rho^{\gamma-1}}{\gamma-1}+\eps_0 \nu$, and integrating in space gives:
\begin{align*}
&\partial_t\Big(\mathcal{E}_{BD}(\rho, u)+2\eps_0 \nu\int_{\T_R}(\rho-\ln(\rho))dx\Big)+\mathcal{D}_{BD}(\rho,u)+\eps_0\int_{\T_R}|u|^2dx+\eps_1\int_{\T_R}\rho|u|^4 \\
&=\eps_1\nu\int_{\T_R}2\sqrt{\rho}|u|^2u\cdot \nabla \sqrt{\rho}dx.
\end{align*}
Then, we estimate
\begin{align*}
|\eps_1\nu\int_{\T_R}2\sqrt{\rho}|u|^2u\cdot \nabla \sqrt{\rho}dx| \leq 3\nu \eps_1\int_{\T_R}\sqrt{\rho}|u|^2\sqrt{\rho}|\nabla u|dx.
\end{align*}
Using Young's inequality gives:
\begin{align}\label{eq:appendixABDentropy}
&\partial_t\Big(\mathcal{E}_{BD}(\rho, u)+2\eps_0 \nu\int_{\T_R}(\rho-\ln(\rho))dx\Big)+\mathcal{D}_{BD}(\rho,u)+\eps_0\int_{\T_R}|u|^2dx+(\eps_1-2\zeta)\int_{\T_R}\rho|u|^4 \\
&-2\zeta \int_{\T_R}|\mathcal{T}(u)|^2dx \leq 0,
\end{align}
where $\zeta$ is defined in $\eqref{eq:zetadef}$. Finally, defining $v=u+\nu \nabla \ln(\rho)$ and $w=\nu \nabla \ln(\rho)$, we note the identity
\begin{align*}
\frac{1}{2}|k|^2+\frac{1}{2}|u|^2=|v|^2+|w|^2.
\end{align*}
So, adding $\frac{1}{2}$ times \eqref{eq:appendixAenergy} and $\frac{1}{2}$ times \eqref{eq:appendixABDentropy} and integrating in time gives exactly \eqref{eq:1/2entropyestimate}.

\section{Proof of \Cref{prop:1dexist+bds}}\label{sec:appendix1dexistence}
	\renewcommand{\theequation}{B.\arabic{equation}}
\setcounter{equation}{0}
In this section, we prove \Cref{prop:1dexist+bds}. We work with the planar strong solution to the system 
\begin{equation}\label{eq:2D1DNS_T_R}
\left\{
\begin{array}{ll}
    (\bar \rho_R)_t +(\bar \rho_R \bar u_{R,1})_{x_1}=0,  \indent (t,x_1,x_2) \in \R^+ \times [-R,2R]_{\text{periodic}} \times \T, \\[2mm]
     (\bar\rho_R \bar u_{R,1})_t+(\bar\rho_R \bar u_{R,1}^2)_{x_1}+p(\bar\rho_R)_{x_1}=(2 \nu \bar\rho_R (\bar u_{R,1})_{x_1})_{x_1}, \\[2mm]
     (\bar\rho_R \bar u_{R,2})_t+(\bar\rho_R \bar u_{R,1}\bar u_{R,2})_{x_1}=(\nu \bar\rho_R (\bar u_{R,2})_{x_1})_{x_1},\\[2mm]
     (\bar\rho_R(0,x_1), (\bar\rho_R \bar u_R)(0,x_1))=((\bar \rho_R)_0, (\bar m_R)_0)(x_1).
\end{array}
\right.
\end{equation}
The initial value $((\bar \rho_R)_0, (\bar m_R)_0)$ verifies (uniformly in $R$):
\begin{alignat}{3}\label{eq:appendixinitdatabds}
&||((\bar \rho_R)_0, (\bar u_R)_0)-(\tilde{\rho}_R, \tilde{u}_R)||_{H^6(\T_R)} \leq 2\delta_3, \\
&|\partial_{x_1}(\bar u_{R,1})_{0}| \leq \frac{1}{2\nu}((\bar \rho_R)_0)^{\gamma-1}, \\
&(\bar \rho_R)_0 \geq c_0 > 0.
\end{alignat}
The global existence and smoothness for solutions such a system on fixed domains $\T \times \T$ and $\R \times \T$ with initial data verifying these bounds is a straightforward generalization of previous works (see \cite{MR2368905}, \cite{MR4049919}, \cite{MR4113341}) and clearly holds on $[-R,2R]_{\text{periodic}} \times \T$, but we redo the estimates here for completeness to show the constant $C(T,\nu, \delta_3, c_0)$ need not depend on $R$ if the initial data verifies the bounds \eqref{eq:appendixinitdatabds} uniformly in $R$. In particular, we are careful not to use any Poincar\'e inequality on $\T_R$ or bounds involving the $L^2$-norm of $(\bar \rho_R, \bar u_R)$. For brevity, we drop the $x_2$ direction, the bar notation, and the parameter $R$ and instead work with the following system
\begin{equation}\label{eq:1DNS_T_R}
\left\{
\begin{array}{ll}
    \rho_t +(\rho u_1)_{x}=0,  \indent (t,x) \in \R^+ \times [-R,2R]_{\text{periodic}}, \\[2mm]
     (\rho u_1)_t+(\rho u_1^2)_{x}+p(\rho)_{x}=(2 \nu \rho u_{1x})_{x}, \\[2mm]
     (\rho u_2)_t+(\rho u_1u_2)_{x}=(\nu \rho u_{2x})_{x},\\[2mm]
     (\rho(0,x), (\rho u)(0,x))=(\rho_0, m_0)(x),
\end{array}
\right.
\end{equation}
where $u(t,x)=\begin{bmatrix}
u_1(t,x) \\
u_2(t,x) 
\end{bmatrix}$. Note that the equation for $u_2$ in \eqref{eq:1DNS_T_R} is decoupled from the equations for $\rho$ and $u_1$. So the system \eqref{eq:1DNS_T_R} can mostly be treated as the usual $1$-D compressible Navier-Stokes equations.
\begin{notation}
For brevity, in the section we will use the following notation.
\begin{alignat*}{6}
    &U_0=(\rho_0, u_0), \\
    &U=(\rho, u), \\
    &\Pi:=[-R,2R]_{\text{periodic}}, \\
    &\tilde{S}=(\tilde{\rho}, \tilde{u}):=(\tilde{\rho}_R, \tilde{u}_R), \\
    &L^p_tL^q_x:=L^p([0,T];L^q(\Pi)) \text{ for $1 \leq p,q \leq \infty$,} \\
    &L^p_tH^q_x:=L^p([0,T];H^q(\Pi)) \text{ for $1 \leq p,q \leq \infty$.}
\end{alignat*}
We let $j$ denote an arbitrary positive integer that may vary line by line, that depends only on $\gamma$. Finally, we let $C$ denote an arbitrary constant that may vary line by line. The constant $C$ may depend on $\delta_3, c_0, ||(\tilde{\rho}, \tilde{u})||_{W^{1,\infty}(\Pi)}, T, \& \ \nu$, but not $R$.
\end{notation}
We will consider initial data $(\rho_0, u_0)$ verifying (by \eqref{eq:appendixinitdatabds}),
\begin{align}\label{eq:uniforminR1dinitdatabds}
&||(\rho_0, u_0)-(\tilde{\rho}, \tilde{u})||_{H^6(\Pi)} \leq 2 \delta_3, \\
&|\partial_{x}(u_1)_0| \leq \frac{1}{2\nu}(\rho^0)^{\gamma-1}, \\
&\rho_0(x) \geq c_0 > 0.
\end{align}
\Cref{prop:1dexist+bds} follows immediately from the following proposition for \eqref{eq:1DNS_T_R}. 
\begin{proposition}\label{prop:1dexist+bdsappendix}
Let $(\rho_0, u_0)$ verify the bounds \eqref{eq:uniforminR1dinitdatabds}. Let $T>0$ be fixed. Then, there exists a strong solution $(\rho, u)$ to $\eqref{eq:1DNS_T_R}$, where $u=\begin{bmatrix}
u_1 \\
u_2
\end{bmatrix}$. Further, there exists a constant $C > 0$ such that
\begin{align}\label{eq:1dbds}
&\inf_{(t,x) \in [0,T] \times \Pi}\rho^{-1}(t,x)+||(\rho, u)-(\tilde{\rho}, \tilde{u})||_{L^\infty([0,T];H^6(\Pi))} \\ 
&+||u-\tilde{u}||_{L^2([0,T];H^7(\Pi))} \leq C.
\end{align}
The constant $C$ does not depend on $R$.
\end{proposition}
Let us begin the proof. Firstly, by Theorem 1.6 in \cite{MR4049919}, there exists a global strong solution $(\rho, u_1)$ to the first two equations of \eqref{eq:1DNS_T_R}, 
with lower bound on the density (uniform in $R$)
\begin{align}\label{eq:densitylowerbd}
\underline{\kappa}:=\inf_{x \in \Pi}\rho(t,x) \geq \left([\inf_{x \in \Pi}\rho_0(x)]^{1-\gamma}+\frac{T}{\nu}(\gamma-1)\right)^{\frac{-1}{\gamma-1}}.
\end{align}
Further, notice that once $(\rho, u_1)$ have been shown to exist, we see that $u_2$ verifies the linear parabolic equation
\begin{align}\label{eq:u2eq}
(u_2)_t=((\nu \log(\rho))_x-u_1)(u_2)_x+\nu (u_2)_{xx}.
\end{align}
By standard theory for linear parabolic equations, there exists a local solution $u_2$, which easily extends to a global solution once the bound \eqref{eq:1dbds} has been established. So it suffices to prove \eqref{eq:1dbds}.
\par Define the following quantities:
\begin{align}
    \eta(U|\tilde{S}):=\frac{\rho|u_1-\tilde{u}_1|^2}{2}+p(\rho|\tilde{\rho}), \\
    \eta_{BD}(U|\tilde{S}):=\frac{\rho|u_1-\tilde{u}_1+\phi(\rho)_x|^2}{2}, 
\end{align}
where $p(\rho|\tilde{\rho})$ is the relative pressure:
\begin{align*}
p(\rho|\tilde{\rho}):=\frac{1}{\gamma-1}\rho^\gamma-\frac{1}{\gamma-1}\tilde{\rho}^\gamma-\frac{\gamma}{\gamma-1}\tilde{\rho}^{\gamma-1}(\rho-\tilde{\rho}),
\end{align*}
and $\phi$ verifies
\begin{align*}
    \phi'(\rho)=\frac{\mu(\rho)}{\rho^2}, \text{ for } \mu(\rho)=2\nu \rho.
\end{align*}
Finally, note that by \eqref{eq:uniforminR1dinitdatabds}, we have bounds (uniform in $R$)
\begin{align*}
\int_{\Pi}\eta(U_0|\tilde{S})dx \leq C, \\
\int_{\Pi}\eta_{BD}(U_0|\tilde{S})dx \leq C.
\end{align*}
The first step is to obtain the relative energy/BD-entropy inequalities.
\par \underline{\bf{Step 1: Energy and BD-entropy inequality}.}
Firstly, we compute the time evolution of the relative entropy. Following the proof of Lemma 3.1 in \cite{MR2368905}, we obtain
\begin{align*}
&\frac{d}{dt}\int_{\Pi}\eta(U|\tilde{S})dx+\int_{\Pi}\mu(\rho)|(u_1)_x|^2dx \\
&\leq -\int_{-1}^1\nabla ^2\eta(\tilde{S})A(\tilde{S})_x((\rho,u_1)-(\tilde{\rho}, \tilde{u}_1))dx+\int_{-1}^1(\tilde{u}_1)_x\mu(\rho)(u_1)_xdx\\
&\quad -\int_{-1}^1(\nabla \eta(\tilde{S}))_xA(U|\tilde{S})dx -\int_{-1}^1D( \eta(\tilde{S}))_xA(\tilde{S})dx,
\end{align*}
where 
\begin{align*}
A(U)&=\begin{bmatrix}
\rho u_1 \\
\rho u_1^2+\rho^\gamma
\end{bmatrix}, \\
A(U|\tilde{S})&=\begin{bmatrix}
0 \\ \rho(u_1-\tilde{u}_1)^2+(\gamma-1)p(\rho|\tilde{\rho}).
\end{bmatrix}.
\end{align*}
We bound each term as follows. For the first,
\begin{align*}
&-\int_{-1}^1\nabla ^2\eta(\tilde{S})A(\tilde{S})_x((\rho,u_1)-(\tilde{\rho}, \tilde{u}_1))dx \leq C\int_{-1}^1|U-\tilde{S}|dx \\
&\qquad \leq \int_{-1}^1|\rho-\tilde{\rho}|dx+\int_{-1}^1\rho|(\rho,u_1)-(\tilde{\rho}, \tilde{u}_1)|dx+\int_{-1}^1|\tilde{u}_1(\rho-\tilde{\rho})|dx \\
& \qquad \leq C\int_{-1}^1(1+p(\rho|\tilde{\rho}))dx+\big(\int_{-1}^1\rho dx\big)^{\frac{1}{2}}\big(\int_{-1}^1 \rho (u_1-\tilde{u}_1)^2dx\big)^{\frac{1}{2}} \\
&\qquad  \leq C\int_{-1}^1(1+p(\rho|\tilde{\rho}))dx+\big(\int_{-1}^1(1+p(\rho|\tilde{\rho})) dx\big)^{\frac{1}{2}}\big(\int_{-1}^1 \eta(U|\tilde{S})dx\big)^{\frac{1}{2}} \\
& \qquad \leq C+C\int_{\Pi}\eta(U|\tilde{S})dx, 
\end{align*}
where we have used \cite[(3.3)]{MR2368905} for the third and fourth inequalities. For the second, we compute
\begin{align*}
\int_{-1}^1(\tilde{u}_1)_x\mu(\rho)(u_1)_xdx &\leq \frac{1}{2}\int_{-1}^1 \mu(\rho)|(u_1)_x|^2dx+C\int_{-1}^1\rho dx \\
&\leq \frac{1}{2}\int_{-1}^1 \mu(\rho)|(u_1)_x|^2dx+C+C\int_{\Pi}\eta(U|\tilde{S})dx.
\end{align*}
For the third:
\begin{align*}
-\int_{-1}^1(D\eta(\tilde{S}))_xA(U|\tilde{S})dx &\leq C\int_{-1}^1A(U|\tilde{S})dx \\
&\leq C\int_{\Pi}\eta(U|\tilde{S})dx,
\end{align*}
as $|A(U|\tilde{S})| \leq \max(1, (\gamma-1))\eta(U|\tilde{S})$. Finally,
\begin{align*}
\int_{-1}^1\nabla ( \eta(\tilde{S}))_xA(\tilde{S})dx \leq C.
\end{align*}
Altogether, we obtain
\begin{align}\label{eq:energyinequality}
\ds{\frac{d}{dt}}\int_{\Pi}\eta(U|\tilde{S})dx+\frac{1}{2}\int_{\Pi}\mu(\rho)|(u_1)_x|^2dx & \leq C+C\int_{\Pi}\eta(U|\tilde{S})dx.
\end{align}
Next, we compute the time evolution of the relative BD entropy. By in \cite[(3.17)]{MR2368905}, we have
\begin{align*}
\ds{\frac{d}{dt}}\Big(\int_{\Pi}&\rho(u_1-\tilde{u}_1)\phi(\rho)_x+\frac{\rho|\phi(\rho)_x|^2}{2}dx\Big)+\int_{\Pi}\phi(\rho)_x(\rho^\gamma)_xdx \\
&=\int_{\Pi}\mu(\rho)((u_1)_x)^2dx-\int_{\Pi} \mu(\rho)(u_1)_x(\tilde{u}_1)_xdx-\int_{\Pi}\rho  \phi(\rho)_xu_1(\tilde{u}_1)_xdx.
\end{align*}
The first term on the right-hand side may be controlled by the dissipation in \eqref{eq:energyinequality} (up to a constant). For the second, we compute
\begin{align*}
-\int_{\Pi} \mu(\rho)(u_1)_x(\tilde{u}_1)_xdx &\leq C\int_{\Pi}\mu(\rho)(u_1)_x^2dx+C\int_{\Pi}\rho((\tilde{u}_1)_x)^2dx\\
& \leq C\int_{\Pi}\mu(\rho)(u_1)_x^2dx+C\int_{-1}^1 (1+p(\rho|\tilde{\rho})dx \\
& \leq C\int_{\Pi}\mu(\rho)(u_1)_x^2dx+C+C\int_{\Pi}\eta(U|\tilde{S})dx.
\end{align*}
For the last, we compute 
\begin{align*}
&-\int_{\Pi}\rho  \phi(\rho)_xu_1(\tilde{u}_1)_xdx \leq C\int_{-1}^1\rho|\phi(\rho)_x|^2dx+C\int_{1}^1 \rho u_1^2dx \\
& \qquad \leq C\int_{-1}^1\rho(\phi(\rho)_x-(u_1-\tilde{u}_1))^2dx+C\int_{-1}^1\rho(u_1-\tilde{u}_1)^2dx+C\int_{-1}^1\rho \tilde{u}_1^2dx \\
&\qquad \leq C+C\int_{\Pi}\eta_{BD}(U|\tilde{S})dx+C\int_{\Pi}\eta(U|\tilde{S})dx.
\end{align*}
Altogether, we obtain
\begin{align}\label{eq:BDentropyinequality}
\ds{\frac{d}{dt}}\Big(\int_{\Pi}&\rho(u_1-\tilde{u}_1)\phi(\rho)_x+\frac{\rho|\phi(\rho)_x|^2}{2}dx\Big)+\int_{\Pi}\phi(\rho)_x(\rho^\gamma)_xdx \\ 
&\leq C+C\int_{\Pi}\mu(\rho)(u_1)_x^2dx+C\int_\Pi\eta(U|\tilde{S})dx+\int_{\Pi}\eta_{BD}(U|\tilde{S})dx.
\end{align}
Adding a constant multiple of \eqref{eq:energyinequality} to \eqref{eq:BDentropyinequality} and using Gronwall, we obtain
\begin{align}\label{eq:0orderestimate}
\int_{\Pi}\eta(U|\tilde{S})dx(t)+\int_{\Pi}\eta_{BD}(U|\tilde{S})dx(t)+\int_0^t\int_{\Pi} \mu(\rho)(u_1)_x^2dxd\tau +\int_0^t\int_{\Pi}\phi(\rho)_x(\rho^\gamma)_xdxd\tau \leq C.
\end{align}
\begin{remark}
At this point, it may be obtained from the above computations that the constant $C$ in \eqref{eq:0orderestimate} is independent of $\nu$. 
\end{remark}
\underline{\textbf{Step 2: Upper and lower bounds on the density}.} In this step, we obtain upper and lower bounds on $\rho(t,x)$ uniform in $R$ (but not $T$). We note that the uniform lower bound $\underline{\kappa}$ (see \eqref{eq:densitylowerbd}) is already obtained in \cite{MR4049919} (it is not stated explicitly therein but it is easy to check that the constant does not depend on the width of the torus). So we need only obtain an upper bound. We follow the same strategy as \cite[Proposition 4.5]{MR2368905}.
\par Let $s=\frac{\gamma-1}{2}$. Then, for $K \subset \subset \Pi$
\begin{align*}
\int_K(\rho^s)_xdx&=s\int_K|\rho^{s-1}\rho_x|dx \\
& \leq \big(\int_K \rho^\gamma dx\big)^{\frac{1}{2}}\big(\int_K\frac{1}{\rho^3}\rho_x^2dx\big)^{\frac{1}{2}} \\
& \leq \frac{C}{\nu \underline{\kappa}}\big(\int_K 1+p(\rho|\tilde{\rho}) dx\big)^{\frac{1}{2}}\big(\int_K\rho(\frac{\mu(\rho)}{\rho^2})^2\rho_x^2dx\big)^{\frac{1}{2}} \\
& \leq \frac{C}{\nu \underline{\kappa}}\big(|K|+1\big)^\frac{1}{2},
\end{align*}
by \eqref{eq:0orderestimate}. Further
\begin{align*}
\int_{K}\rho^sdx \leq \int_K 1+\rho^\gamma dx \leq C(|K|+1),
\end{align*}
by \eqref{eq:0orderestimate}. It follows that $\rho^s \in L^\infty([0,T];W^{1,1}_{\text{loc}}(\Pi))$, where the $W^{1,1}(K)$ norm of $\rho^s(t,\cdot)$ depends only on $|K|$ (crucially, it does not depend on $R$). Then, fix $(t,x) \in [0,T] \times \Pi$. We wish to estimate $\rho^s(t,x)$. Let $K:=[x-1,x+1]$. Then, uniformly in $R$, by Markov's inequality we have
\begin{align*}
|\{\rho^s(t,x) \geq M, x \in K\}| \leq \frac{||\rho^s(t,\cdot)||_{L^1(K)}}{M}.
\end{align*}
Choosing $M \geq 2||\rho^s||_{L^\infty([0,T];L^1(K))}$, we obtain
\begin{align*}
|\{\rho^s(t,x) \geq M, x \in K\}| \leq \frac{1}{2}.
\end{align*}
So, there exists $y \in K$ so that $\rho^s(t,y) \leq M$. Then, 
\begin{align*}
|\rho^s(t,x)| \leq |\rho^s(t,y)|+||(\rho^s)_x||_{L^1(K)}.
\end{align*}
This implies that we have an upper bound on the density uniformly in $R$
\begin{align}\label{eq:densityupperbd}
\rho^s(t,x) \leq \frac{C}{\nu \underline{\kappa}} \leq C.
\end{align}
\begin{remark}\label{rem:densityupperbdscaling}
At this point, it may be computed that the constant C in \eqref{eq:densityupperbd} scales in $\nu$ like $\frac{C}{\nu^j}$ for some integer $j>0$.
\end{remark}
\underline{\textbf{Step 3: $H^1$-estimate}.}
In this step, we obtain the estimate
\begin{align}\label{eq:1orderestimate}
||\rho-\tilde{\rho}||_{L^\infty_tH^1_x}+||u_1-\tilde{u}_1||_{L^\infty_tH^1_x}+||u_1-\tilde{u}_1||_{L^2_tH^2_x} \leq C.
\end{align}
Firstly, as both $\rho,\tilde{\rho}$ are bounded from above and below, we compute
\begin{align*}
\int_\Pi(\rho-\tilde{\rho})^2dx &\leq C\int_\Pi p(\rho|\tilde{\rho})dx \leq C,
\end{align*}
by \eqref{eq:0orderestimate}. For the derivative estimate,
\begin{align*}
\int_\Pi \rho_x^2dx \leq C\int_\Pi \rho \phi(\rho)_x^2dx \leq C,
\end{align*}
by \eqref{eq:0orderestimate}, as $\rho$ is bounded from above. As $\text{supp}(\tilde{\rho}) \subset [-1,1]$, this grants
\begin{align}\label{eq:rhoh1}
||\rho-\tilde{\rho}||_{L^\infty_tH^1_x} \leq C.
\end{align}
For the $H^1$ bound on $u_1-\tilde{u}_1$, first note that $||u_1-\tilde{u}_1||_{L^\infty_tL^2_x} \leq C$ by \eqref{eq:0orderestimate} and \eqref{eq:densitylowerbd}. Next, we do an energy estimate on $(u_1)_x$, which verifies the equation
\begin{align*}
&(u_1)_{xt}-2\nu (u_1)_{xxx}=-\gamma(\gamma-2)\rho^{\gamma-3}\rho_x-\gamma \rho^{\gamma-2}\rho_{xx}+(\tilde{u}_1)_x(u_1)_x+\tilde{u}(u_1)_{xx} \\
&\qquad\quad  +(\nu\log(\rho)_{xx}-((u_1)_x-(\tilde{u}_1)_x))(u_1)_x+(\nu \log(\rho)_x-(u_1-\tilde{u}_1))(u_1)_{xx}.
\end{align*}
Multiplying the equation by $(u_1)_x$ and integrating by parts, we obtain
\begin{align*}
&\ds{\frac{d}{dt}}\int_\Pi \frac{1}{2}(u_1)_x^2dx+2 \nu \int_{\Pi}(u_1)_{xx}^2dx=-\int_\Pi \gamma(\gamma-2)\rho^{\gamma-3}\rho_x(u_1)_xdx \\
&\qquad -\gamma\int_\Pi\rho^{\gamma-2}\rho_{xx}(u_1)_xdx+\int_\Pi (\tilde{u}_1)_x(u_1)_x^2dx+\int_\Pi \tilde{u}_1(u_1)_{xx}(u_1)_xdx \\
&\qquad +\frac{1}{2}\int_\Pi (\nu \log(\rho)_{xx})(u_1)_x^2dx-\frac{1}{2}\int_\Pi((u_1)_x-(\tilde{u}_1)_x))(u_1)_x^2dx.
\end{align*}
For the first two terms, we compute
\begin{align*}
-\int_\Pi \gamma(\gamma-2)\rho^{\gamma-3}\rho_x(u_1)_xdx-\gamma\int_\Pi\rho^{\gamma-2}\rho_{xx}(u_1)_xdx&=\int_\Pi(-\gamma \rho^{\gamma-2}\rho_x)_x(u_1)_xdx \\
&=\int_\Pi \gamma \rho^{\gamma-2}\rho_x(u_1)_{xx}dx \\
& \leq \frac{\nu}{4}\int_{\Pi}(u_1)_{xx}^2dx+C\int_\Pi\rho_x^2dx.
\end{align*}
For the third and fourth
\begin{align*}
\int_\Pi (\tilde{u}_1)_x(u_1)_x^2dx+\int_\Pi \tilde{u}_1(u_1)_{xx}(u_1)_xdx \leq C\int_{\Pi}(u_1)_x^2dx.
\end{align*}
For the fifth
\begin{align*}
&\frac{1}{2}\int_\Pi (\nu \log(\rho)_{xx})(u_1)_x^2dx \leq C||(u_1)_{xx}||_2||(u_1)_x||_\infty||\rho_x||_2 \\
&\qquad \leq C||(u_1)_{xx}||^{\frac{3}{2}}||(u_1)_x||_2^{\frac{1}{2}}||\rho_x||_2 \leq \frac{\nu}{4}||(u_1)_{xx}||_2^2+C||(u_1)_x||_2^2||\rho_x||_2^4 \\
&\qquad \leq \frac{\nu}{4}||(u_1)_{xx}||_2^2+C||(u_1)_x||_2^2,
\end{align*}
where the second inequality follows from the Gagliardo-Nirenberg inequality (note that the constant is uniform in $R$ as derivatives have mean zero), and we note by \eqref{eq:rhoh1}, $||\rho_x||_2$ is uniformly bounded in time. Finally, for the last
\begin{align*}
-\frac{1}{2}\int_\Pi((u_1)_x-(\tilde{u}_1)_x))(u_1)_x^2dx&=\int_\Pi (u_1-\tilde{u}_1)(u_1)_{xx}(u_1)_xdx\\
&\leq ||(u_1)_{xx}||_2||(u_1)_x||_\infty||u_1-\tilde{u}_1||_2 \\
&\leq ||(u_1)_{xx}||_2^{\frac{3}{2}}||(u_1)_x||_2^{\frac{1}{2}}||u_1-\tilde{u}_1||_{2} \\
&\leq \frac{\nu}{4}||(u_1)_{xx}||_2^2+C||(u_1)_x||_2^2,
\end{align*}
by the Gagliardo-Nirenberg inequality. Applying Gronwall gives \eqref{eq:1orderestimate}.
\begin{remark}
We remark that at this point, the constant $C$ in \eqref{eq:1orderestimate} can be computed to scale in $\nu$ like $C\exp(\frac{C}{\nu^j})$ for some $j \in \N$.
\end{remark}
\underline{\textbf{Step 3: $H^2$-estimate}} In this step, we obtain the estimate
\begin{align}\label{eq:2orderestimate}
||\rho-\tilde{\rho}||_{L^\infty_tH^2_x}+||u_1-\tilde{u}_1||_{L^\infty_tH^2_x}+||u_1-\tilde{u}_1||_{L^2_tH^3_x} \leq C.
\end{align}
The strategy is to do energy estimates simultaneously on $\rho_{xx}$ and also the derivative of the \textit{active potential} $w$, which is defined as 
\begin{align}\label{eq:activepotential}
w:=-p(\rho)+\mu(\rho)(u_1)_x.
\end{align}
Note that $w_x$ verifies the equation
\begin{align}\label{eq:w_xeqn}
w_x=-p'(\rho)\rho_x+\mu'(\rho)\rho_x(u_1)_x+\mu(\rho)(u_1)_{xx}.
\end{align}
The active potential was introduced in \cite{MR4049919} and is crucially used to obtain existence results for more general relations between $\gamma$ and $\mu(\rho)$ than are treated here. Firstly, note that from \eqref{eq:1orderestimate} and the Gagliardo-Nirenberg inequality, we have bounds
\begin{align}\label{eq:0orderw}
||w||_{L^2_tL^\infty_x}+||w_x||_{L^2_tL^2_x} \leq C.
\end{align}
Differentiating \cite[(2.10)]{MR4113341} in space, multiplying by $w_x$, and integrating by parts, we obtain
\begin{align*}
\ds{\frac{d}{dt}}\int_{\Pi}\frac{w_x^2}{2}dx&=-\int_\Pi \frac{\mu(\rho)}{\rho}|w_{xx}^2|dx+\int_{\Pi}\frac{\mu(\rho)}{\rho^2}\rho_xw_xw_{xx}dx \\
&-\int_\Pi u_x\frac{w_x^2}{2}dx+\int_{\Pi}f_1(\rho)|w_x|^2dx+\int_\Pi f_1'(\rho)\rho_xww_xdx \\
&-2\int_{\Pi}f_2(\rho)w|w_x|^2dx-\int_\Pi f_2'(\rho)\rho_xw^2w_x+\int_\Pi f_3'(\rho)\rho_xw_xdx \\
&=-2\nu \int_\Pi|w_{xx}^2|dx+\sum_{k=1}^7I_k,
\end{align*}
where
\begin{align*}
f_1(\rho)&:=\rho\frac{p'(\rho)}{\mu(\rho)}-2p(\rho)\frac{\rho \mu'(\rho)+\mu(\rho)}{\mu(\rho)^2}, \\
f_2(\rho)&:=\frac{\rho \mu'(\rho)+\mu(\rho)}{\mu(\rho)^2}, \\
f_3(\rho)&:=\big(\rho \frac{p'(\rho)}{\mu(\rho)}-p(\rho)\frac{\rho\mu'(\rho)+\mu(\rho)}{\mu(\rho)^2}\big)p(\rho).
\end{align*}
We show how to bound each term. For $I_1$:
\begin{align*}
I_1 \leq C||\rho_x||_\infty ||w_x||_2||w_{xx}||_2 &\leq \frac{\nu}{2}||w_{xx}||_2^2+C||\rho_x||_\infty^2||w_x||_2^2 \\
&\leq \frac{\nu}{2}||w_{xx}||_2^2+C||\rho_x||_2||\rho_{xx}||_2||w_x||_2^2 \\
&\leq \frac{\nu}{2}||w_{xx}||_2^2+C(||\rho_x||_2^2+||\rho_{xx}||_2^2)||w_x||_2^2.
\end{align*}
Note that we do not sufficiently control $\rho_{xx}$ in $L^2_tL^2_x$. This is why we need to simultaneously do an energy estimate on $\rho_{xx}$. However, the rest of the terms may be bounded in terms of norms we already control as follows. 
\begin{align*}
I_2+I_3 &\leq C||w_x||_2^2, \\
I_4 &\leq C||\rho_x||_2(||w||_\infty^2+||w_x||_2^2), \\
\end{align*}
where we note that $||\rho_x||_2||w||_\infty^2 \in L^1([0,T])$ by \eqref{eq:1orderestimate} and \eqref{eq:0orderw}. Continuing:
\begin{align*}
I_5 &\leq C||w||_\infty ||w_x||_2^2, \\
I_6 &\leq C(||\rho_x||_2^2||w||_\infty^2+||w||_\infty^2||w_x||_2^2), \\
I_7 &\leq C(||\rho_x||_2^2+||w_x||_2^2).
\end{align*}
So, altogether, we obtain that there exists $F_1, F_2 \in L^1([0,T])$ so that
\begin{align}\label{eq:part1}
\ds{\frac{d}{dt}}\int_{\Pi}\frac{w_x^2}{2}dx&\leq -\nu\int_\Pi|w_{xx}^2|dx+F_1(t)+F_2(t)\int_{\Pi}\frac{w_x^2}{2}dx+C||\rho_{xx}||_2^2||w_x||_2^2.
\end{align}
Next, we do the energy estimate on $\rho_{xx}$. Differentiating the continuity equation in \eqref{eq:1DNS_T_R} twice in space, multiplying by $\rho_{xx}$, and integrating in space, we obtain
\begin{align*}
\ds{\frac{d}{dt}}\int_\Pi \frac{\rho_{xx}^2}{2}dx&=\int_\Pi (u_1)_x\frac{\rho_{xx}^2}{2}dx-\int_\Pi(u_1\rho_x)_{xx}-u_1\rho_{xxx})\rho_{xx}dx \\
&-\int_\Pi \rho (u_1)_{xxx}\rho_{xx}dx-\int_\Pi ((\rho (u_1)_x)_{xx}-\rho (u_1)_{xxx})\rho_{xx}dx \\
&=\sum_{k=1}^4J_k.
\end{align*}
Each term may be estimated as follows.
\begin{align*}
J_1 &\leq ||(u_1)_x||_\infty||\rho_{xx}||_2^2, \\
J_2 &=-\int_\Pi ((u_1)_{xx}\rho_x+2(u_1)_x\rho_{xx})\rho_{xx}dx \\
&\leq C(||(u_1)_{xx}||_2||\rho_{x}||_\infty||\rho_{xx}||_2+||(u_1)_x||_\infty||\rho_{xx}||_2^2) \\
&\leq C(||(u_1)_{xx}||_2^2||\rho_{x}||_2^2+||(u_1)_{xx}||_2^2||\rho_{xx}||_2^2+||(u_1)_x||_\infty||\rho_{xx}||_2^2),
\end{align*}
where we note $||(u_1)_{xx}||_2^2||\rho_x||_2^2 \in L^1([0,T])$ by \eqref{eq:1orderestimate}. For $J_3$, it will be necessary to estimate $||(u_1)_{xxx}||_2$. Similar to \cite[(A.6)]{MR4113341}, we estimate
\begin{align*}
||(u_1)_{xxx}||_2  &\leq \big(||w||_\infty||\rho_x||_2^2+||w||_\infty||\rho_{xx}||_2||\rho_x||_2 \\
&+||\rho_{xx}||_2||w||_\infty+||\rho_{x}||_2||w_x||_\infty \\
&+||w_{xx}||_2+||\rho_x||_\infty+||\rho_x||_2^2+||\rho_{xx}||_2\big).
\end{align*}
Noting $||w||_\infty \in L^2([0,T])$ by \eqref{eq:1orderestimate} and using the Gagliardo-Nirenberg inequality to control $||w_{x}||_\infty$ and $||\rho_x||_\infty$, for $J_3$ we obtain
\begin{align*}
J_3 \leq C||(u_1)_{xxx}||_2||\rho_{xx}||_2 \leq H_1(t)+H_2(t)(||\rho_{xx}||_2^2+||w_x||_2^2)+\frac{\nu}{2}||w_{xx}||_2^2,
\end{align*}
where $H_1, H_2 \in L^1([0,T])$. Finally for $J_4$
\begin{align}\label{eq:part2}
J_4&=-\int_{\Pi}(\rho_{xx}u_{1x}+2\rho_xu_{1xx})\rho_{xx}dx \\
&\leq C(\|u_{1x}\|_\infty\|\rho_{xx}\|_2^2+\|\rho_{x}\|_\infty\|u_{1xx}\|_2\|\rho_{xx}\|_2) \\
&\leq C(\|u_{1x}\|_\infty \|\rho_{xx}\|_2^2+\|u_{1xx}\|_2[\|\rho_x\|_2^4+\|\rho_{xx}\|_2^2]).
\end{align}
Summing \eqref{eq:part1} and \eqref{eq:part2}, we obtain that there exists $G_1, G_2 \in L^1([0,T])$ so that 
\begin{align*}
\ds{\frac{d}{dt}}\Big(\int_{\Pi}\frac{w_x^2}{2}dx+\int_\Pi \frac{\rho_{xx}^2}{2}dx\Big)+\nu \int_{\Pi}w_{xx}^2dx \leq G_1(t)+G_2(t)\Big(\int_{\Pi}\frac{w_x^2}{2}dx+\int_\Pi \frac{\rho_{xx}^2}{2}dx\Big).
\end{align*}
Applying Gronwall gives the estimate
\begin{align}\label{eq:2orderestimatew}
||\rho_{xx}||_{L^\infty_tL^2_x}+||w_x||_{L^\infty_tL^2_x}+||w_{xx}||_{L^2_tL^2_x} \leq C.
\end{align}
Using the equation for $w$ \eqref{eq:activepotential} and the fact that $\text{supp}(\tilde{S}) \subset [-1,1]$, this easily implies \eqref{eq:2orderestimate}.
\begin{remark}
Note that at this point, the constant $C$ in \eqref{eq:2orderestimate} scales in $\nu$ like $C\exp(\exp(\frac{C}{\nu^j}))$. 
\end{remark}
\underline{\textbf{Step 5: $H^2$-estimate for $u_2-\tilde{u}_2$}.} Next, we obtain the estimate
\begin{align}\label{eq:u2est}
||u_2-\tilde{u}_2||_{L^\infty_tH^2_x}+||u_2-\tilde{u}_2||_{L^2_tH^3_x} \leq C.
\end{align}
By \eqref{eq:u2eq},  we see that $v:=u_2-\tilde{u}_2$ verifies
\begin{align}\label{eq:veqn}
v_t&=\nu \log(\rho)_x v_x+\nu v_{xx}-u_1 v_x+\nu (\tilde{u}_2)_{xx}-u_1(\tilde{u}_2)_x+\nu \log(\rho)_x (\tilde{u}_2)_x \notag \\
&=:\nu \log(\rho)_x v_x+\nu v_{xx}-u_1 v_x+I.
\end{align}
Multiplying \eqref{eq:veqn} by $v$ and integrating gives
\begin{align*}
\ds{\frac{d}{dt}}\int_\Pi \frac{v^2}{2}dx+\nu \int_\Pi v_x^2dx=- \nu \int_\Pi \frac{\rho_x}{\rho}v_xvdx+\int_\Pi Ivdx.
\end{align*}
Using the fact that $||\frac{\rho_x}{\rho}||_{L^\infty([0,T] \times \Pi)} \leq C$ and Young's inequality gives
\begin{align}\label{eq:u20order}
\ds{\frac{d}{dt}}\int_\Pi \frac{v^2}{2}dx+\frac{\nu}{2} \int_\Pi v_x^2dx\leq C\int_\Pi v^2dx+F_1(t),
\end{align}
where $F_1(t)=\int_\Pi I^2 dx \in L^1(0,T)$. Next, we differentiate \eqref{eq:u2eq}, multiply by $(u_2)_x$, and integrate to obtain
\begin{align*}
\ds{\frac{d}{dt}} \int_\Pi \frac{(u_2)_x^2}{2}dx+\nu \int_\Pi (u_2)_{xx}^2dx=-\nu \int_\Pi \log(\rho)_x(u_2)_x(u_2)_{xx}dx-\frac{1}{2}\int_\Pi (u_1)_x(u_2)_x^2dx.
\end{align*}
Using the fact that $||\log(\rho)_x||_{L^\infty([0,T] \times \Pi)}+||(u_1)_x||_{L^\infty([0,T] \times \Pi)} \leq C$ and Young's inequality gives
\begin{align}\label{eq:u21order}
\ds{\frac{d}{dt}} \int_\Pi \frac{(u_2)_x^2}{2}dx+\frac{\nu}{2} \int_\Pi (u_2)_{xx}^2dx \leq C\int_\Pi (u_2)_x^2dx.
\end{align}
Finally, differentiating \eqref{eq:u2eq} twice in space, multiplying by $(u_2)_{xx}$, and integrating gives
\begin{align*}
\ds{\frac{d}{dt}} \int_\Pi \frac{(u_2)_{xx}^2}{2}dx+\nu \int_\Pi (u_2)_{xxx}^2dx&=\nu \int_\Pi (\log(\rho)_x (u_2)_x)_{xx}(u_{2})_{xx}dx \\
&-\int_\Pi (u_1)_{xx}(u_2)_x(u_2)_{xx}dx-2\int_\Pi (u_1)_x(u_2)_{xx}^2dx \\
&-\int_{\Pi}(u_1)_x(u_2)_{xxx}(u_2)_{xx}dx. 
\end{align*}
For the latter three terms, integration by parts and Young's inequality give
\begin{align}\label{eq:u2step1}
\ds{\frac{d}{dt}} \int_\Pi \frac{(u_2)_{xx}^2}{2}dx+\frac{\nu}{2} \int_\Pi (u_2)_{xxx}^2dx&\leq \nu \int_\Pi (\log(\rho)_x (u_2)_x)_{xx}(u_{2})_{xx}dx \notag \\
&+C\Big(\int_\Pi (u_2)_x^2dx+\int_\Pi (u_2)_{xx}^2dx\Big).
\end{align}
It remains the handle the first term. Integrating by parts and expanding we obtain
\begin{align}\label{eq:u2step2}
\nu \int_\Pi (\log(\rho)_x (u_2)_x)_{xx}(u_{2})_{xx}dx &\leq \nu \int_\Pi \frac{\rho_{xx}}{\rho}(u_{2})_x(u_2)_{xxx}dx \notag \\
&+\nu \int_\Pi \frac{\rho_x^2}{\rho^2}(u_{2})_x(u_2)_{xxx}dx \\
&+\int_\Pi \log(\rho)_xu_{2xx}u_{2xxx}dx.
\end{align}
The second term above can be handled by Young's inequality. For the first we use the Gagliardo-Nirenberg inequality to obtain
\begin{align}\label{eq:u2step3}
\nu \int_\Pi \frac{\rho_{xx}}{\rho}(u_{2})_x(u_2)_{xxx}dx &\leq C||\rho_{xx}||_2||(u_{2})_x||_\infty||(u_2)_{xxx}||_2 \notag \\
&\leq C||\rho_{xx}||_2||(u_2)_x||_2^{\frac{1}{2}}||(u_2)_{xx}||_2^{\frac{1}{2}}||(u_2)_{xxx}||_2 \notag \\
&\leq C||\rho_{xx}||_2^2 \left(||(u_2)_x||_2^2+||(u_2)_{xx}||_2^2\right)+\frac{\nu}{8}||(u_2)_{xxx}||_2^2.
\end{align}
The third is also easily handled using Young's inequality. Plugging \eqref{eq:u2step2} and \eqref{eq:u2step3} into \eqref{eq:u2step1} and using Young's inequality gives
\begin{align}\label{eq:u22order}
\ds{\frac{d}{dt}} \int_\Pi \frac{(u_2)_{xx}^2}{2}dx+\frac{\nu}{8} \int_\Pi (u_2)_{xxx}^2dx&\leq F_2(t)\left(||(u_2)_x||_2^2+||(u_2)_{xx}||_2^2\right),
\end{align}
where $F_2(t) \in L^1(0,T)$. Finally, combining \eqref{eq:u20order}, \eqref{eq:u21order} and \eqref{eq:u22order} gives
\begin{align*}
&\ds{\frac{d}{dt}} \Big(\int_\Pi \frac{(u_2-\tilde{u}_2)^2}{2}dx+\int_\Pi \frac{(u_2)_x^2}{2}dx+\int_\Pi \frac{(u_2)_{xx}^2}{2}dx \Big)+\frac{\nu}{8}\int_{\Pi}(u_2)_{xxx}^2dx \\
&\leq F_1(t)+F_2(t) \Big(\int_\Pi \frac{(u_2-\tilde{u}_2)^2}{2}dx +\int_\Pi \frac{(u_2)_x^2}{2}dx+\int_\Pi \frac{(u_2)_{xx}^2}{2}dx\Big), 
\end{align*}
where $F_1, F_2 \in L^1(0,T)$. Applying Gronwall's inequality gives \eqref{eq:u2est}. 
\begin{remark}\label{rem:triplebad}
Note that constant $C$ in \eqref{eq:u2est} can be computed to scale in $\nu$ like $C\exp(\exp(\exp(\frac{C}{\nu^j})))$. 
\end{remark}
\underline{\textbf{Step 6: Higher-order estimates via induction.}} At this point, combining \eqref{eq:2orderestimate}, \eqref{eq:2orderestimatew}, and \eqref{eq:u2est}, we have shown
\begin{align}\label{eq:k=2}
||\rho-\tilde{\rho}||_{L^\infty_tH^2_x}+||w-\tilde{w}||_{L^\infty_tH^1_x}+||w-\tilde{w}||_{L^2_tH^2_x}+||u-\tilde{u}||_{L^\infty_tH^2_x}+||u-\tilde{u}||_{L^2_tH^3_x} \leq M_2,
\end{align}
where we denote by $M_k$ a constant which, in terms of the Sobolev norms of $\rho^0-\tilde{\rho}$ and $u^0-\tilde{u}$, depends only on their $H^k$ norm (in terms of the other constants, it has the same dependencies as the constant C). In this step, we prove the following lemma:
\begin{lemma}\label{lem:sobolevinduction}
For any $k \geq 2$, there exists $M_k >0$ so that
\begin{align}\label{eq:higherorderestinduction}
||\partial_x^{k}\rho||_{L^\infty_tL^2_x}+||\partial_x^{k-1}w||_{L^\infty_tL^2_x}+||\partial_x^kw||_{L^2_tL^2_x}+||\partial_x^ku||_{L^\infty_tL^2_x}+||\partial_x^{k+1}u||_{L^2_tL^2_x} \leq M_k.
\end{align}
\end{lemma}
\begin{proof}
The proof is by induction. Note that \eqref{eq:higherorderestinduction} holds for $k=2$ by \eqref{eq:k=2}. So, assume that it holds for $k-1$ for $k \geq 3$. Then, differentiating the continuity equation $k$-times, multiplying by $\partial_x^k\rho$, and integrating in space gives
\begin{align*}
\frac{1}{2}\frac{d}{dt}\int_{\Pi}|\partial_x^k \rho|^2dx&=-\int_{\Pi}u_1\partial_x^{k+1}\rho\partial_x^k\rho dx-\int_{\Pi}([\partial_x^k,u_1]\partial_x \rho)\partial_x^k \rho dx \\
&-\int_{\Pi}([\partial_x^k,\rho]\partial_xu_1)\partial_x^k\rho dx-\int_{\Pi}\rho \partial_x^{k+1}u_1\partial_x^k \rho dx \\
&=:\sum_{j=1}^4I_j.
\end{align*}
The terms are controlled as follows:
\begin{align*}
I_1 \leq C||(u_1)_x||_\infty||\partial_x^k\rho||_2^2, 
\end{align*}
and using the Kato-Ponce commutator inequality
\begin{align*}
I_2 \leq C\big(||(u_1)_x||_\infty||\partial_x^k \rho||_2+||\partial_x^ku_1||_2||\rho_x||_\infty\big)||\partial_x^k\rho||_2 \leq M_{k-1}(1+||\partial_x^k\rho||_2^2),
\end{align*}
by the induction hypothesis. Next
\begin{align*}
I_3 \leq C\big(||\rho_x||_\infty||\partial_x^ku_1||_2+||\partial_x^k\rho||_2||(u_1)_x||_\infty\big)||\partial_x^k\rho||_2 \leq M_{k-1}(1+||\partial_x^k\rho||_2^2),
\end{align*}
where again we control some higher norms by $M_{k-1}$ by the induction hypothesis. Next,
\begin{align*}
I_4 \leq ||\rho||_\infty||\partial_x^{k+1}u_1||_2||\partial_x^k \rho||_2.
\end{align*}
To control the term $I_4$, we need to estimate $\partial_x^{k+1}u_1$. Using the equation for $w$ \eqref{eq:activepotential}, solving for $(u_1)_x$, and differentiating $k$ times, we have the estimate
\begin{align*}
||\partial_x^{k+1}u_1||_2 &\leq C\big(||\partial_x^k(w\rho^{-1})||_2+||\partial_x^k \rho^{\gamma-1}||_2\big) \\
&\leq C\big(||[\partial_x^k,\rho^{-1}]w||_2+||\partial_x^k \rho||_2+1+||\rho^{-1}\partial_x^kw||_2\big),
\end{align*}
where for the last term we have used that $\rho$ is bounded from above and below away from zero and the induction hypothesis. Using again the Kato-Ponce commutator inequality to control the first term, we obtain
\begin{align*}
||\partial_x^{k+1}u_1||_2 \leq C\big(||\partial_x^k\rho||_2+1+||\partial_x^kw||_2\big).
\end{align*}
So, combining all the terms, we obtain
\begin{align}\label{eq:inductionstep1}
\frac{1}{2}\frac{d}{dt}\int_{\Pi}|\partial_x^k \rho|^2dx& \leq M_{k-1}\big(||\partial_x^k\rho||_2^2+1\big)+\frac{\nu}{2}||\partial_x^kw||_2^2.
\end{align}
We need now to simultaneously do the energy estimate on $\partial_x^{k-1}w$. Differentiating \cite[(2.10)]{MR4113341} $k-1$ times in space, multiplying by $\partial_x^{k-1}w$, and integrating in space, we obtain
\begin{align*}
\frac{1}{2}\frac{d}{dt}\int_{\Pi}|\partial_x^{k-1}w|^2dx+2\nu ||\partial_x^kw||_2^2&=\int_{\Pi}\partial_x^{k-2}\big((u_1+2\nu \ln(\rho)_x)w_x\big)\partial_x^kwdx \\
&+C \int_\Pi \partial_x^k(\rho^{-1}w^2)\partial_{x}^{k-1}wdx+C\int_{\Pi}\partial_{x}^{k-1}\big(\rho^{\gamma-1}(w+\rho)\big)\partial_x^{k-1}w \\
&=:\sum_{j=1}^3I_j.
\end{align*}
For the first term, we control
\begin{align*}
I_1 &\leq \frac{\nu}{2}||\partial_x^kw||_2^2+C||\partial_x^{k-2}\big((u_1+2\nu \ln(\rho)_x)w_x\big)||_2^2 \\
&\leq \frac{\nu}{2}||\partial_x^{k}w||_2^2+C\big(||\partial_{x}^{k-1}w(u_1+2\nu \ln(\rho)_x)||_2^2+||[\partial_x^{k-2},(u_1+2\nu \ln(\rho)_x)]w_x||_2^2\big) \\
&\leq \frac{\nu}{2}||\partial_x^{k}w||_2^2+M_{k-1}||\partial_x^{k-1}w||_2^2 \\
&+C\big(||(u_1+2\nu \ln(\rho)_x)_x||_\infty^2||\partial_x^{k-2}w||_2^2+||\partial_x^{k-2}(u_1+2\nu \ln(\rho)_x)||_2^2||w_x||_\infty^2\big) \\
&\leq \frac{\nu}{2}||\partial_x^{k}w||_2^2+M_{k-1}(1+||\partial_x^{k-1}w||_2^2),
\end{align*}
by the Gagliardo-Nirenberg inequality to control the $L^\infty$ norms in the third line and the induction hypothesis.
\begin{remark}
Note that in the case $k=3$, we do not control $w_x$ in $L^\infty$. However, using the Gagliardo-Nirenberg inequality
\begin{align*}
||w_x||_\infty^2 \leq C||w_x||_2||w_{xx}||_2,
\end{align*}
and $||w_{xx}||_2 \in L^1(0,T)$ by the induction hypothesis, which is sufficient for the Gronwall inequality. So the proof proceeds with this small modification.
\end{remark}

The bounds on $I_2$ and $I_3$ are similar. So, we obtain
\begin{align}\label{eq:inductionstep2}
\frac{1}{2}\frac{d}{dt}\int_{\Pi}|\partial_x^{k-1}w|^2dx+\nu ||\partial_x^kw||_2^2 \leq M_{k-1}(||\partial_x^{k-1}w||_2^2+1).
\end{align}
Summing \eqref{eq:inductionstep1} and \eqref{eq:inductionstep2} and using Gronwall's inequality, we obtain 
\end{proof}
\begin{align}\label{eq:higherorderestnou}
||\partial_x^{k}\rho||_{L^\infty_tL^2_x}+||\partial_x^{k-1}w||_{L^\infty_tL^2_x}+||\partial_x^kw||_{L^2_tL^2_x} \leq M_k.
\end{align}
To obtain \eqref{eq:higherorderestinduction}, it remains to prove the estimates on the higher-derivatives of $u$. The estimate
\begin{align*}
||\partial_x^ku_1||_{L^\infty_tL^2_x}+||\partial_x^{k+1}u_1||_{L^2_tL^2_x} \leq M_k,
\end{align*}
is easily recovered from \eqref{eq:higherorderestnou} and the definition of $w$ \eqref{eq:activepotential}. So it remains only to do the $H^k$-estimate on $u_2$. Differentiating \eqref{eq:u2eq} $k$ times in space, multiplying by $\partial_x^ku_2$, and integrating, we obtain
\begin{align*}
\frac{1}{2}\frac{d}{dt}\int_{\Pi}|\partial_x^ku_2|^2dx+\nu ||\partial_{x}^{k+1}u_2||_2^2=\int_{\Pi}\partial_{x}^{k-1}\big((\log(\rho)_x-u_1)(u_2)_x\big)\partial_x^{k+1}u_2dx.
\end{align*}
Using Young's inequality and then \eqref{eq:higherorderestnou} and the induction hypothesis similarly to above, we obtain
\begin{align*}
\frac{1}{2}\frac{d}{dt}\int_{\Pi}|\partial_x^ku_2|^2dx+\frac{\nu}{2}||\partial_{x}^{k+1}u_2||_2^2 \leq M_k(1+||\partial_x^ku_2||_2^2).
\end{align*}
So, using Gronwall's inequality, we finish the proof of \Cref{lem:sobolevinduction}. Applying \Cref{lem:sobolevinduction} with $k=6$, this completes the proof of \Cref{prop:1dexist+bdsappendix}.

\section{Proof of \Cref{lem:appclem}}\label{sec:normsapp}
\renewcommand{\theequation}{C.\arabic{equation}}
\setcounter{equation}{0}
In this section, we present a proof of \Cref{lem:appclem}. Recall that we have strong solutions $(\bar \rho_R, \bar u_R)$ to \eqref{eq:2DNS_T_R} with initial data given by \eqref{eq:reflectdef2}. Further, owing to \Cref{def:suitableinitial}, by \cite{MR4113341} there exists a global-in-time unique smooth solution $(\bar \rho, \bar u)$ with initial data $(\bar \rho_0, \bar u_0)$ verifying for fixed $T>0$ the bounds
\begin{align}\label{eq:fulllinebds}
||\bar \rho-\tilde{\rho}, \bar u-\tilde{u}||_{C([0,T];H^6(\R))} &\leq C\delta_3, \\
0 < \underline{\kappa} \leq \bar \rho(x_1) &\leq \overline{\kappa}.
\end{align}
In \eqref{eq:localconvergence}, we have already established the local strong convergence of $(\bar \rho_R, \bar u_R)$ to $(\bar \rho, \bar u)$. Here, we would like to compare some norms of $(\bar \rho_R, \bar u_R)$ to $(\bar \rho, \bar u)$ on the full region $\T_R$. The main issue is that on the full region, the initial data of $(\bar \rho_R, \bar u_R)$ and $(\bar \rho, \bar u)$ are not the same (due to the reflection). For this reason, we will use $(\bar \rho, \bar u)$ to create an auxiliary function $(\rho^{app}_R, u^{app}_R)$ to compare to $(\bar \rho_R, \bar u_R)$. Then, we will compare $(\rho^{app}_R, u^{app}_R)$ to $(\bar \rho, \bar u)$ to obtain the result. 

\subsubsection{Definition of the auxiliary function}
Recall the function $\chi$ defined in \eqref{eq:chiprops}, \eqref{eq:chiRdef}. Also recall the definitions of the periodized reference functions $(\tilde{\rho}_R, \tilde{u}_R)$ in \eqref{eq:reflectdef3}, \eqref{eq:reflectdef4}. Then, we define the auxiliary function $(\rho^{app}_R, u^{app}_R)$ on $\T_R$:
\begin{align}\label{eq:appsolndef}
&\rho^{app}_R(t,x_1)=\tilde{\rho}_R(x_1)+\chi_R(x_1)(\bar{\rho}(t,x_1)-\tilde{\rho}(x_1))+\chi_R(R-x_1)(\bar{\rho}(t,R-x_1)-\tilde{\rho}(R-x_1)), \\
&u^{app}_{R,1}(t,x_1)=\tilde{u}_{R,1}(x_1)+\chi_R(x_1)(\bar{u}_1)(t,x_1)-\tilde{u}_1(x_1))-\chi_R(R-x_1)(\bar{u}_1(t,R-x_1)-\tilde{u}_1(R-x_1)), \\
&u^{app}_{R,2}(x_1)=\tilde{u}_{R,2}(x_1)+\chi_R(x_1)(\bar u_2(t,x_1)-\tilde{u}_2(x_1))+\chi_R(R-x_1)(\bar u_2(t,R-x_1)-\tilde{u}_2(R-x_1)). 
\end{align}
The function $(\rho^{app}_R, u^{app}_R)$ has a similar piecewise representation to \eqref{eq:piecewiserep}. Indeed, we have 
\begin{align}\label{eq:piecewiserepapprox}
&\rho^{app}_R, u^{app}_{R,1}, u^{app}_{R,2}(t,x_1)= \\
&\indent \begin{cases}
(\bar \rho, \bar u_1, \bar u_2)(t,x_1) & |x_1| \leq \frac{R}{4}, \\
(\bar \rho, -\bar u_1, \bar u_2)(t,R-x_1) & |x_1-R| \leq \frac{R}{4}, \\
\chi_R(x_1)(\bar \rho, \bar{u}_{1}, \bar{u}_{2})(t,x_1)+(1-\chi_R(x_1))(\tilde{\rho},\tilde{u}_{1}, \tilde{u}_2)(x_1) & x_1 \in [-\frac{R}{3}, -\frac{R}{4}] \cup [\frac{R}{4}, \frac{R}{3}], \\
\chi_R(R-x_1)(\bar{\rho}, -\bar{u}_{1}, \bar{u}_{2})(t,R-x_1)+(1-\chi_R(R-x_1))(\tilde{\rho},-\tilde{u}_{1}, \tilde{u}_2)(R-x_1) & x_1 \in [\frac{2R}{3}, \frac{3R}{R}] \cup [\frac{5R}{4}, \frac{4R}{3}], \\
(\tilde{\rho}_R, \tilde{u}_{R,1}, \tilde{u}_{R,2})(x_1) & \text{otherwise}.
\end{cases}
\end{align}
In other words, $(\rho^{app}_R, u^{app}_R)$ is either equal to the (reflected) full-line solution, the periodic reference function, or a convex combination of them. Due to this, and also \eqref{eq:fulllinebds}, it is clear that
\begin{align}\label{eq:approxfnisregular}
||\rho^{app}_R-\tilde{\rho}_R, u^{app}_R-\tilde{u}_R||_{C^([0,T];H^6(\T_R))} &\leq C\delta_3, \\
0 < \underline{\kappa} \leq \rho^{app}_R &\leq \overline{\kappa},
\end{align}
where we have changed the values of $\underline{\kappa}, \overline{\kappa}$ slightly if needed. 
Further, note that 
\begin{align}\label{eq:equalinitdata}
(\rho^{app}_R(0,\cdot), u^{app}_R(0,\cdot))=((\bar \rho_R)_0, (\bar u_R)_0).
\end{align}
\subsubsection{The error in the equation for the auxiliary function} Now, we would like to do an $H^3$-estimate on the difference between $(\rho^{app}_R, u^{app}_R)$ and $(\bar \rho_R, \bar u_R)$. The main difficulty is that $(\rho^{app}_R, u^{app}_R)$ is not a solution to \eqref{eq:2DNS_T_R}. In this section, we compute the error that arises from plugging $(\rho_R^{app}, u^{app}_R)$ into the equation \eqref{eq:2DNS_T_R}. For brevity, we drop the $x_2$ direction. Further, for the rest of this section, as $(\rho^{app}_R, u^{app}_R)$ and $(\bar \rho_R, \bar u_R)$ are both sufficiently regular and bounded away from vacuum, we may instead consider the equivalent system
\begin{align}\label{eq:1DNSappC}
\begin{cases}
\rho_t+u_1 \rho_x+\rho (u_1)_x=0, & (t,x) \in \R^+ \times \Pi, \\
(u_1)_t+u_1(u_1)_{x}+\gamma \rho^{\gamma-2}\rho_x-2\nu\ln(\rho)_x(u_1)_x=2\nu (u_1)_{xx}, \\
(u_2)_t+u_1(u_2)_x-\nu \ln(\rho)_x(u_2)_x=\nu (u_2)_{xx},
\end{cases}
\end{align}
where $\Pi=[-R, 2R]_{\text{periodic}}$. Note that $(\bar{\rho}_R, \bar{u}_R)$ verifies \eqref{eq:1DNSappC} exactly. The main lemma in this section will estimate the error obtained from plugging $(\rho^{app}_R, u^{app}_R)$ into \eqref{eq:1DNSappC}.
\begin{lemma}\label{lem:errorlemma}
Let $(\rho^{app}_R, u^{app}_R)$ be defined in \eqref{eq:appsolndef}. Then, $(\rho^{app}_R, u^{app}_R)$ verifies the equations
\begin{align}\label{eq:1DNSapproxerror}
\begin{cases}
\rho_t+u_1 \rho_x+\rho (u_1)_x=f_1(t,x), & (t,x) \in \R^+ \times [-R, 2R]_{\text{periodic}}, \\
(u_1)_t+u_1(u_1)_{x}+\gamma \rho^{\gamma-2}\rho_x-2\nu\ln(\rho)_x(u_1)_x-2\nu (u_1)_{xx}=f_2(t,x), \\
(u_2)_t+u_1(u_2)_x-\nu \ln(\rho)_x(u_2)_x-\nu (u_2)_{xx}=f_3(t,x),
\end{cases}
\end{align}
where $f_i$ obeys the bound for $1 \leq i \leq 3$,
\begin{align}\label{eq:errorlemmabd}
||f_i(t,x)||_{L^\infty([0,T];H^3(\Pi))} \leq C(s_R+R^{\frac{-1}{2}}),
\end{align}
where 
\begin{align}\label{eq:sobolevtailsappC}
s_R:=||\bar \rho-\tilde{\rho}, \bar u-\tilde{u}||_{C([0,T];H^5(\{|x_1 > \frac{R}{5}\}))} \to 0 \text{ as } R \to \infty. 
\end{align}
\end{lemma}
\begin{proof}
Let us show only the estimate on $f_2$, as the rest are somewhat simpler. Note that on the two regions $\{|x| \leq \frac{R}{4}\}$ and $\{|x-R| \leq \frac{R}{4}\}$, $(\rho^{app}_R, u^{app}_R)$ is an exact solution to \eqref{eq:1DNSappC}. So $f_2(t,x)=0$ there. Next, we consider the regions where $(\rho^{app}_R, u^{app}_R)=(\tilde{\rho}_R, \tilde{u}_R)$. Let us consider only $[\frac{R}{3}, \frac{2R}{3}]$ for simplicity (the other one is identical). Noting that $\rho^{app}_R$ and $u^{app}_{R,2}$ are constant on this region, we see
\begin{align*}
f_2(t,x)=\tilde{u}_{R,1}(\tilde{u}_{R,1})_x-2\nu(\tilde{u}_{R,1})_{xx}.
\end{align*}
From the definition \eqref{eq:reflectdef4}, we can directly compute that
\begin{align}\label{eq:f2computeregion2}
||\tilde{u}_{R,1}(\tilde{u}_{R,1})_x-2\nu(\tilde{u}_{R,1})_{xx}||_{H^3([\frac{R}{3}, \frac{2R}{3}])} \leq R^{\frac{-1}{2}}.
\end{align}
Finally, we consider the regions where $(\rho^{app}_R, u^{app}_R)$ is a convex combination of $(\bar \rho, \bar u)$ and $\tilde{\rho}, (\tilde{u})$. Again, let us consider only $[\frac{R}{4}, \frac{R}{3}]$ for simplicity. In this region, we have representations
\begin{align*}
\rho^{app}_R&=\rho_++\chi_R(\bar{\rho}-\rho_+), \\
u^{app}_R&=u_++\chi_R(\bar{u}-u_+).
\end{align*}
Plugging these formulas into the second equation in \eqref{eq:1DNSappC} gives 
\begin{align}\label{eq:u1erroreq}
f_2(t,x)=(u^{app}_{R,1})+&(u^{app}_{R,1})(u^{app}_{R,1})_x+\gamma(\rho^{app}_R)^{\gamma-2}(\rho^{app}_R)_x-2\ln(\rho^{app}_R)_x(u^{app}_{R,1})_x-2\nu(u^{app}_{R,1})_{xx} \\
&=\chi_R(u^{app}_{R,1}-\bar{u}_{1,+})(\bar u_1-u_+)_x+\gamma \chi_R((\rho^{app}_R)^{\gamma-2}-\bar{\rho}^{\gamma-2})\bar \rho_x+(\chi_R)_x\big(u^{app}_{R,1}(\bar{u}_1-u_{1,+}) \\
&+\gamma (\rho^{app}_R)^{\gamma-2}(\bar{\rho}-\rho_+)\big)-4\nu(\chi_R)_x(\bar u_1-u_{1,+})_x-2\nu (\chi_R)_{xx}(\bar{u}_1-u_{1,+}) \\
&+2\nu \chi_R(\ln(\bar \rho)_x-\ln(\rho^{app}_R)_x)(\bar u_1-u_{1,+})_x-2\nu(\chi_R)_x\ln(\rho^{app}_R)_x(\bar{u}_1-u_{1,+}).
\end{align}
Note that every single term on the right-hand side of \eqref{eq:u1erroreq} consists of either $(\bar{\rho}-\rho_+)$ or $(\bar{u}_1-u_{1,+})$ (or their derivatives), times a coefficient which is uniformly bounded in $W^{4,\infty}(\T_R)$ (due to \eqref{eq:fulllinebds} and \Cref{prop:1dexist+bds}). So, applying the Leibniz formula and using Young's inequality to bound the $H^3$-norm gives
\begin{align}\label{eq:f2computeregion3}
||f_2(t,x)||_{L^\infty([0,T];H^3([\frac{R}{4}, \frac{R}{3}]))} \leq Cs_R.
\end{align}
Combining \eqref{eq:f2computeregion2} and \eqref{eq:f2computeregion3} gives \eqref{eq:errorlemmabd} for $i=2$. The computations for $i=1,3$ are similar. 
\end{proof}

\subsubsection{The $H^3$-estimate between the approximate solution and the periodic solution}
Now, denote:
\begin{align*}
r=\rho^{app}_R-\bar \rho_R, \indent U=u^{app}_R-\bar u_R.
\end{align*}
Further note from \eqref{eq:equalinitdata} that 
\begin{align}\label{eq:perturb0initdata}
r(0,\cdot)=U(0,\cdot)=0.
\end{align}
Using the fact that $(\rho^{app}_R, u^{app}_R)$ solves \eqref{eq:1DNSapproxerror} and $(\bar \rho_R, \bar u_R)$ solves \eqref{eq:1DNSappC}, we derive that $(r,U)$ verifies the equations
\begin{align*}
r_t+\bar{u}_{R,1}r_x+\bar \rho_R(U_1)_x=-(\rho^{app}_R)_xU_1-(u^{app}_{R,1})_xr-f_1(t,x), \\
\end{align*}
\begin{align*}
(U_1)_t+(\bar{u}_{R,1}-2\nu\ln(\rho^{app}_R)_x)(U_1)_x+(\gamma \bar \rho^{\gamma-2}-2\nu \frac{(u^{app}_{R,1})_x}{\bar \rho_R})r_x-2\nu(U_1)_{xx}&=(u^{app}_{R,1})_xU_1-\gamma(\bar \rho^{\gamma-2}-(\rho^{app}_R)^{\gamma-2})(\rho^{app}_R)_x \\
&-2\nu \frac{(\rho^{app}_R)_x(u^{app}_{R,1})_x}{\bar \rho_R \rho^{app}_R}r-f_2(t,x), \\
\end{align*}
\begin{align}\label{eq:appCperturbationeqns}
(U_2)_t+(\bar{u}_{R,1}-\nu\ln(\rho^{app}_R)_x)(U_2)_x-\nu \frac{(u^{app}_{R,1})_x}{\bar \rho_R}r_x-\nu(U_2)_{xx}=-(u^{app}_{R,2})_x(U_1)-\nu \frac{(\rho^{app}_R)_x(u^{app}_{R,1})_x}{\bar \rho_R \rho^{app}_R}r-f_3(t,x).
\end{align}
As the energy estimate on this system is very similar to what we have already done in \Cref{sec:appendix1dexistence} (actually it is much easier, since we already have established that all of the coefficients are bounded in $W^{4,\infty}$), we will omit the details. In the end, we may establish
\begin{align*}
\frac{d}{dt}||r,U||_{H^3(\T_R)}^2 \leq C(||r,U||_{H^3(\T_R)}^2+\sum_{i=1}^3||f_i||_{H^3(\T_R)}).
\end{align*}
Using Gronwall's inequality, \eqref{eq:perturb0initdata}, and \Cref{lem:errorlemma}, we obtain
\begin{align}\label{eq:compareappandper}
||r,U||_{L^\infty([0,T];H^3(\T_R))}^2 \leq C(s_R+R^{\frac{-1}{2}})^2 \to 0 \text{ as } R \to \infty.
\end{align}

\subsubsection{Comparing the approximate solution and the full-line solution}
Define $v^{app}_R=u^{app}_R+\nu\ln(\rho^{app}_R)$ and $w^{app}_R=\nu\ln(\rho^{app}_R)$. Via \eqref{eq:compareappandper}, we have already that
\begin{align}\label{eq:compareappandper2}
&||(\bar{v}_R)_x-(v^{app}_R)_x||_{L^\infty([0,T] \times \T_R)}+||(\bar{w}_R)_x-(w^{app}_R)_x||_{L^\infty([0,T] \times \T_R)} \\
&+||\ln(\bar{\rho}_R)_{xx}-\ln(\rho^{app}_R)_{xx}||_{L^\infty([0,T] \times \T_R)}+||\ln(\bar{\rho}_R)_{x}-\ln(\rho^{app}_R)_{x}||_{L^\infty([0,T] \times \T_R)} \to 0 \text{ as }R \to \infty. 
\end{align}
It remains to compare the norms of the approximate solution to the norms of the full-line solution $(\bar \rho, \bar u)$. We consider each of the three regimes separately. Firstly, when $\{|x| \leq \frac{R}{4}\}$ and $\{|x-R| \leq \frac{R}{4}\}$, we have exactly
\begin{align*}
||(v^{app}_R)_x||_{\infty}=||(\bar{v})_x||_\infty, \\
||(w^{app}_R)_x||_{\infty}=||(\bar{w})_x||_\infty, \\
||\ln(\rho^{app}_R)_{xx}||_{\infty}=||\ln(\bar \rho)_{xx}||_\infty, \\
||\ln(\rho^{app}_R)_{x}||_{\infty}=||\ln(\bar \rho)_{x}||_\infty.
\end{align*}
On the regions where $(\rho^{app}_R, u^{app}_R)$ is equal to $(\tilde{\rho}_R, \tilde{u}_R)$, we see that
\begin{align*}
||(v^{app}_R)_x||_{\infty}=||(\tilde{u}_R)_x||_\infty \leq \frac{C}{R}, \\
||(w^{app}_R)_x||_{\infty}=0, \\
||\ln(\rho^{app}_R)_{xx}||_{\infty}=0, \\
||\ln(\rho^{app}_R)_{x}||_{\infty}=0.
\end{align*}
Finally, on the regions where $(\rho^{app}_R, u^{app}_R)$ is a convex combination of $(\bar{\rho}, \bar{u})$ and $(\tilde{\rho}, \tilde{u})$, the norms of derivatives of the approximate solution all converge to $0$ as $R \to \infty$. Indeed, on these regions all the above norms may be controlled by $s_R$ (recall \eqref{eq:sobolevtailsappC}). Combining the contributions from all three regions, we see
\begin{align}\label{eq:compareappandfull}
||(v^{app}_R)_x||_{L^\infty([0,T] \times \T_R)} \to ||\bar{v}_x||_{L^\infty([0,T] \times \R)}, \\
||(w^{app}_R)_x||_{L^\infty([0,T] \times \T_R)}\to ||\bar{w}_x||_{L^\infty([0,T] \times \R)} \\
||\ln(\rho^{app}_R)_{xx}||_{L^\infty([0,T] \times \T_R)}\to ||\ln(\bar \rho)_{xx}||_{L^\infty([0,T] \times \R)}, \\
||\ln(\rho^{app}_R)_{x}||_{L^\infty([0,T] \times \T_R)}\to ||\ln(\bar \rho)_{x}||_{L^\infty([0,T] \times \R)},
\end{align}
as $R \to \infty$. Combining \eqref{eq:compareappandper2} and \eqref{eq:compareappandfull} gives exactly \Cref{lem:appclem}.

\section{Optimal norm growth for the planar strong solution}\label{sec:optimalbds}
\renewcommand{\theequation}{D.\arabic{equation}}
\setcounter{equation}{0}
In this section, we compute how some norms of the planar strong solution on the channel scale with the viscosity. By taking the limit $R \to \infty$ in \Cref{prop:1dexist+bdsappendix}, we obtain already a poor (triple exponential) bound (see \Cref{rem:triplebad}). It is our objective to obtain better scaling with respect to the viscosity. For this reason, we will need to pass to the Lagrangian mass coordinates, which we are able to do on the full space. Note that the energy estimates are possible in the Lagrangian coordinates due to the absence of the triple product $\rho u^2$ in the Lagrangian coordinates. This term is especially harmful because we only have an estimate $||u||_{L^\infty} \leq \exp(\frac{C}{\nu^j})$ from \Cref{sec:appendix1dexistence}, compared to the bound $||\rho||_{L^\infty} \leq \frac{C}{\nu^j}$. We work with the strong solution to the equation (where we have dropped the bars for convenience)
\begin{equation}\label{eq:1DNS}
\left\{
\begin{array}{ll}
    \rho_t +(\rho u_1)_{x}=0,  \indent (t,x) \in \R^+ \times \R, \\[2mm]
     (\rho u_1)_t+(\rho u_1^2)_{x}+p(\rho)_{x}=(2 \nu \rho u_{1x})_{x}, \\[2mm]
     (\rho u_2)_t+(\rho u_1u_2)_{x}=(\nu \rho u_{2x})_{x},\\[2mm]
     (\rho(0,x), (\rho u)(0,x))=(\rho^0, m^0)(x),
\end{array}
\right.
\end{equation}
where $u(t,x)=\begin{bmatrix}
u_1(t,x) \\
u_2(t,x) 
\end{bmatrix}$
The initial value $(\rho^0, u^0)$ verifies:
\begin{align}\label{eq:1Dchannelbds}
&||(\rho^0, u^0)-(\tilde{\rho}, \tilde{u})||_{H^6(\R)} \leq \frac{1}{\nu}, \\
&|(u_{1}^0)_{x_1}| \leq \frac{1}{2\nu}(\rho^0)^{\gamma-1}, \\
&\rho^0 \geq c_0 > \frac{1}{\nu} > 0. 
\end{align}
For brevity, in the section we will use the following notation.
\begin{notation}
\begin{alignat*}{6}
    &U^0=(\rho^0, u^0), \\
    &U=(\rho, u), \\
    &\tilde{S}=(\tilde{\rho}, \tilde{u}), \\
    &L^p_tL^q_y:=L^p([0,T];L^q(\R)) \text{ for $1 \leq p,q \leq \infty$,} \\
    &L^p_tH^q_y:=L^p([0,T];H^q(\R)) \text{ for $1 \leq p,q \leq \infty$.}
\end{alignat*}
We let $j$ denote an arbitrary positive integer that may vary line by line and depends only on $\gamma$. Finally, we let $C$ denote an arbitrary constant that may vary line by line. The constant $C$ may depend on $||(\rho^0, u^0)-(\tilde{\rho}, \tilde{u})||_{H^6(\R)}, c_0, ||(\tilde{\rho}, \tilde{u})||_{W^{1,\infty}(\R)}, T, \& \ \delta_3$. Note that in this section $C$ will not depend on $\nu$, as we are careful to track the norm dependence on $\nu$.
\end{notation}
Our main proposition is the following. Note that it is exactly \Cref{prop:1doptimalbdsforuse} in the new notation.
\begin{proposition}\label{prop:1optimalbdsappendix}
Let $(\rho^0, u^0)$ verify the bounds \eqref{eq:1Dchannelbds}. Let $T>0$ be fixed. Consider the strong solution $(\rho, u)$ to $\eqref{eq:1DNS}$ with initial data $(\rho^0, u^0)$, where $u=\begin{bmatrix}
u_1 \\
u_2
\end{bmatrix}$. Then, there exists a constant $C(T,\delta_3, c_0) > 0$ and an integer $j \in \N$ such that
\begin{align}\label{eq:optimal1dbds}
&\inf_{(t,x) \in [0,T] \times \R}\rho^{-1}(t,x)+||\rho-\tilde{\rho}||_{L^\infty([0,T];H^3(\R))} \notag \\ 
&+||(u_1, u_2)-(\tilde{u}_1, \tilde{u}_2)||_{L^\infty([0,T];H^2(\R))} \leq \frac{C}{\nu^j}.
\end{align}
\end{proposition}

\subsubsection{Lagrangian mass coordinates}
Our strategy to prove \Cref{prop:1optimalbdsappendix} is to do energy estimates in the Lagrangian mass coordinates. Making the change of variables $dy=\rho dx-\rho  u_1dt$, we obtain the system:
\begin{equation}\label{eq:1DNSlagrangian}
\left\{
\begin{array}{ll}
v_t-(u_1)_y=0, \\
(u_1)_t+p^{L}(v)_y=2\nu\big(\frac{(u_1)_y}{v^2}\big)_y, \\
(u_2)_t=\nu\big(\frac{(u_2)_y}{v^2}\big)_y, \\
\end{array}
\right.
\end{equation}
where $p^L(v)=v^{-\gamma}$. Now, let $\tilde{v}(y)=\frac{1}{\tilde{\rho}(x)}$. Recalling that 
\begin{align}\label{eq:vupperlowerbd}
\underline{\kappa} \leq v(t,y) \leq \overline{\kappa} \text{ for } \frac{1}{\underline{\kappa}} \leq \frac{C}{\nu^j}, \overline{\kappa} \geq C\nu^j,
\end{align} 
(see \eqref{eq:densitylowerbd} and \Cref{rem:densityupperbdscaling}), it is straightforward to obtain
\begin{align}\label{eq:lagrangianinitdatabd}
||(v^0, u^0)-(\tilde{v}, \tilde{u})||_{H^4(\R)} \leq \frac{C}{\nu^j},
\end{align}
where the reference functions are now
\begin{align*}
\tilde{v}(y)=\frac{1}{\tilde{\rho}(x(0,y))}, \\
\tilde{u}(y)=\tilde{u}(x(0,y)).
\end{align*}
\Cref{prop:1optimalbdsappendix} follows directly from the following proposition in the Lagrangian mass coordinates:
\begin{proposition}\label{prop:1doptimalbdsappendixlagrangian}
Let $(\rho^0, u^0)$ verify the bounds \eqref{eq:1Dchannelbds}. Let $T > 0$ be fixed. Let $(v, u)$ be the associated solution to \eqref{eq:1DNSlagrangian} with initial data $(v^0, u^0)$. Then, there exists a constant $C(T,\delta_3, c_0) > 0$ and an integer $j \in \N$ so that 
\begin{align}\label{eq:1dbdslagrangian}
&||v-\tilde{v}||_{L^\infty([0,T];H^3(\R))}+||(u_1, u_2)-(\tilde{u}_1, \tilde{u}_2)||_{L^\infty([0,T];H^2(\R))} \leq \frac{C}{\nu^j}.
\end{align}
\end{proposition}
\subsubsection{$L^2$-estimate}
We start by obtaining the analog of the BD-entropy inequality in Lagrangian mass coordinates. Using $(\tilde{u}_1)_t=0$ and combining the first and second equation in \eqref{eq:1DNSlagrangian} gives
\begin{align}\label{eq:combine12}
((u_1-\tilde{u}_1)+2\nu \big(\frac{1}{v}\big)_y)_t+p^L(v)_y=0.
\end{align}
Multiplying by $(u_1-\tilde{u}_1)+2\nu\big(\frac{1}{v}\big)_y$ gives
\begin{align*}
\Big(\frac{((u_1-\tilde{u}_1)+2\nu \big(\frac{1}{v}\big)_y)^2}{2}\Big)_t+p^L(v)_y(u_1-\tilde{u}_1)+2\nu \gamma v^{-\gamma-3}v_y^2=0.
\end{align*}
Now, adding and subtracting $(\tilde{u}_1)_y$ to the first equation in \eqref{eq:1DNSlagrangian} and multiplying by $-(p^L(v)-p^L(\tilde{v}))$ gives
\begin{align}\label{eq:antipressure}
\big(-\int_{\tilde{v}}^vp^L(s)ds+p^L(\tilde{v})(v-\tilde{v})\big)_t&+(p^L(v)-p^L(\tilde{v}))(u_1-\tilde{u}_1)_y \notag \\
&+(\tilde{u}_1)_y(p^L(v)-p^L(\tilde{v})=0.
\end{align}
Adding the previous two equations and integrating in space gives
\begin{align}\label{eq:0orderBDlag}
&\partial_t \int_{\R}\Big(\frac{((u_1-\tilde{u}_1)+2\nu \big(\frac{1}{v}\big)_y)^2}{2}+\big(-\int_{\tilde{v}}^vp^L(s)ds+p^L(\tilde{v})(v-\tilde{v})\big) \Big)dy \notag \\
&+2\nu \gamma \int_{\R} v^{-\gamma-3}v_y^2dy+\int_{\R} p^L(\tilde{v})_y(u_1-\tilde{u}_1)dy+\int_{\R}(\tilde{u}_1)_y(p^L(v)-p^L(\tilde{v}))dy=0.
\end{align}
Now, adding and subtracting  $2\nu \big(\frac{(\tilde{u}_1)_y}{v^2}\big)_y$ to the second equation in \eqref{eq:1DNSlagrangian}, multiplying by $(u_1-\tilde{u_1})$, and integrating in space:
\begin{align*}
&\partial_t \int_{\R} \frac{(u_1-\tilde{u}_1)^2}{2}dy+\int_{\R}p^L(v)_y(u_1-\tilde{u_1})dy \\
&+2\nu \int_{\R}\frac{(u_1-\tilde{u}_1)_y}{v^2}dy+2\nu \int_{\R}\frac{(\tilde{u}_1)_y}{v^2}(u_1-\tilde{u}_1)_ydy=0.
\end{align*}
Adding \eqref{eq:antipressure} to the previous equation, we obtain
\begin{align}\label{eq:0orderu1lag}
&\partial_t \Big(\int_{\R} \frac{(u_1-\tilde{u}_1)^2}{2}+\big(-\int_{\tilde{v}}^vp^L(s)ds+p^L(\tilde{v})(v-\tilde{v})\big)\Big) dy+2\nu\int_{\R}\frac{(u_1-\tilde{u}_1)_y}{v^2}dy \\
&+2\nu \int_{\R}\frac{(\tilde{u}_1)_y}{v^2}(u_1-\tilde{u}_1)_ydy+\int_{\R} p^L(\tilde{v})_y(u_1-\tilde{u}_1)dy\\
&+\int_{\R}(\tilde{u}_1)_y(p^L(v)-p^L(\tilde{v}))dy=0.
\end{align}
Adding and subtracting $\nu \big(\frac{(\tilde{u}_2)_y}{v^2}\big)_y$ to the third equation in \eqref{eq:1DNSlagrangian}, multiplying by $(u_2-\tilde{u_2})$, and integrating in space:
\begin{align}\label{eq:0orderu2lag}
&\partial_t \int_{\R} \frac{(u_2-\tilde{u}_2)^2}{2}dy+\nu \int_{\R}\frac{(u_2-\tilde{u}_2)_y}{v^2}dy+\nu \int_{\R}\frac{(\tilde{u}_2)_y}{v^2}(u_2-\tilde{u}_2)_ydy=0.
\end{align}
Finally, adding \eqref{eq:0orderBDlag}, \eqref{eq:0orderu1lag}, and \eqref{eq:0orderu2lag} gives
\begin{align*}
&\partial_t \int_{\R} \Big(\frac{((u_1-\tilde{u}_1)+2\nu \big(\frac{1}{v}\big)_y)^2}{2}+2\big(-\int_{\tilde{v}}^vp^L(s)ds+p^L(\tilde{v})(v-\tilde{v})\big) \\
&+\frac{(u_1-\tilde{u}_1)^2}{2}+\frac{(u_2-\tilde{u}_2)^2}{2}\Big)dy+2\nu\int_{\R}\frac{(u_1-\tilde{u}_1)_y}{v^2}dy \\
&+\nu \int_{\R}\frac{(u_2-\tilde{u}_2)_y}{v^2}dy+2\nu \gamma \int_{\R} v^{-\gamma-3}v_y^2dy \\
&=-2\int_{\R}(\tilde{u}_1)_y(p^L(v)-p^L(\tilde{v}))dy-2\int_{\R} p^L(\tilde{v})_y(u_1-\tilde{u}_1)dy \\
&-2\nu \int_{\R}\frac{(\tilde{u}_1)_y}{v^2}(u_1-\tilde{u}_1)_ydy-\nu \int_{\R}\frac{(\tilde{u}_2)_y}{v^2}(u_2-\tilde{u}_2)_ydy \\
&=:\sum_{k=1}^4I_k. 
\end{align*}
We estimate each term as follows:
\begin{align*}
I_1 \leq (||p^L(v)||_{L^\infty([0,T] \times \R)}+||p^L(\tilde{v})||_{L^\infty([0,T] \times \R)})\int_{\R}|(\tilde{u}_1)_y|dy \leq \frac{C}{\nu^j}.
\end{align*}
For the second, using Young's inequality
\begin{align*}
I_2 \leq C+\int_{\R}(u_1-\tilde{u_1})^2dy.
\end{align*}
For the third, using Young's inequality and \eqref{eq:vupperlowerbd}
\begin{align*}
I_3 \leq \frac{C}{\nu^j}+\frac{\nu}{2}\int_{\R}\frac{(u_1-\tilde{u}_1)_y}{v^2}dy.
\end{align*}
The last term is handled identically:
\begin{align*}
I_4 \leq \frac{C}{\nu^j}+\frac{\nu}{2}\int_{\R}\frac{(u_2-\tilde{u}_2)_y}{v^2}dy.
\end{align*}
Finally, Gronwall's inequality gives the bound:
\begin{align}\label{eq:0orderestimatelag}
||u_1-\tilde{u}_1||_{L^\infty_tL^2_y}+||u_2-\tilde{u}_2||_{L^\infty_tL^2_y}+||(\frac{1}{v})_y||_{L^\infty_tL^2_y}+||(v_y, (u_1)_y, (u_2)_y)||_{L^2_tL^2_y} \leq \frac{C}{\nu^j}. 
\end{align}
Note that by \eqref{eq:vupperlowerbd} we also have a bound $||v_y||_{L^\infty_tL^2_y} \leq \frac{C}{\nu^j}$.
\subsubsection{$H^1$-estimate}
Now, we obtain the $H^1$-estimate. We start by differentiating \eqref{eq:combine12}, multiplying by $(u_1+2\nu \big(\frac{1}{v}\big)_y)_{y}$, and integrating in space to obtain
\begin{align}\label{eq:1orderBDlag}
&\partial_t \int_{\R}\Big(\frac{((u_1)+2\nu \big(\frac{1}{v}\big)_y)_y^2}{2}\Big)dy+2\nu \gamma \int_{\R}v^{-\gamma-3}v_{yy}^2dy+4 \nu \gamma (\gamma+1)\int_{\R}v^{-\gamma-5}v_y^4dy \\
&+\int_{\R}(-2\nu\gamma(\gamma+1)-4\nu \gamma)v^{-\gamma-4}v_y^2v_{yy}dy+\int_{\R}p^L(v)_{yy}(u_1)_ydy=0.
\end{align}
Now, differentiating the second equation in \eqref{eq:1DNSlagrangian} in $y$, multiplying by $(u_1)_y$, and integrating:
\begin{align}\label{eq:1orderu1lag}
\partial_t \int_{\R}\frac{(u_1)_y^2}{2}dy+\int_{\R}p^L(v)_{yy}(u_1)_ydy+2\nu \int_{\R}\frac{(u_1)_{yy}^2}{v^2}dy=4\nu \int_{\R}\frac{(u_1)_yv_y(u_1)_{yy}}{v^3}dy.
\end{align}
Doing the same for the third equation gives
\begin{align}\label{eq:1orderu2lag}
&\partial_t \int_{\R} \frac{(u_2)_y^2}{2}dy+\nu \int_{\R}\frac{(u_2)_{yy}}{v^2}dy=2\nu \int_{\R}\frac{(u_2)_yv_y(u_2)_{yy}}{v^3}dy.
\end{align}
Finally, summing \eqref{eq:1orderBDlag}, \eqref{eq:1orderu1lag}, and \eqref{eq:1orderu2lag} gives
\begin{align*}
&\partial_t\Big(\int_{\R}\Big(\frac{((u_1)+2\nu \big(\frac{1}{v}\big)_y)_y^2}{2}\Big)dy+\int_{\R}\frac{(u_1)_y^2}{2}dy+\int_{\R} \frac{(u_2)_y^2}{2}dy\Big)+2\nu \gamma \int_{\R}v^{-\gamma-3}v_{yy}^2dy \\
&+4 \nu \gamma (\gamma+1)\int_{\R}v^{-\gamma-5}v_y^4dy+2\nu \int_{\R}\frac{(u_1)_{yy}^2}{v^2}dy+\nu \int_{\R}\frac{(u_2)_{yy}}{v^2}dy \\
&=-2\int_{\R}p^L(v)_{yy}(u_1)_ydy+\int_{\R}(2\nu\gamma(\gamma+1)4\nu \gamma)v^{-\gamma-4}v_y^2v_{yy}dy \\
&+4\nu \int_{\R}\frac{(u_1)_yv_y(u_1)_{yy}}{v^3}dy+2\nu \int_{\R}\frac{(u_2)_yv_y(u_2)_{yy}}{v^3}dy=:\sum_{k=1}^4J_k.
\end{align*}
We estimate each term as follows:
\begin{align*}
J_1 &\leq \frac{C}{\nu^j}\big(\int_{\R}v_y^2(u_1)_ydy+\int_{\R}v_{yy}(u_1)_ydy\big) \\
&\leq \frac{C}{\nu^j}\int_{\R}(u_1)_y^2dy +2 \nu \gamma (\gamma+1)\int_{\R}v^{-\gamma-5}v_y^4dy+\nu\gamma \int_{\R}v^{-\gamma-3}v_{yy}^2dy+\frac{C}{\nu^j} \int_{\R}(u_1)_y^2dy,
\end{align*}
using \eqref{eq:vupperlowerbd}. For the second, using Gagliardo-Nirenberg interpolation inequality
\begin{align*}
J_2 \leq \frac{C}{\nu^j}||v_y||_\infty ||v_y||_2||v_{yy}||_2 \leq \frac{C}{\nu^j}||v_y||_2^\frac{3}{2}||v_{yy}||_2^\frac{3}{2} \leq \frac{C}{\nu^j}||v_y||_2^6+\frac{\nu \gamma}{2}\overline{\kappa}^{-\gamma-3}||v_{yy}||_2^2.
\end{align*}
For the third:
\begin{align*}
J_3 \leq \frac{C}{\nu^j}||(u_1)_{yy}||_2||(u_1)_y||_\infty||v_y||_2 &\leq \frac{C}{\nu^j}||(u_1)_y||_2^\frac{1}{2}||(u_1)_{yy}||_2^\frac{3}{2}||v_y||_2 \\
&\leq \nu \overline{\kappa}^{2}||(u_1)_{yy}||_2^2+\frac{C}{\nu^j}||(u_1)_y||_2^2||v_y||_2^4.
\end{align*}
The term $J_4$ may be handled identically. Finally, integrating in time and using \eqref{eq:0orderestimatelag} gives
\begin{align}\label{eq:1orderestimatelag}
&||u_1-\tilde{u}_1||_{L^\infty_tH^1_y}+||u_2-\tilde{u}_2||_{L^\infty_tH^1_y}+||(\frac{1}{v})_{yy}||_{L^\infty_tL^2_y} \notag \\
&+||v_{y}||_{L^4_tL^4_y}+||(v_{yy}, (u_1)_{yy}, (u_2)_{yy})||_{L^2_tL^2_y} \leq \frac{C}{\nu^j}.
\end{align}
Note that this implies $||v_y||_{L^\infty([0,T] \times \R)}+||v_{yy}||_{L^\infty_tL^2_y} \leq \frac{C}{\nu^j}$.
\subsubsection{$H^2$-estimate}
Finally, we obtain the $H^2$-estimate. Differentiating \eqref{eq:combine12} twice, multiplying by $(u_1+2\nu\big(\frac{1}{v})_y)_{yy}$, and integrating in space gives
\begin{align}\label{eq:2orderBDlag}
&\partial_t \int_{\R}\Big(\frac{((u_1)+2\nu \big(\frac{1}{v}\big)_y)_{yy}^2}{2}\Big)dy+\int_{\R}p^L(v)_{yyy}(u_1)_{yy}dy \notag \\
&+\phi_1(\nu)\int_{\R}(v_{yyy}^2+v_y^2v_{yy}^2+v_y^6)dy \leq \phi_2(\nu) \int_{\R}(v_y^3v_{yyy}+v_y^4v_{yy}+v_{yyy}v_yv_{yy})dy,
\end{align}
for some functions $\phi_1, \phi_2$ verifying
\begin{align*}
C\nu^j &\leq \phi_1(\nu), \\
\phi_2(\nu)&\leq \frac{C}{\nu^j},
\end{align*}
where we have used \eqref{eq:vupperlowerbd}. Next, differentiating the second equation in \eqref{eq:1DNSlagrangian} twice, multiplying by $(u_1)_{yy}$, and integrating:
\begin{align}\label{eq:2orderu1lag}
&\partial_t \int_{\R}\frac{(u_1)_{yy}^2}{2}dy+2\nu \int_{\R}\frac{(u_1)_{yyy}^2}{v^2}dy=-\int_{\R}p^L(v)_{yyy}(u_1)_{yy}dy \\
&+8\nu \int_{\R}\frac{(u_1)_{yy}v_y(u_1)_{yyy}}{v^3}dy+4\nu \int_{\R}\frac{(u_1)_yv_{yy}(u_1)_{yyy}}{v^3}dy \\
&-12\nu \int_{\R}\frac{(u_1)_yv_y^2(u_1)_{yyy}}{v^4}dy.
\end{align}
Doing the same for the third equation in \eqref{eq:1DNSlagrangian} gives
\begin{align}\label{eq:2orderu2lag} 
&\partial_t \int_{\R}\frac{(u_2)_{yy}^2}{2}dy+\nu \int_{\R}\frac{(u_2)_{yyy}^2}{v^2}dy=4\nu \int_{\R}\frac{(u_2)_{yy}v_y(u_2)_{yyy}}{v^3}dy\\
&+2\nu \int_{\R}\frac{(u_2)_yv_{yy}(u_2)_{yyy}}{v^3}dy -6\nu \int_{\R}\frac{(u_2)_yv_y^2(u_2)_{yyy}}{v^4}dy.
\end{align}
Combining \eqref{eq:2orderBDlag}, \eqref{eq:2orderu1lag}, and \eqref{eq:2orderu2lag} gives
\begin{align*}
&\partial_t \Big(\int_{\R}\Big(\frac{((u_1)+2\nu \big(\frac{1}{v}\big)_y)_{yy}^2}{2}\Big)dy+\int_{\R}\frac{(u_1)_{yy}^2}{2}dy+\int_{\R}\frac{(u_2)_{yy}^2}{2}dy\Big) \\
&+\phi_1(\nu)\int_{\R}(v_{yyy}^2+v_y^2v_{yy}^2+v_y^6)dy+2\nu \int_{\R}\frac{(u_1)_{yyy}^2}{v^2}dy+\nu \int_{\R}\frac{(u_2)_{yyy}^2}{v^2}dy \\
& \leq-2\int_{\R}p^L(v)_{yyy}(u_1)_{yy}dy+\phi_2(\nu) \int_{\R}(v_y^3v_{yyy}+v_y^4v_{yy}+v_{yyy}v_yv_{yy})dy \\
&+8\nu \int_{\R}\frac{(u_1)_{yy}v_y(u_1)_{yyy}}{v^3}dy+4\nu \int_{\R}\frac{(u_1)_yv_{yy}(u_1)_{yyy}}{v^3}dy \\
&-12\nu \int_{\R}\frac{(u_1)_yv_y^2(u_1)_{yyy}}{v^4}dy+4\nu \int_{\R}\frac{(u_2)_{yy}v_y(u_2)_{yyy}}{v^3}dy \\
&+2\nu \int_{\R}\frac{(u_2)_yv_{yy}(u_2)_{yyy}}{v^3}dy-6\nu \int_{\R}\frac{(u_2)_yv_y^2(u_2)_{yyy}}{v^4}dy \\
&=:\sum_{k=1}^8 I_k.
\end{align*}
We control the terms as follows:
\begin{align*}
I_1\leq \frac{C}{\nu^j}\int_{\R}(v_y^3+v_yv_{yy}+v_{yyy})(u_1)_{yy}dy &\leq \frac{\phi_1(\nu)}{8}\int_{\R}(v_y^6+v_y^2v_{yy}^2+v_{yyy}^2)dy \\
&+\frac{C}{\nu^j}\int_{\R}(u_1)_{yy}^2dy.
\end{align*}
For the second, we use Young's inequality
\begin{align*}
I_2 \leq \frac{\phi_1(\nu)}{8}\int_{\R}(v_{yyy}^2+v_y^2v_{yy}^2+v_y^6)dy+\frac{C}{\nu^j}(||v_y||_\infty^4||v_y||_2^2+||v_y||_\infty^2||v_{yy}||_2^2).
\end{align*}
For the third,
\begin{align*}
I_3 &\leq \frac{C}{\nu^j}||(u_1)_{yy}||_2||v_y||_\infty||(u_1)_{yyy}||_2 \\
&\leq \frac{\nu}{2} \overline{\kappa}^2 ||(u_1)_{yyy}||_2^2+\frac{C}{\nu^j}||(u_1)_{yy}||_2^2||v_y||_\infty^2.
\end{align*}
Next, we have
\begin{align*}
I_4 &\leq \frac{C}{\nu^j}||(u_1)_{yyy}||_2||v_{yy}||_\infty||(u_1)_y||_2 \leq \frac{C}{\nu^j}||(u_1)_{yyy}||_2||v_{yy}||_2^\frac{1}{2}||v_{yyy}||_2^\frac{1}{2}||(u_1)_y||_2 \\
&\leq \frac{\nu}{2} \overline{\kappa}^2 ||(u_1)_{yyy}||_2^2+\frac{\phi_1(\nu)}{8}||v_{yyy}||_2^2+\frac{C}{\nu^j}||v_{yy}||_2^2||(u_1)_y||_2^4.
\end{align*}
Further
\begin{align*}
I_5 \leq \frac{C}{\nu^j}||v_y||_\infty^2||(u_1)_y||_2||(u_1)_{yyy}||_2 \leq  \frac{\nu}{2} \overline{\kappa}^2 ||(u_1)_{yyy}||_2^2+\frac{C}{\nu^j}||v_y||_\infty^4||(u_1)_y||_2^2.
\end{align*}
The terms $I_6$ to $I_8$ may be handled identically. So, integrating in time and using the bounds \eqref{eq:1orderestimatelag} gives
\begin{align}\label{eq:2orderestimatelag}
&||u_1-\tilde{u}_1||_{L^\infty_tH^2_y}+||u_2-\tilde{u}_2||_{L^\infty_tH^2_y}+||(\frac{1}{v})_{yyy}||_{L^\infty_tL^2_y} \leq \frac{C}{\nu^j}.
\end{align}
Finally, using the formula $(\frac{1}{v})_{yyy}=-\frac{v_{yyy}}{v^2}+\frac{6v_{yy}v_y}{v^3}-\frac{6v_y^3}{v^4}$ and the bounds $\eqref{eq:0orderestimatelag}, \eqref{eq:1orderestimatelag}$, we obtain exactly \Cref{prop:1doptimalbdsappendixlagrangian}. Making the change of variables back to the Eulerian coordinates gives \Cref{prop:1optimalbdsappendix}.

\bibliographystyle{plain}
\bibliography{refs}
\end{document}